\documentclass[a4paper]{amsart}

\usepackage{color}
\usepackage[
frame,cmtip,arrow,matrix,line,graph]{xy}
\usepackage{comment}
\definecolor{darkgreen}{RGB}{47,139,79}
\definecolor{darkblue}{RGB}{36,24,130}

\usepackage{hyperref}
\hypersetup{
    colorlinks=true, 
    linktoc=all,     
    linktocpage,
    citecolor=darkgreen,
    filecolor=black,
    linkcolor=darkblue,
    urlcolor=black
}

\usepackage{amsmath}
\usepackage{amsfonts}
\usepackage{amssymb}
\usepackage{graphicx}%
\usepackage{xfrac}

\let\oldtocsection=\tocsection

\let\oldtocsubsection=\tocsubsection

\let\oldtocsubsubsection=\tocsubsubsection

\renewcommand{\tocsection}[2]{\hspace{0em}\oldtocsection{#1}{#2}}
\renewcommand{\tocsubsection}[2]{\hspace{1em}\oldtocsubsection{#1}{#2}}
\renewcommand{\tocsubsubsection}[2]{\hspace{2em}\oldtocsubsubsection{#1}{#2}}

\DeclareRobustCommand{\SkipTocEntry}[5]{}

\newtheorem{thm}{Theorem}[section]
\newtheorem{lem}[thm]{Lemma}
\newtheorem{prop}[thm]{Proposition}
\newtheorem{cor}[thm]{Corollary}

\newtheorem{Th}{Theorem}

\theoremstyle{definition}
\newtheorem{Def}[thm]{Definition}

\newtheorem{ex}[thm]{Example}
\newtheorem{non-ex}[thm]{Non-example}
\newtheorem{assum}{Assumption}
\newtheorem{convention}{Convention}
\theoremstyle{remark}
\newtheorem{rem}[thm]{Remark}

\numberwithin{equation}{section}

\usepackage[frame,cmtip,arrow,matrix,line,graph]{xy}
\input xy
\xyoption{all}
\usepackage{color}
\usepackage{lpic,enumerate}
\usepackage{stmaryrd}

\newcommand{\al}{\alpha}
\newcommand{\C}{\mathcal{C}}
\newcommand{\del}{\partial}

\newcommand{\De}{\Delta}
\newcommand{\D}{\mathcal{D}}
\newcommand{\diag}{\operatorname{di\overline{ag}}}
\newcommand{\e}{\mathbf{ev}}
\newcommand{\E}{\mathcal{E}}

\newcommand{\F}{\mathrm{F}}
\newcommand{\bF}{\overline{F}}
\newcommand{\ga}{\gamma}
\newcommand{\la}{\lambda}

\newcommand{\La}{\Lambda}

\newcommand{\eps}{\varepsilon}

\newcommand{\RR}{\mathbb{R}}
\newcommand{\R}{\mathcal{R}}

\newcommand{\T}{\mathcal{T}}
\newcommand{\Z}{\mathbb{Z}}
\newcommand{\M}{\mathcal{M}}

\newcommand{\ha}{\frac{1}{2}}

\newcommand{\taue}{\overline\tau_e}
\newcommand{\etab}{\overline\eta}

\newcommand{\maps}{\operatorname{Maps}}
\newcommand{\id}{\operatorname{id}}

\newcommand{\inc}{\hookrightarrow}
\newcommand{\surj}{\twoheadrightarrow}
\newcommand{\lar}{\longleftarrow}
\newcommand{\rar}{\longrightarrow}
\newcommand{\sta}{\stackrel}
\newcommand{\arsim}{\sta{\simeq}{\rar}}

\newcommand{\minus}{\backslash}
\newcommand{\x}{\times}
\newcommand{\ot}{\otimes}
\newcommand{\op}{\oplus}

\newcommand{\reread}{\operatorname{reread}}

\newcommand{\note}[1]
{{\bf [#1]}}

\title{On the invariance of the string topology coproduct}
\author{Nathalie Wahl}
\date{\today}

\begin{document}

\maketitle

\begin{abstract}
 We give a variant of Naef's formula for the failure of invariance of the string topology coproduct under homotopy equivalences, using an obstruction class built from homotopy data associated to a homotopy equivalence as well as the ``fake diagonal''.  The vanishing of our obstruction class can be seen as a way to measure a form of boundedness for homotopy equivalences. We show that the same obstruction rules the failure of invariance for a generalization of the coproduct to higher dimensional loops. 
\end{abstract}

\section*{Introduction}

For a closed oriented manifold $M$, 
the Chas-Sullivan product \cite{CS99} is a  product on the homology of its free loop space $\La M=\operatorname{Maps}(S^1,M)$ of the form
$$\wedge: H_p(\La M)\ot  H_q(\La M)\rar H_{p+q-n}(\La M),$$
where $H_*(-)=H_*(-;\Z)$ denotes singular homology with integral coefficients and $n=\dim M$. 
 Goresky and Hingston defined in \cite{GorHin} a ``dual'' product in cohomology relative to the constant loops
$$\oast: H^p(\La M,M)\ot H^q(\La M,M) \to H^{p+q+n-1}(\La M,M).$$
The associated homology coproduct 
$$\vee: H_p(\La M,M)\rar H_{p-n+1}(\La M\x \La M,M\x \La M\cup \La M\x M) $$
was defined by Sullivan \cite{Sul04}. 
(As we work with homology with integral coefficients, there is no K{\"u}nneth isomorphism and hence the coproduct in integral homology has target $H_*(\La M\x \La M, M\x \La M\cup \La M\x M)$ rather than $H_*(\La M,M)\ot H_*(\La M,M)$.) 
The idea of the Chas-Sullivan product is to intersect the chains of basepoints of two chains of loops and concatenate the loops at the common basepoints. The homology coproduct looks instead for self-intersections  at the basepoint within a single chain of loops and cuts. 
Both structures are represented geometrically by the same figure, 
\vspace{-5mm}
\begin{figure}[h]
\includegraphics[width=6.5cm]{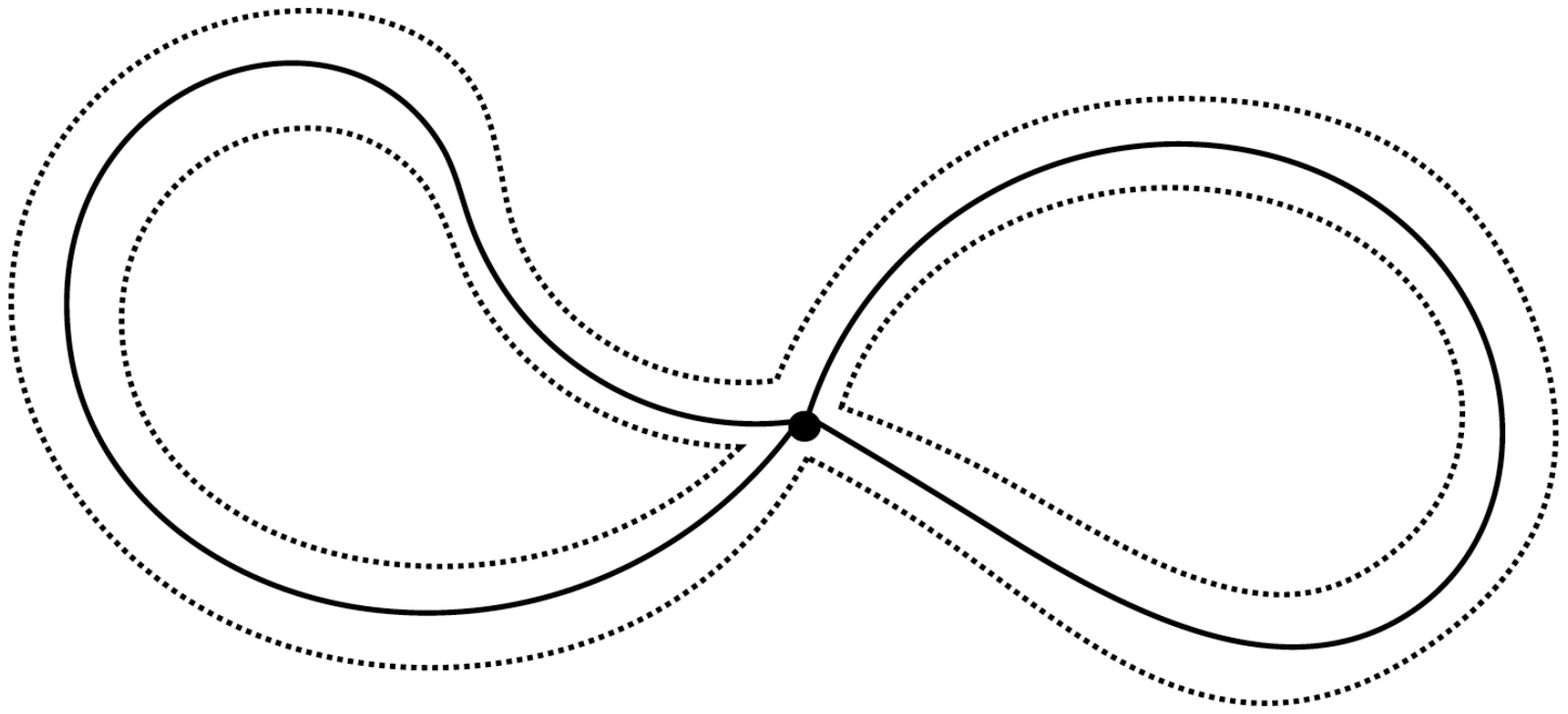}
\end{figure}
\vspace{-7mm}
either thought of  as two loops intersecting at their basepoint, or as a single loop having a self-intersection at its basepoint. 
For the coproduct, it is essential to look for self-intersections with the basepoint at all times $t$ along the loop, including $t=0$ and $t=1$ (where all loops tautologically have a self-intersection!)---Tamanoi showed that the coproduct that only looks for self-intersections at $t=\frac{1}{2}$ is almost completely trivial \cite{Tam}. In contrast, $H^*(\La S^n,S^n)$ for an odd sphere $S^n$ is generated by just 4 classes as an algebra with the cohomology product $\oast$ \cite[Thm 15.3]{GorHin}. The domain $[0,1]$ of the time variable $t$ is the reason why the coproduct 
naturally is an operation in relative homology: it has non-trivial boundary terms coming from the boundary of the interval. 
This operation can be associated to a 1-parameter family of surfaces in the harmonic compactification of the moduli space of Riemann surfaces, making it a so-called {\em compactified operation}. 

Iterations of this coproduct detects {\em intersection multiplicity} of homology classes in the loop space, 
the largest number $k$ such that any representing chain of a given homology class in $H_*(\La M)$ necessarily has a loop with a $k$-fold self-intersection at its basepoint.
(See \cite[Thm 3.10]{HinWah0}.)

\medskip

It has been shown by several authors that the Chas-Sullivan product is homotopy invariant, in the sense 
 that a homotopy equivalence  $f:M_1\to M_2$ induces an isomorphism 
$\La f:H_*(\La M_1)\to H_*(\La M_2)$ of algebras with respect to  the Chas-Sullivan product (see \cite[Thm 1]{CKS},\cite[Prop 23]{GruSal},\cite[Thm 3.7]{Cra}). 
In contrast, a coproduct of the same flavor was used in \cite[Sec 2.1]{CELN} to define the differential in {\em string homology}, a knot invariant that detects the unknot.   Also the paper \cite{NaeWil}
found indications of non-homotopy invariance in the $S^1$--equivariant version of the coproduct.  And indeed, more recently 
Naef showed  through a computation with lens spaces that the coproduct is not homotopy invariant \cite{Nae21}! (See also \cite[Sec 2.3-4]{NRW} for an overview and generalization of Naef's example.)
Note though that, over the rationals and for simply-connected manifolds, Rivera-Wang showed  that the coproduct is homotopy invariant using a Hochschild homology model \cite{RivWan}. The restriction of the coproduct to the based loop space is also known to be homotopy invariant \cite{Hin10,Cli25}. 

\smallskip

Naef suggested in \cite{Nae21} the following transformation formula: given a homotopy equivalence $f:M_1\to M_2$ and writing $\vee_1$ and $\vee_2$ for the string coproducts of $M_1$ and $M_2$ respectively, for $A\in H_*(\La M_1)$, 
$$\vee_2(\La f_*(A))-(\La f\x \La f)_*(\vee_1 A)=\diag(\tau_{THH}(f))\wedge_2\La f_*(A) -\La f_*(A)\wedge_2 \diag(\tau_{THH}(f))$$
for $\tau_{THH}(f)\in H_1(\La M_2,M_2)$ the Dennis trace of the Whitehead torsion of $f$, mapped via the antidiagonal $$\diag:\La M_2\to \La M_2\x \La M_2$$ taking a loop $\ga$ to the pair $(\ga,\ga^{-1})$, and where $\wedge_2$ denotes the Chas-Sullivan product in $M_2$. 
The formula was recently  proved by Naef-Safronov \cite{NaeSaf}, with an alternative proof given by Kenigsberg-Porcelli in \cite{KenPor}  when $\pi_2(M)=0$. 
We prove here a version of this formula, with an alternative obstruction class, built from {\em higher homotopy data} for $f$: let $g$ be a homotopy inverse for $f$, with associated homotopies $h_1:M_1\x I\to M_1$, $h_2:M_2\x I\to M_2$  between the identity and $g\circ f$ and $f\circ g$ respectively, chosen so that there is also a further $K:M_1\x I^2\to M_2$ is a homotopy between $h_2\circ (f\x \id)$ and $f\circ h_1$.  (Given $f$, such a tuple $(g,h_1,h_2,K)$ always exists and is a contractible choice, see Section~\ref{sec:K}.)

\begin{Th}\label{thm:formula}
  Let $f:M_1\to M_2$ be a degree 1 homotopy equivalence between closed oriented manifolds, with associated homotopy data $(g,h_1,h_2,K)$,  and denote by $\vee_1$ and $\vee_2$ the string coproducts of the manifolds $M_1$ and $M_2$. 
  The map induced by $\La f:\La M_1\to \La M_2$ in homology respects the string coproduct, up to an error term given by the following formula: 
  for $A\in H_*(\La M_1)$, 
  $$\vee_2(\La f_*(A))-(\La f\x \La f)_*(\vee_1 A)=\diag(\T_f)\wedge_2\La f_*(A) -\La f_*(A)\wedge_2 \diag(\T_f)$$
  where  
  $\T_f\in H_1(\La M_2,M_2)$, depicted in Figure~\ref{fig:KK}, 
  can be represented by the 1-dimensional family of loops 
  $$\la_{m,m',t}=\overline{K(m,t,-)}*\mu_2*K(m',t,-)*\mu_1 \ \ \  \textrm{for}\ \ (m,m',t)\in \D_f$$ 
  where $\D_f$ is a 1-dimensional submanifold of $M_1^2\x I$ determined by the conditions $f(m)=f(m')$ and $h_1(m,t)=h_1(m',t)$ (made transverse), and   
 where $\mu_1,\mu_2$ are short connecting paths. 
\end{Th}

                                                        \begin{figure}[ht]
                    \begin{lpic}{KK3(0.3,0.3)}
                      \lbl[tr]{32,35;$f\!\circ\! h_1(m,t)$}
                        \lbl[bl]{206,105;$f\!\circ\! h_1(m',t)$}
                     \lbl[tl]{75,52;$h_2(f(m),t)$}
                      \lbl[tl]{125,71;$h_2(f(m'),t)$}
                         \lbl[l]{18,75;$K(m,-,-)$}
                     \lbl[l]{100,110;$K(m',-,-)$}
                      \lbl[bl]{58,148;$f(m')$}
                       \lbl[br]{46,142;$f(m)$}
                         \lbl[t]{90,0;$f\!\circ\! g\!\circ\! f(m)$}
                           \lbl[tl]{130,11;$f\!\circ\! g\!\circ\! f(m')$}
                      \end{lpic}
                    \caption{The dotted line is a loop $\la_{m,m',t}$ in the family defining $T_f$ when $h_1(m,t) \sim h_1(m',t)$}\label{fig:KK}
                    \end{figure}

More explicitly, let $\tilde f\simeq f$ be a small transverse deformation of $f$ and let $\tilde h_1$ be likewise a transverse deformation of $h_1$. 
Then we can choose  $$\D_f=\{(m,m',t)\in M_1^2\x I\ |\ f(m)=\tilde f(m) \ \textrm{and}\ h_1(m,t)=\tilde h_1(m',t)\},$$
 and for $(m,m',t)\in \D_f$, the loop $\la_{m,m',t}$ is the concatenation  
 $$f(h_1(m,t)) \xrightarrow{\overline{K(m,t,-)}}  h_2(f(m),t) \xrightarrow{\mu_2} h_2(f(m'),t) \xrightarrow{K(m',t,-)} f(h_1(m',t))\xrightarrow{\mu_1} (h_1(m,t)).$$ 
(See Definition~\ref{def:Tf} and Proposition~\ref{prop:TfK} for more details.)

A first ingredient in the construction of $\D_f$ is  
  a transverse version of the  {\em fake diagonal}, the space $\De_1^f=\{(m,m')\in M_1\x M_1\ |\ f(m)=f(m')\}$ of points of $M_1\x M_1$ that are ``on the diagonal'' according to $f$. 
As stated in Corollary~\ref{cor:fake}, the difference between the actual diagonal and the fake diagonal $\De_1^f$ is a first obstruction to homotopy invariance of the coproduct.

 \smallskip

Kenigsberg and Porcelli also give a version of Naef's formula in \cite{KenPor},  with an obstruction class built using h-cobordisms. This latter obstruction class is known from the work of Geoghegan-Nicas \cite{GeoNic} to agree with $\tau_{THH}(f)$ at least when $\pi_2(M)=0$. We expect that  the three obstruction classes appearing in the papers \cite{NaeSaf,KenPor} and here all agree, but this does not follow directly from the above statement and its analogue in \cite{NaeSaf,KenPor}. A comparison of the three obstruction classes would yield new descriptions of the trace of Whitehead torsion, which would be of independent interest. See Remark~\ref{rem:KP} for first steps towards a comparison between our obstruction and that of \cite{KenPor}.

\medskip

We use in the present paper the  Cohen-Jones approach to string topology \cite{CohJon}, with the products and coproducts described above defined as lifts to the loop space $\La M$ of the intersection product on $H_*(M)$, using a Thom-Pontryagin construction for the intersection product. In the spirit of e.g.; \cite{GruSal} and \cite{NRW}, we will consider more general {\em intersection products}: 
to  maps $\R\to \E\to M\x M$, as in \cite{NRW},  we associate a map 
$$int_M: C_*(\E,\R)\xrightarrow{int_{U_M}} C_{*-n}(\E|_{U_M},\R|_{U_M}) \xrightarrow{\ \ r\ \ }  C_{*-n}(\E|_M,\R|_M)$$
where the first map comes from capping with the pull-back of the Thom class of the diagonal embedding $M\inc M\x M$, and the second is induced by the retraction $U_M\to M$ of a tubular neighborhood of this embedding onto its 0-section, and is well-defined e.g.; when $\E|_{U_M}\to U_M$ and $\R|_{U_M}\to \E|_{U_M}\to U_M$ are fibrations. See Section~\ref{sec:intprod} for more details. 
This chosen level of generality for intersection products is forced upon us by the proof of Theorem~\ref{thm:formula}, where several ``exotic'' intersection products appear.

We consider in this paper the question of whether a map $F=(F,F_0): (\E_1,\R_1)\rar (\E_2,\R_2)$ preserves the intersection product in two types of situations, as depicted in the following diagram: 
\begin{align*}
 \xymatrix@R=0.5pc@C=1pc{
 \R_1\ar[dddr]\ar[dr] \ar@{-}[rr]^-{F_0} & & \R_2 \ar[dr]\ar@{..>}[dddr] & &&  & \R_1\ar[dr]\ar[dd]\ar@{-}[rr]^-{F_0} & & \R_2 \ar[dr]\ar@{..>}[dd] &
\\
&  \E_1\ar[rr]^-(.4)F \ar[dd]^(.35){p_1} & & \E_2\ar[dd]^(.4){p_2}  &&&  &  \E_1\ar[rr]^-(.35)F \ar[dd]^(.35){p_1} & & \E_2\ar[dd]^(.4){p_2} 
\\
&&&&&  & {M_1}  \ar[dr]\ar@{..}[rr] & & {M_2}\ar[dr] &
\\
\textrm{(A)}   & M_1\x M_2\ar[rr]^-{f\x f}  & & M_2\x M_2 & & &  \textrm{(B)}\ \  & M_1\x M_2\ar[rr]^-{f\x f}  & & M_2\x M_2 
}
\end{align*}
We will call {\em Assumption (A)}, the case where both $\E_i$ and $\R_i$ are fibrations over $M_i\x M_i$, and {\em Assumption (B)}, the case where $\R_i$ is instead a fibration over the diagonal $M_i$. 
While the intersection product relevant to Theorem~\ref{thm:formula} will be as in Assumption (B), where we will see that there is an obstruction to invariance, the proof will use that there is no obstruction to homotopy invariance for intersection products under Assumption (A):

\begin{Th}\label{thm:fibinv0} 
  Let $f:M_1\to M_2$ be a degree 1 homotopy equivalence and $F:(\E_1,\R_1)\to (\E_2,\R_2)$ a map of pairs of fibrations over the map $f\x f: M_1\x M_1\to M_2\x M_2$ (Assumption (A)).
  Then there is a chain homotopy 
  $$F\circ int_{M_1}\simeq_H int_{M_2}\circ F: C_*(\E_1,\R_1)\rar C_{*-n}(\E_2|_{{M_2}}, \R_2|_{{M_2}}).$$
\end{Th}

The above result is essentially stated as Theorem 4.11 in \cite{NRW},
where a sketch proof is given. We give here a full proof as it is an
important ingredient for the proof of Theorem~\ref{thm:formula}.
Theorem~\ref{thm:fibinv0} recovers for example the homotopy invariance of the Chas-Sullivan product when $\E_i=\maps(S^1\sqcup S^1,M_i)$ with $p_i$ the evaluation map at the basepoint of the circles, also in a relative version, see Corollary~\ref{cor:CSrel}, and a more general version for mapping spaces $\E_i=\maps(X_1\sqcup X_2,M_i)$, see Example~\ref{ex:pmap}. 

\medskip

The starting point of the proof of Theorem~\ref{thm:formula} is the description of the string coproduct as a relative version of Tamanoi's trivial coproduct, the coproduct looking only for self-intersections at time $\frac{1}{2}$, see Section~\ref{sec:trivial}. The trivial coproduct satisfies the assumptions of Theorem~\ref{thm:fibinv0}, and thus the failure of invariance of the string coproduct can be described in terms of the failure of the homotopy for the invariance of the trivial coproduct to preserve the relative part. (See Lemma~\ref{lem:alg}.) Here the trivial coproduct will be considered relative to half-constant loops, a situation satisfying Assumption (B) since the restriction of the evaluation map $\e_{0,\ha}:\La M\to M\x M$ to half-constant loops is a fibration over $M$ rather than over $M\x M$.

More generally, for a map of pairs $F:(\E_1,\R_1)\to (\E_2,\R_2)$ over $f\x f$ as in Assumption (B), Theorem~\ref{thm:fibinv0} gives that $F_*: C_*(\E_1)\to C_*(\E_2)$ respects the intersection products up to homotopy, but not necessarily relatively to $\R_1$ and $\R_2$.  
A first step in the proof of Theorem~\ref{thm:formula} is a general computation of the failure of $F_*: H_*(\E_1,\R_1)\to H_*(\E_2,\R_2)$ to preserve the intersection products under Assumption (B): for $A\in H_*(\E_1,\R_1)$, Theorem~\ref{thm:Hcommute0} gives that the failure of invariance can be computed as 
\begin{equation}\label{eq:obs}
  \big(F\circ int_{M_1}- int_{M_2}\circ F\big) (A)=int_{U_1}(L_{N_2}(\al(\del A)\x I)),
  \end{equation}
  where $\del:H_*(\E_1,\R_1)\to H_{*-1}(\R_1)$ is the boundary map in the long exact sequence in homology, $\al$ is a map with the property that it takes the class of constant loops $[M_1]$ to the transverse fake diagonal $\underline{\De}_1^f$ mentioned above (see Proposition~\ref{prop:alphaDe1}), and $L_{N_2}$  is a certain (a priori non-relative!) homotopy equivalence, whose existence in this general context is given by Proposition~\ref{prop:j}. Restricting our attention to the case of mapping spaces $p_i:\E_i=\maps(S,M_i)\rar M_i\x M_i$, with the map $p_i$ evaluating at two points of $S$, 
we give in Section~\ref{sec:invert} an explicit construction of the map $L_{N_2}$ using the higher homotopy data $(g,h_1,h_2,K)$ for $f$.

The formula given in Theorem~\ref{thm:formula} will be obtained from \eqref{eq:obs}. In the case of the string coproduct, for a class $A\in H_*(\La M_1)$, we first note that the relevant boundary 
can be described, {\em in  half-constant loops}, as $[M_1]\wedge_1 A-A\wedge_1 [M_1]$, for $\wedge_1$ the string product of $M_1$. We will show that the operations  $[M_1]\wedge_1-$ and $- \wedge_1 [M_1]$ can be ``pulled outside'' in the right hand side of \eqref{eq:obs}, at the cost of replacing the class of constant loops $[M_1]$ with the class $\diag(\T_f)$. This latter class is essentially the result of applying the sequence of maps on the right hand side of \eqref{eq:obs} to $[M_1]$. The final form of the formula uses Theorem~\ref{thm:fibinv0} for the relative Chas-Sullivan product given in Corollary~\ref{cor:CSrel}, to replace the string product $\wedge_1$ computed in $M_1$ by the product $\wedge_2$ of $M_2$.

\addtocontents{toc}{\SkipTocEntry}
\subsection*{Higher dimensional loops}

The methods of the present paper apply to invariance questions for quite general relative intersection product. We exemplify this here through a
generalization of the string coproduct to higher dimensional loops, that will make some features of the formula in Theorem~\ref{thm:formula} more visible. 

Let $\La^rM:=\maps(S^r,M)=\maps((I^r/\del I^r),M)$ and consider the evaluation map
$$\e_{{\bf 0},{\bf t}}: \La^rM\x I^r\to M\x M$$
evaluating the loops at the basepoint ${\bf 0}$ and at time ${\bf t}=(t_1,\dots,t_r)\in I^r$. We define the higher coproduct $\vee^r$ as the operation  
\begin{multline*}
  \vee^r: H_*(\La^rM)\xrightarrow{\x I^r} H_{*+r}(\La^rM\x I^r,\La^rM\x \del I^r)\xrightarrow{int_M} H_{*+r-n}(\e_{{\bf 0},{\bf t}}^{-1}(\De M),\La^rM\x \del I^r) \\
  \xrightarrow{\reread} H_{*+r-n}(\maps(T^r,M),\R_{T^r})
  \end{multline*}
  where $T^r:=S^1\x S^{r-1} $ and the last map cuts the self-intersecting sphere $S^r=I^r/\del I^r$ to form a kind of torus, taking a pair $(\ga,{\bf t})$ to the composition
  $$T^r=S^1\x S^{r-1}\cong S^1\x \del I^r \rar (I^r/\del I^r)/_{{\bf 0}\sim {\bf c}=(\ha,\dots,\ha)} \xrightarrow{\ \theta_{\bf t}\ }  (I^r/\del I^r)/_{{\bf 0}\sim {\bf t}}\xrightarrow{\ \ga\ } M$$
 where the first map uses that the rays between the center ${\bf c}\in I^r$ and the points of $\del I^r$ become circles under the equivalence ${\bf 0}\sim {\bf c}$, and where $\theta_{\bf t}$ is a reparametrization map defined by piecewise linear scaling in each of the $r$ directions. This ``reread" map takes $\La^r\x \del I^r$ to a subspace $\R_{T^r}$ of half-constant maps. 
 (See Section~\ref{sec:higher} for more details.) 

  When $r=1$,  after identifying $\maps(S^1\x S^0,M)$ with $\La M\x \La M$,  the reread map precisely corresponds to the map that cuts a self-intersecting loop to reread it as a left and a right loop, and the coproduct $\vee^r$ specializes to the string coproduct $\vee$ when $r=1$.

\smallskip

There is a map $s^{r-1}:\La M=\maps(S^1,M) \to \maps(T^r,M)$ defined by precomposing with the projection $T^r=S^1\x S^{r-1}\to S^1$ on the first factor. This map respects constant maps, and 
hence the obstruction class $\T_f\in H_1(\La M_2,M_2)$ defines a class $s^{r-1}\T_f\in H_{1}(\maps(T^{r},M_2),M_2)$. In the case $r=1$, the class $s^0\T_f$ identifies the anti-diagonal $\diag(\T_f)$ appearing in Theorem~\ref{thm:formula}.

\smallskip

  Theorem~\ref{thm:formula} generalizes to the following:

\begin{Th}\label{thm:formula2}
  Let $f:M_1\to M_2$ be a degree 1 homotopy equivalence between closed oriented manifolds and denote by $\vee^r_1$ and $\vee^r_2$ the higher string coproducts of the manifolds $M_1$ and $M_2$ defined as above. 
  The map induced by $\La^r f:\La^r M_1\to \La^r M_2$ in homology respects these coproducts, up to an error term given by the following formula: for $A\in H_*(\La^r M_1)$, 
  $$\vee^r_2\big(\La^r f_*(A)\big)\ -\ \maps(T^{r},f)_*\big(\vee^r_1 A\big)=(-1)^{(r-1)n}s^{r-1}\T_f\,\wedge^2_{\del I^r}\,\La^r f_*(A)$$
  in $H_{*+r-n}(\maps(T^{r},M_2),\R_{T^r})$, where
  \begin{multline*}
    \wedge_{\del I^r}^2: H_*\big(\maps(T^{r},M_2)\x\La^r M_2\big) \xrightarrow{int_{M_2}}
    H_{*-n}\big(\maps(T^{r},M_2)\x_{M_2}\La^r M_2\big)  \\
    \xrightarrow{[-,-]} H_{*+r-1-n}\big(\maps(T^{r},M_2)\big)\end{multline*}
  is an intersection product composed with a form of Browder bracket $[-,-]$ (see Section~\ref{sec:higher}).
  \end{Th}

In particular the failure of invariance for such operations is ruled by the same obstruction class $\T_f$ as in the case of single loop spaces $\La M$ in Theorem~\ref{thm:formula}. 
Both results will be proved together, Theorem~\ref{thm:formula} being the case $r=1$ in Theorem~\ref{thm:formula2}.

\medskip

\noindent
{\em Organization of the paper.}  In Section~\ref{sec:int}, we give a  proof of the homotopy invariance of the homology intersection product  (Proposition~\ref{prop:invint}) 
using only its definition in terms of capping with a Thom class. This proof is not at all efficient, but the intermediate results proved in that section will be used later in the paper. 
In Section~\ref{sec:intprod}, we define lifted intersection products, in the generality needed in the paper. We treat there the invariance property of intersection products over a fixed manifold. Section~\ref{sec:inv1} considers the question of  invariance of intersection products over a homotopy equivalence, proving Theorem~\ref{thm:fibinv0} in Section~\ref{sec:assA}, the invariance under Assumption (A),  and Theorem~\ref{thm:Hcommute0} in Section~\ref{sec:assB}, giving a first description of the failure of invariance under Assumption  (B). 
In Section~\ref{sec:map}, we restrict attention to the case of mapping spaces, and use higher homotopy data to give in Proposition~\ref{prop:L12} an explicit version of the homotopy $L_{N_2}$ appearing in Theorem~\ref{thm:Hcommute0}. This description is used in  Section~\ref{sec:finalproof} to deduce Theorems~\ref{thm:formula} and~\ref{thm:formula2}. 
Finally Appendix~\ref{app:tubcomp} details the relationship between the tubular neighborhoods of two composable embedding and the tubular neighborhood of their composition.

\addtocontents{toc}{\SkipTocEntry}
\subsection*{Acknowledgements}
This paper originates in a joint work with Nancy Hingston, and the spirit of our joint paper \cite{HinWah0} is still very much present here. 
The first version of this paper contained a ``proof'' of homotopy
invariance of the coproduct that was contradicted by Naef's example
\cite{Nae21}. The error in the original version of the paper originated in a bad choice of relative term making a suspension ``isomorphism'' not an isomorphism anymore. 
We have
here corrected that relative term to revive the isomorphism part of the 
suspension isomorphism, and this makes the  less standard map denoted $\llbracket
\bar p^*\taue\cap\rrbracket$ in e.g.; Diagram~\eqref{equ:copdiag2.0} not necessarily an isomorphism, a fact tightly related to the failure of the homotopy called $L_{N_2}$ in Theorem~\ref{thm:Hcommute0} (and Proposition~\ref{prop:L12}), to be a relative homotopy equivalence. 
 We thank Florian Naef for discussing his
example with us early on, as well as for extended discussions on the topic of invariance over the following years. The revision of this paper has also been influenced by exchanges with Alexandru Oancea,  Kai Cieliebak and Noah Porcelli, and we thank Maxime Ramzi for helpfull discussions about homotopy equivalences. 
In the course of this work, the  author was
supported  by the Danish National Research Foundation through
the Centre for Symmetry and Deformation (DNRF92) and the Copenhagen Centre for Geometry and Topology (DNRF151), the  European Research
Council (ERC) under the European Union's Horizon 2020 research and innovation programme (grant agreement No.~772960), the Max Planck Institute in Bonn, the Heilbronn Institute for Mathematical Research, and 
would like to thank the Isaac Newton Institute for Mathematical Sciences, Cambridge, for support and hospitality during the programme Homotopy Harnessing Higher Structures (EPSRC grant numbers EP/K032208/1 and EP/R014604/1).

\tableofcontents

\section{Homotopy invariance of the intersection product}\label{sec:int}

The easiest way to prove the homotopy invariance of the intersection product on the homology of closed manifolds is to use its definition
as Poincar\'e dual of the cup product, and use the homotopy invariance of the cup product. In fact, the compatibility is easily checked with
any degree $d$ map, up to appropriate scaling, see e.g., \cite[VI, Prop 14.2]{Bredon}. We will present here  a proof of the homotopy invariance using instead the definition of the intersection product via capping with a Thom class, as a warm-up, and partial preparation, for the invariance of the string topology coproduct.

\medskip

Let $M$ be a closed oriented Riemannian manifold of injectivity radius $\rho$, and let $TM$ denote its $\eps$--tangent bundle, where $\eps$ is small compared to $\rho$. (To be precise, the paper \cite{HinWah0} whose constructions we use assumes $\eps<\frac{\rho}{14}$.) We will also assume that $\eps<\frac{1}{2}$.

The  normal bundle of the diagonal embedding
$M\inc M\x M$ is isomorphic to $TM$, and we can choose as tubular embedding the map   $\nu_M:TM\to M\x M$ given by $\nu_M(m,V)=(m,\exp_mV)$. (By an {\em embedding}, we will in this paper always mean a $C^1$--map that is an immersion and a homeomorphism onto its image.)
Note that $\nu_M$ has image the $\eps$--neighborhood of the diagonal 
 $$U_{M}=\{(m,m')\in M\x M \ |\ |m-m'|<\eps\}\subset M\x M.$$
 Let  $\tau_M\in C^n(M\x M,(\De M)^c)$ be a chosen Thom class for this tubular neighborhood, in cochains relative to the complement of the diagonal $\De: M\to M\x M$, and
 $r\colon U_M\to M$ 
 the associated retraction. After going through ``small simplices'' to get a chain inverse $\rho$ to excision, capping with the Thom class defines a map
$$[\tau_M\cap]\colon   C_{p+q}(M\x M,(\De M)^c)\xrightarrow{\ \rho\ } C_*(U_M,(\De M)^c) \xrightarrow{\tau_M\cap} C_{*-n}(U_M).$$
(See eg., \cite[A.2]{HinWah0} for a detailed description.)

Up to sign, the intersection product on a closed manifold $M$ can be defined as the composition 
$$C_p(M)\ot C_q(M) \xrightarrow{\x} C_{p+q}(M\x M)\rar  C_{p+q}(M\x M,(\De M)^c)\xrightarrow{[\tau_M\cap]} C_{p+q-n}(U_M) \xrightarrow{\ r\ } C_{p+q-n}(M)$$ 
(See eg., \cite[App B]{HinWah0}.)

\medskip

Let $f:M_1\to M_2$ be a homotopy equivalence between two closed oriented Riemannian manifolds $M_1$ and $M_2$. In particular  $f$ has degree $\pm 1$, that is $f[M_1]=\pm [M_2]$. 
Let  $U_i=U_{M_i}$ be the $\eps_i$--tubular neighborhoods of the diagonal in $M_i\x M_i$ as above, with Thom class $\tau_{M_i}$ and $r_i:U_i\to M_i$ the associated retraction.
After changing $f$ in its homotopy class, we may assume that
\begin{itemize}
\item $f$ is smooth;
  \item $\eps_1$ is small enough so that $f(U_1)\subset U_2$,
  \end{itemize}
  where the last property uses that $M_1$ is assumed closed and hence is compact. 
  Leaving out the cross-product, which is well-known to be natural, we want to show that the diagram 
$$\xymatrix{  H_{p+q}(M_1\x M_1)\ar[r]\ar[d]_{f\x f}&  H_{p+q}(M_1\x M_1,(\De M_1)^c)\ar[r]^-{[\tau_{M_1}\cap]} &  H_{p+q-n}(U_{1})\ar[d]^{f\x f} \ar[r]^-{r_1}&  H_{p+q-n}(M_1)\ar[d]^f\\
H_{p+q}(M_2\x M_2)\ar[r]&  H_{p+q}(M_2\x M_2,(\De M_2)^c)\ar[r]^-{[\tau_{M_2}\cap]} &  H_{p+q-n}(U_{2}) \ar[r]^-{r_2}&  H_{p+q-n}(M_2)
}$$
commutes. 
One difficulty is that we cannot assume that, in general, the map $f\x f$ does not take  the complement of the diagonal in $M_1\x M_1$ to the corresponding complement in $M_2\x M_2$, so that we do not have a vertical map in the second column in the diagram. For the same reason, we cannot use $f$ directly to compare the Thom classes $\tau_{M_1}$  and $\tau_{M_2}$. To make up for this, we will replace $f$ by an embedding and use instead the larger diagram (\ref{equ:intdiag}) below.

\medskip

Pick an embedding $e: M_1\inc D^k$, for some $k$ large. Then  the map $(f,e)\colon M_1\to M_2\x D^k$ is an embedding too and we can consider 
the composition of embeddings
\begin{equation}\label{copembeddings}
\xymatrix{ \ar@{^(->}[rr]^{\De_{1}} M_1 \ar@{_(->}[drr]_{(\id,f,e)}  && M_1 \x M_1\ar@{^(->}[d]^{id\x (f,e)}  \\
&& M_1\x M_2 \x D^k. 
}
\end{equation}
To this situation, we can associate a diagram 
\begin{equation*}
\xymatrix{TM_1\ar[d]\ar@{^(->}[drr]^{\nu_1=\nu_{M_1}} &&  \ \ \ \ \ \ N_2=M_1\x N(f,e)\ar[d]_{p_{N_2}}\ar@{^(->}[drr]^{\ \ \nu_2=id\x \nu_e} && \\
M_1 \ar@{^(->}[rr]^-{\De_1}&& M_1\x M_1 \ar@{^(->}[rr]^-{id\x(f,e)}&& M_1\x M_2\x D^k,}
\end{equation*}
where $\nu_1:=\nu_{M_1}:TM_1\to M_1\x M_1$ is the tubular neighborhood of $\De M_1$ defined  above, 
and $N(f,e)$ is the normal bundle of $(f,e)$ with 
$\nu_e:N(f,e)\to M_2\x D^k$ a chosen tubular embedding; we choose here, as in \cite[Thm 12.11]{BroJan}, to define $\nu_e$ using the restriction of the exponential map of $M_2\x D^k$ to the subbundle $N(f,e)\le TM_2\op \RR^k$ of vectors normal to $d(f,e)TM_1$. 
Pick  a Thom class $\tau_e\in C^k(M_2\x D^k,((f,e)M_1)^c)$  for $\nu_e$. Then 
\begin{equation}\label{equ:tau2} 
\taue:=1\x \tau_e\in C^k(M_1\x M_2\x D^k,M_1\x ((f,e)M_1)^c)
\end{equation}
is a Thom class of $\nu_2$. 
Applying Proposition~\ref{prop:comptub}, 
we now have  the following:

\begin{prop}\label{prop:tub12} Consider the sequence of embeddings $M_1\sta{\De_{1}}\inc M_1\x M_1\sta{id\x (f,e)}\inc M_1\x M_2\x D^k$ and let $\nu_1,\nu_2$ and $\tau_1=\tau_{M_1},\tau_2=\taue$ be the embeddings and Thom classes chosen above. Then the following hold: 
\begin{enumerate} 
\item The bundle $N_{3}:=TM_{1}\oplus N(f,e) \rar M_1$ is isomorphic to the normal bundle of the composed embedding $(id,f,e):M_1\to M_1\x M_2\x D^k$. 
\item There is an isomorphism $\hat\nu_1:N_3 =TM_{1}\oplus N(f,e)\xrightarrow{\cong} N_2|_{U_1}=(M_1\x N(f,e))|_{U_1}$, as bundles over $TM_1\cong U_1$, with the property that
  $$\nu_{3}:=\nu_2\circ\hat\nu_1:N_3\rar M_1\x M_2\x D^k$$ 
  is a tubular embedding for $(id,f,e)$.
  Explicitly,  $$\nu _{3}(m,V,W)=\nu_2(\nu _{1}(m,V),W^{\prime })=(m,\nu_e(\exp_mV,W'))$$
  for $W^{\prime }\in (N(f,e))_{\nu _{1}(m,V)}$ the parallel transport of $W$ along the path $\nu _{1}(m,tV)$.
\item  The class $$\tau_{3}:=p_{N_2}^*\tau_{M_1}\cup \taue\in C^{n+k}(M_1\x M_2\x D^k,M_1^c)$$ is a Thom class for the tubular embedding $\nu_{3}$, where $M_1$ is identified with its image under the map $(\id,f,e)$. Moreover, the diagram 
$$\xymatrix{H_{*+k}(M_1\x M_2\x D^k)\ar[rr]^-{\llbracket \tau_e\cap\rrbracket} \ar[drr]_{\llbracket\tau_{3}\cap\rrbracket} && H_{*}(M_1\x M_1)\ar[d]^{\llbracket\tau_{M_1}\cap\rrbracket}  \\
& & H_{*-n}(M_1) }$$
commutes,  where we use the notation $\llbracket\tau\cap\rrbracket\colon C_*(X) \to C_*(Y)$ for maps that are compositions of the form 
$$\llbracket\tau\cap\rrbracket : C_*(X)\to C_*(X,Y^c) \xrightarrow{[\tau\cap]} C_*(U)\xrightarrow{r} C_*(Y)$$ 
for $U$ a tubular neighborhood of $X$ inside $Y$ with associated Thom class $\tau$ and retraction $r$.    
\end{enumerate}
\end{prop}

\medskip

The embedding $(\id,f,e):M_1\to M_1\x M_2\x D^k$ is also equal to  the composition of embeddings
\begin{equation}\label{copembeddings2}
\xymatrix{ \ar@{^(->}[rr]^{(\id,f)} M_1 \ar@{_(->}[drr]_{(\id,f,e)}  && M_1 \x M_2\ar@{^(->}[d]^{(\id, e\circ \pi)}  \\
&& M_1\x M_2 \x D^k. 
}
\end{equation}
for $(\id,e\circ \pi):M_1\x M_2\to M_1\x M_2\x D^k$ taking $(m,n)$ to $(m,n,e(m))$. 
This gives another associated  diagram
of tubular embeddings, namely 
\begin{equation*}
\xymatrix{f^*TM_2\ar[d]\ar@{^(->}[drr]^{\bar\nu_1} &&  \bar
  N_2=M_1\x M_2\x \RR^k \ar[d]_{p_{\bar N_2}}\ar@{^(->}[drr]^{\ \bar\nu_2} && \\
M_1 \ar@{^(->}[rr]^-{(\id,f)}&& M_1\x M_2 \ar@{^(->}[rr]^-{(\id,e\circ \pi)}&& M_1\x M_2\x D^k,}
\end{equation*}
where we choose the explicit tubular embeddings 
$$\bar\nu_1(m,V)=(m,\exp_{f(m)}V) \ \ \ \ \textrm{and}\ \ \ \ \bar\nu_2(m,n,x)=(m,n,e(m)+\ell(x))$$ 
with $\ell$ a scaling map so that $\bar \nu_2$ is well-defined, recalling also that $TM_2$ denotes the $\eps_2$--tangent bundle of $M_2$ so that $\bar \nu_1$ is well-defined.
Note that $\bar \nu_1$ has image
$$U_{1,2}:=\{(m,n)\in M_1\x M_2\ |\ |f(m)-n|<\eps_2\}=(f\x1)^{-1}(U_2)$$
for $U_2\subset M_2\x M_2$ the $\eps_2$--neighborhood of the diagonal as above.

\begin{lem}\label{lem:pull2} The map $f\x 1$ induces a map of pairs $f\x 1:(M_1\x M_2,M_{1}^c)\to (M_2\x M_2,(\De M_2)^c)$ and, if $f$ is a homotopy equivalence, 
  the class  
$(f\x 1)^*\tau_{M_2}\in C^n(M_1\x M_2,M_{1}^c)$ is a Thom class for the tubular neighborhood $(f^*TM_2,\bar \nu_1)$. (Here $M_1^c$ denotes $((\id,f)M_1)^c\in M_1\x M_2$.) 
\end{lem}

\begin{proof}
  The Thom class $\tau\in C^n(M_1\x M_2,M_{1}^c)$ is characterized by the property $\tau\cap [M_1\x M_2]=[(1\x f)M_1]$. We need to check that the class $\tau=(f\x 1)^*\tau_{M_2}$ satisfies this equation. 
  By the naturality of the cap product,
\begin{align*}
  (f\x 1)_*((f\x 1)^*\tau_{M_2}\cap [M_1\x M_2])&=\tau_{M_2}\cap (f\x 1)_*[M_1\x M_2]\\
  &=\deg(f) \tau_{M_2}\cap [M_2\x M_2]=\deg(f)[\De M_2].
\end{align*}
On the other hand, $(f\x 1)_*[(1\x f)M_1]=(f\x f)[\De M_1]=\deg(f)[\De M_2]$.  The result then follows from the fact that $(f\x 1)_*$ is invertible.  
\end{proof}

For the tubular neighborhood $(M_1\x M_2\x \RR^k,\bar \nu_2)$, the inclusion $M_1\x M_2\x \del D^k\inc (M_1\x M_2)^c\subset M_1\x M_2\x D^k$, where the second inclusion takes $(m,n)$ to $(m,n,e(m))$,   
induces a quasi-isomorphism 
$$\bar i^*: C^*(M_1\x M_2\x D^k,M_1\x M_2\x \del D^k)\xrightarrow{\simeq} C^*(M_1\x M_2\x D^k,(M_1\x M_2)^c)$$
Let $$\etab:=\bar i^*(1\x\eta)$$
for $\eta\in C^k(D^k,\del D^k)$ a chosen representative of the dual of the fundamental class.
Then $\etab$ is a Thom class for $\bar N_2$ as $\bar i_*(\etab\cap [M_1\x M_2\x D^k])=(1\x \eta)\cap \bar i_*[M_1\x M_2\x D^k]=[M_1\x M_2\x \{0\}]$ is homologous to the image of the zero-section of $\bar N_2$ under $\bar i_*$. 
Indeed, the embedding $(\id,e\circ \pi)$ is isotopic to the embedding $(\id,0)$, replacing $e$ by the constant map at $0\in D^k$, through an isotopy only affecting the disc.

\medskip

Applying  Proposition~\ref{prop:comptub} and Lemma~\ref{lem:pull2} to this new factorisation of the embedding $(\id,f,e)$,  we get

\begin{prop}\label{prop:tub12bis} Consider the sequence of embeddings $M_1\sta{(\id,f)}\inc M_1\x M_2\sta{(\id,e\circ \pi)}\inc M_1\x M_2\x D^k$ and let $\bar\nu_1,\bar\nu_2$ and $\bar\tau_1=(f\x 1)^*\tau_{M_2},\bar\tau_2=\etab$ be the embeddings and associated Thom classes chosen above. Assume that $f$ is a homotopy equivalence. Then the following hold: 
\begin{enumerate} 
\item The bundle $\bar N_{3}:=f^*TM_{2}\oplus \RR^k \rar M_1$ is isomorphic to the normal bundle of the composed embedding $(id,f,e):M_1\rar M_1\x M_2\x D^k$. 
\item There is an isomorphism $\tilde\nu_1:\bar N_3=f^*TM_{2}\oplus \RR^k\xrightarrow{\cong} 
 \bar N_2|_{U_{1,2}}= M_1\x M_2\x \RR^k|_{U_{1,2}}$ as bundles over $f^*TM_2\cong U_{1,2}$, with the property that 
$$\bar\nu_{3}:=\bar\nu_2\circ\tilde\nu_1\colon \bar N_3\rar M_1\x M_2\x D^k$$ 
is a tubular embedding for $(\id,f,e)$. 
Explicitly,  $$\bar\nu _{3}(m,V,W)=\bar\nu_2(\bar\nu _{1}(m,V),W)=(m,\exp_{f(m)}V,e(m)+W)$$
  for $W\in \RR^k_{(m,f(m))}\equiv \RR^k_{(m,\exp_{f(m)}V)}$.
\item  The class $$\bar\tau_{3}:=p_{\bar N_2}^*(f\x 1)^*\tau_{M_2}\cup \etab\in C^{n+k}(M_1\x M_2\x D^k, M_{1}^c)$$ is a Thom class for the tubular embedding $\bar\nu_{3}$ and the diagram 
$$\xymatrix{H_{*+k}(M_1\x M_2\x D^k)\ar[rr]^-{\llbracket \etab\,\cap\rrbracket} \ar[drr]_{\llbracket\bar\tau_{3}\,\cap\rrbracket} && H_{*}(M_1\x M_2)\ar[d]^{\llbracket(f\x 1)^*\tau_{M_2}\cap\rrbracket}  \\
& & H_{*-n}(M_1) }$$
commutes,  where the notation $\llbracket \tau\cap\rrbracket$ is as in Proposition~\ref{prop:tub12}.   
\end{enumerate}
\end{prop}

We have given two explicit tubular neighborhoods $(N_{3},\nu_{3})$ and $(\bar N_{3},\bar \nu_{3})$ of the embedding $(\id,f,e)$, coming from two different ways of considering it as a composition of embeddings. 
The bundles $N_{3}$ and $\bar N_{3}$ are isomorphic, as they are both isomorphic to the normal bundle of the embedding. More than that, 
the uniqueness theorem for tubular neighborhoods implies that the embeddings are necessarily isotopic, and the following result holds:

\begin{prop}\label{prop:tub3}
Let $(N_{3},\nu_{3}),\tau_{3}$ and $(\bar N_{3},\bar\nu_{3}),\bar\tau_{3}$ be the tubular neighborhoods and Thom classes for the embedding $(\id,f,e)$ obtained in Propositions~\ref{prop:tub12} and \ref{prop:tub12bis}.  There exists a bundle isomorphism $\phi:N_{3}\to \bar N_{3}$ 
and a diffeomorphism $h=(\id\x\, \tilde h): M_{1}\x M_2\x D^k\rar M_{1}\x M_2\x D^k$, fibered over the projection to $M_1$, such that 
\begin{enumerate}
\item $h$ restricts to the identity on $M_1\stackrel{(\id,f,e)}\inc M_1\x M_2\x D^k$ and outside $U_{1,2}\x \mathring{D}^k$; 
 \item $h$ is isotopic over $M_1$ to  the identity relative to  the same subspaces; 
\item $h\circ \nu_{3}=\bar\nu_{3}\circ\phi : \delta N_{3}\rar M_1\x M_2\x D^k$, for $\delta N_3$ (resp.~$\delta\bar N_3=\phi(\delta N_3)$)  an appropriately chosen isomorphic subbundle. 
  \item $[\bar\tau_{3}]=[h^*\bar\tau_{3}]=[\tau_{3}]\in H^{n+k}(M_1\x M_2\x D^k,M_{1}^c)$. 
\end{enumerate}
\end{prop}

Note that (1) implies that $h(M_1^c)=M_1^c$ and (3)  that  $h(\nu_3\delta N_3)=\bar\nu_3 \delta\bar N_3$.

\begin{proof}
  Uniqueness of tubular neighborhoods (see e.g. \cite[Chap 4, Thm 5.3]{Hir94}) 
  gives that the tubular neighborhoods $(N_3,\nu_{3})$ and $(\bar N_3,\bar \nu_{3})$ are isotopic, and this is what we will use to produce the diffeomorphism $h$. To obtain detailed properties of the resulting $h$ in our particular case, we follow the actual construction of the isotopy in the proof of this theorem in \cite{Hir94}. The  first step in the proof is to replace $N_3$ by an isomorphic subbundle $\delta N_3$ such that $\nu_3(\delta N_3)\subset \bar \nu_3(\bar N_3)$. Then one considers the map $g=\bar \nu_3^{-1}\circ \nu_3:\delta N_3\to \bar N_3$. In our case, because $\nu_3$ and $\bar \nu_3$ both restrict to the identity on $M_1$ in the first factor, having the form $\nu_3(m,V,W)=(m,\nu_e(\exp_mV,W'))$ and $\bar \nu_3(m,X,Y)=(m,\exp_{f(m)}X,e(m)+Y)$, we have that $g$ has the form $g(m,V,W)=(m,X,Y)$ sends fibers to fibers. Now the map $g$ is fiberwise homotopic to its fiber derivative $T_{M_1}g:\delta N_3\to \bar N_3$ through the homotopy
  $$H(m,V,W,t)=\left\{ \begin{array}{ll} t^{-1} g(m,tV,tW) & t\in (0,1]\\
                         T_{M_1}g(m,V,W) & t=0
                       \end{array}\right.$$
We set $\phi=T_{M_1}g$ and obtained the desired isotopy $F: \delta N_3 \x I \to M_1\x M_2\x D^k$ from $\nu_3=\bar\nu_3\circ g$ to $\bar\nu_3\circ \phi$ by setting $$F(m,V,W,1-t)=\bar \nu_3(H(m,V,W,t)).$$

By the isotopy extension theorem  (see eg.~\cite[Chap 8, Thm 1.3]{Hir94}),  
there exists an isotopy of diffeomorphisms 
$${\bf F}: (U_{1,2}\x D^k) \x I \ \rar \ U_{1,2}\x D^k$$ 
such that ${\bf F}(-,0)$ is the identity on  $U_{1,2}\x D^k$ and such that  ${\bf F}$ restricts to $F$ on $\nu_{3}(\delta N_3)\x I$. As ${\bf F}$ is only non-trivial in a neighborhood of the image of $F$,
can also be chosen to be fiberwise, and $F$ has support inside $\bar N_3\cong TM_2\op \RR^k$, 
${\bf F}$ can be constructed so that it has support within the slices $\exp_{(m,f(m),e(m))}TM_2\op \RR^k\subset \{m\}\x M_2\x D^k$. 
It also fixes $(\id,f,e)M_1$ because $F$ fixes it. Hence by construction, $h:={\bf F}(-,1)$ is a diffeomorphism of $U_{1,2}\x D^k$ of the form $h=\id \x\, \tilde h$, that extends to $M_1\x M_2\x D^k$ via the identity, and satisfies (1)--(3) in the statement.

Because $h$ is a diffeomorphism compatible with the tubular embeddings, it pulls back the Thom class to a Thom class. Unicity of the Thom class in cohomology gives the last part of the statement.
\end{proof}

\begin{convention}
 Recall that $N_3\cong N_2|_{U_1}$ and $\bar N_3\cong \bar N_2|_{U_{1,2}}$. 
We replace $N_3$ by the smaller subbundle $\delta N_3$, and $\bar N_3$ by $\delta\bar N_3=\phi(\delta N_3)$, as defined in the above proof, scaling $U_1$, $N_2$, $U_{1,2}$, $U_2$ and $\bar N_2$ accordingly. 
  \end{convention}

Consider now the diagram 
\begin{equation}\label{equ:intdiag}
  \xymatrix{H_*(M_1\x M_1) \ar[dd]_{1\x f_*} \ar[rrrr]^-{\llbracket\tau_{M_1}\cap\rrbracket} 
    &&&& H_{*-n}(M_1) \ar@{=}[dd]\\
&  \ar@{}[l]|{\mbox{(1)}}
&  H_{*+k}(M_1\x M_2\x D^k, M_1\x M_2\x \del D^k) 
\ar[ull]_-{\llbracket \taue\,\cap\rrbracket}^-\cong 
\ar[urr]^-{\llbracket\tau_{3}\,\cap\rrbracket} \ar[dll]_-{\llbracket \etab\,\cap\rrbracket=[(1\x \eta)\cap] \ \ \ \ } \ar[drr]^-{\llbracket \bar\tau_{3}\,\cap\rrbracket}&
\ar@{}[r]|{\mbox{(2)}} &\\
H_*(M_1\x M_2) \ar[d]_{f_*\x 1}\ar[rrrr]^-{\llbracket(f\x 1)^*\tau_{M_2}\cap\rrbracket} 
& & && H_{*-n}(M_1) \ar[d]^{f_*}\\
 H_*(M_2\x M_2)\ar[rrrr]^-{\llbracket\tau_{M_2}\cap\rrbracket}   && && H_{*-n}(M_2) 
}\end{equation}

Invariance of the intersection product will follow from the commutativity of this diagram once we have also shown that the map $\llbracket\taue\,\cap\rrbracket$ is an isomorphism.
This last fact
will follow from the commutativity of (1) in the diagram, and the fact that the two other maps in that triangle are more easily seen to be homology isomorphisms, one of them because $f$ is a homotopy equivalence, and the other one as a  consequence of the homology suspension theorem. 
Note that the map $\llbracket\tau_e\,\cap\rrbracket$ in the diagram is not directly a case of a Thom isomorphism because the source of the map is not correct, so that we cannot deduce that it is an isomorphism from the Thom isomorphism theorem.

We analyse the diagram step by step. The top triangle commutes by Proposition~\ref{prop:tub12}(3), and the middle triangle by  Proposition~\ref{prop:tub12bis}(3). 
The bottom square commutes by the naturality of the cap product. 
So what we need  to show is that the triangles labeled (1) and (2) commute. 
To prove commutativity of  (1), we use the following lemma:

\begin{lem}\label{lem:degree} 
Let $i:(M_2\x D^k,M_2\x \del D^k) \rar (M_2\x D^k,((f,e)M_1)^c) $
be the inclusion of pairs. If $f$ is a degree $d$ map, then $$i^*[\tau_e]=d[1\x \eta]\in H^k (M_2\x D^k,M_2\x \del D^k)$$
for $\tau_e$ the Thom class of $(N(f,e),\nu_e)$ and $\eta\in C^k(D^k,\del D^k)$ a representative of the dual of $[D^n]$ as above.  
\end{lem}

\begin{proof}
Identifying $M_1$ with its image under $(f,e)$ in $M_2\x D^k$, consider the diagram 
$$\xymatrix{H^*(N(f,e),M_1^c) \ar[d]_{\cap\, [N(f,e)]}  & H^*(M_2\x D^k,M_1^c) \ar[r]^-{i^*} \ar[l]_-{\nu_e^*}^-\cong\ar[d]_{\cap \,[M_2\x D^k]} & H^*(M_2\x D^k,M_2\x \del D^{k}) \ar[d]^{\cap\, [M_2\x D^k]}\\
H_{n+k-*}(N(f,e)) \ar[r]^-{\nu_{e*}} & H_{n+k-*}(M_2\x D^k) & H_{n+k-*}(M_2\x D^k)  \ar[l]_-{i_*=id}
}$$
where the vertical maps cap with the fundamental class in each case and $\nu_e^*$ is an isomorphism by excision.  
The two squares commute by naturality of the cap product, given that $[M_2\x D^k]=i_*[M_2\x D^k]=\nu_{e*}[N(f,e)]\in H^*(M_2\x D^k,M_1^c)$. Hence in homology we have a commuting diagram 
$$\xymatrix{[\tau_e]\in & H^k(M_2\x D^k,M_1^c) \ar[r]^-{i^*} \ar[d]_{(\cap \,[N(e,f)])\circ \nu_e^*} & H^k(M_2\x D^k,M_2\x \del D^{k}) \ar[d]^{\cap [M_2\x D^k]} & \ni  \deg(f)[1\x \eta]\\
& H_n(N(f,e)) \ar[r]^-{\nu_{e*}} & H_n(M_2\x D^k) \ar[d]^\cong & \\
[M_1]\in & H_n(M_1) \ar[u]_\cong\ar[r]^-f & H_n(M_2) & \ni \deg(f)[M_2]
}$$
where the first square commutes by the commutativity of the previous diagram, and the bottom square commutes as it comes from a commuting square in the category of spaces. Now by definition of the Thom classes, $[\tau_e]$ is taken to $[M_1]$ along the left vertical maps, and $[1\x\eta]$ to $[M_2]$ along the right vertical maps as $[1\x\eta]$ is a Thom class for the trivial $D^k$--bundle on $M_2$. On the other hand $f$ takes $[M_1]$ to $d[M_2]$ along the bottom horizontal map, from which the result follows.  
\end{proof}

\begin{prop}\label{prop(1)}
  Assume that $f$ is a degree 1 map. The subdiagram labeled (1) in Diagram (\ref{equ:intdiag}) commutes.  
  Moreover, the map $\llbracket\etab\,\cap\rrbracket$ in the diagram is an isomorphism. 
\end{prop}

\begin{proof} Recall that $N_2=M_1\x N(f,e)$. 
Spelling out the maps, the diagram becomes the following: 
$$\xymatrix{H_*(M_1\x M_1)  
  \ar[dd]_{1\x f_*} & H_*(N_2) \ar@<3ex>@{}[dr]|(.8){(a)} \ar[l]_{\cong} \ar[d]_{\nu_{2*}} & H_{*+k}(M_1\x M_2\x D^k,(M_1\x M_1)^c) \ar[l]_-{[\taue\,\cap]}^-\cong \ar@/_2ex/[ld]^(0.5){\taue\,\cap} \\
& H_*(M_1\x M_2\x D^k)\ar[dl]_\cong & H_{*+k}(M_1\x M_2\x D^k,M_1\x M_2\x \del D^k) \ar@<-7ex>[u]_{1\x i_*}\ar@<7ex>[d]^{\bar i_*}_\cong 
\ar[l]_-{(1\x \eta)\,\cap}^-\cong \\
H_*(M_1\x M_2)& H_*(\bar N_2)  \ar@<-3ex>@{}[ur]|(.8){(b)}  \ar[l]^{\cong} \ar[u]^{\bar\nu_{2*}}  & H_{*+k}(M_1\x M_2\x D^k,(M_1\x M_2)^c) \ar[l]_-{[\etab\,\cap]}^-\cong \ar@/^2ex/[lu]_(0.5){\etab\,\cap} 
}$$

The subdiagram labeled $(a)$ commutes by naturality of the cap product since $i^*[\taue]=[1\x \eta]$ by Lemma~\ref{lem:degree} as $f$ has degree 1.
Likewise subdiagram $(b)$ commutes as $\etab=\bar i^*(1\x \eta)$ by definition. The rest commutes on the chain or space level by definition of the maps.

The isomorphisms marked in the diagram are given by the Thom isomorphism, excision, or homotopy equivalences coming either from contracting a disc or a bundle.
From this we see that the map  $\llbracket\etab\,\cap\rrbracket$  is a composition of  isomorphisms, and hence is itself an isomorphism.
\end{proof}

Note that the top part of the diagram in the proof shows that the left-hand map $1\x f_*$ is invertible if and only if the right-hand map $1\x i_*$ is, and these two maps are in this case, up the a twisted suspension, inverses of each other.

\begin{prop}\label{prop:invint}
Let $f:M_1\to M_2$ be a homotopy  equivalence between closed manifolds.  Then $f$ induces a ring map $H_*(M_1)\to H_*(M_2)$, with ring structure given by the intersection product, up to the sign $\deg(f)$, that is the diagram 
$$\xymatrix{H_*(M_1\x M_1) \ar[d]_{f_*\x f_*} \ar[rr]^-{\llbracket\tau_{M_1}\cap\rrbracket} && H_{*-n}(M_1) \ar[d]^{f_*}\\
H_*(M_2\x M_2)  \ar[rr]^-{\llbracket\tau_{M_2}\cap\rrbracket} && H_{*-n}(M_2)}$$
commutes  up to the sign $\deg(f)=\pm 1$. 
\end{prop}

We repeat here that this is {\em  not} the optimal result for invariance of the intersection product, and absolutely not the optimal proof! See eg., \cite[VI, Prop 14.2]{Bredon} for the stronger statement for any degree $d$ map as a direct consequence of Poincar\'e duality. The proof given here has though the advantage of being liftable to the loop space, as we will see in the following sections.

\begin{proof}
Given that $f$ is a homotopy equivalence, it has degree $\pm 1$. If $\deg(f)=-1$, changing the orientation of $M_2$, so that the degree of $f$ becomes 1, changes the intersection product of $M_2$ by a sign $(-1)$. Hence it is enough to consider the case where $\deg(f)=1$.

We need to show that Diagram (\ref{equ:intdiag}) commutes, and that the two maps $\llbracket\taue\,\cap\rrbracket$ and $\llbracket\etab\,\cap\rrbracket$  are isomorphisms.

We have already seen that  the top and  middle triangles commute by  Propositions~\ref{prop:tub12}(3) and~\ref{prop:tub12bis}(3), and the 
bottom square by the naturality of the cap product. The commutativity of the triangle labeled (1) and the fact that  $\llbracket\etab\,\cap\rrbracket$ is an isomorphism are given by Proposition~\ref{prop(1)}.
Given that $f$ is assumed to be a homotopy equivalence, $1\x f_*$ is an isomorphism and the commutativity of (1) implies that also $\llbracket\taue\,\cap\rrbracket$ is an isomorphism.

For  the triangle labeled (2), commutativity follows from Proposition~\ref{prop:tub3}. Indeed, the map $h$ in the proposition gives a commutative diagram 
$${\small \xymatrix{H_{*+k}(M_1\x M_2\x D^k,M_1\x M_2\x \del D^k) \ar[r]\ar[dr] & H_{*+k}(M_1\x M_2\x D^k,M_{1}^c) \ar[r]^-{[\tau_{3}\,\cap]} \ar[d]^h & H_{*-n}(N_{3}) \ar[r]\ar[d]^h & H_{*-n}(M_{1})  \ar@{=}[d] \\ 
& H_{*+k}(M_1\x M_2\x D^k,M_{1}^c) \ar[r]^-{[\bar\tau_{3}\,\cap]}  & H_{*-n}(\bar N_{3}) \ar[r] & H_{*-n}(M_{1}) }}$$
because $h$ is isotopic to the identity relative to $M_1\x M_2\x \del D^k$ (for the first triangle), pulls back $\bar \tau_{3}$ to $\tau_{3}$ (for the middle square), and fixes $M_1$ (for the last square). 
Hence Diagram (\ref{equ:intdiag}) commutes, which finishes the proof. 
\end{proof}

\section{Lifted intersection products}\label{sec:intprod}

In this section,  following \cite[Sec 4.1]{NRW} we lift the
intersection product along maps $p:\E\to M\x M$, and give examples that will be relevant to us. We will then show a first basic invariance property of lifted intersection products. 

\smallskip

Recall the Thom class $\tau_M\in C^n(U_M,M^c)\simeq C^n(M\x M,M^c)$ from the previous section. Pulling back along a map  $p:\E\to M\x M$ defines a class
$p^*\tau_M\in C^n(\E|_{U_M},(\E|_M)^c)$ that can be used to define a lift of the intersection product of $M$ using essentially the same definition:

\begin{Def}
  Given a map $p:\E\to M\x M$, we define the  {\em short intersection product} as the composition 
\begin{equation}\label{eq:int2}
 int_{U_M}: C_*(\E)\rar C_*(\E,\E|_{M^c})  \xrightarrow{\rho}  C_*(\E|_{U_M},\E|_{M^c}) \xrightarrow{p^*\tau_M\cap} C_*(\E|_{U_M}). 
\end{equation}
 where $\rho$ is a choice of homotopy inverse for excision. 
When $p$ is a fibration,  we get moreover an associated {\em intersection product} $int_M$ defined as the composition
\begin{equation}\label{eq:int}
  int_M: C_*(\E) \xrightarrow{int_{U_M}} C_*(\E|_{U_M}) \xrightarrow{\sim}  C_{*-n}(\E|_M)
  \end{equation}
 where the second map is induced by the equivalence $U_M\arsim M$ using that $p|_{U_M}$ is a fibration.
\end{Def}

Let $\La M=\textrm{Maps}(S^1,M)$ denote the free loop space. Evaluation at any point $t\in S^1$ defines a fibration $\e_t:\La M\to M$. This will be our main source of examples of fibrations.

\begin{ex}\label{ex:CS+trivial} Though not originally introduced this way, 
 a widely studied example of such a lifted intersection product is the {\em Chas-Sullivan product} on $H_*(\La M)$. This product can be obtained by applying the above construction to the fibration $\e_0\x \e_0: \La M\x \La M\to M\x M$, followed by loop concatenation:
 $$C_*(\La M\x\La M)\xrightarrow{int_M} C_{*-n}(\La M\x_M\La M) \xrightarrow{concat} C_{*-n}(\La M).$$
Here we have  identified $(\La M\x \La M)|_M$ with the figure eight space $\La M\x_M\La M$. (See \cite[Sec 2]{HinWah0} for an account of how this definition relates to the earlier definitions of \cite{CS99,CohJon}.)

 Another example is the so-called {\em trivial coproduct}, obtained in a similar fashion using the fibration $\e_{0,\ha}: \La M\to M\x M$, evaluating a loop at both $t=0$ and $t=\ha$, followed by the cut map:
$$C_*(\La M)\xrightarrow{int_M} C_{*-n}(\La M\x_M\La M) \xrightarrow{cut} C_{*-n}(\La M\x \La M).$$
Here again $\La M|_M$ identifies with the figure eight space $\La
M\x_M\La M$. Tamanoi showed in \cite[Thm B]{Tam}  that this coproduct is essentially trivial in homology, hence its name. 
\end{ex}

We give now two ``exotic'' examples of such intersection products that will be used in Section~\ref{sec:inv1}.

\begin{ex}\label{ex3} Recall from Section~\ref{sec:int} the tubular neighborhoods $N_2$ (resp.~$N_3$) of $M_1\x M_1$ (resp.~$M_1$)  in $M_1\x M_2\x D^k$: 
  $$\xymatrix{& N_3\ar@{^(->}[r] \ar[d]& N_2=M_1\x N(f,e)\ar@{^(->}[r]^-{\nu_2}\ar[d]_{p_{N_2}}&  M_1\x M_2\x D^k \\ M_1\ar@{^(->}[r] & U_{1}\ar@{^(->}[r] & M_1\x M_1 \ar@{^(->}[ur] &
   }$$
  where $N_3\cong N_2|_{U_{1}}$, see Proposition~\ref{prop:tub12}. 
  Given a map $p:\E_2\to M_2\x M_2$, consider the pull-backs
  $$\xymatrix{\E_{N_3}\ar[d]_{p}\ar[r] & \E_{N_2}\ar[d]_{p}\ar[r] & \E_{1,2}\x D^k\ar[d]^{p}\ar[r] & \E_{1,2}\ar[d]^{p}\ar[r] & \E_2 \ar[d]^p\\
 N_3\ar@{^(->}[r]&N_2\ar@{^(->}[r]&   M_1\x M_2\x D^k \ar@{->>}[r] & M_1\x M_2\ar[r]^-{f\x \id} & M_2\x M_2  }$$
Now the above construction associates to the composition $$\E_{N_2}\xrightarrow{\  p\ } N_2\xrightarrow{p_{N_2}} M_1\x M_1$$ 
a short intersection product of the form
$$int_{U_{1}}: C_*(\E_{N_2})  \rar 
C_{*-n}(\E_{N_2}|_{U_{1}}) \cong C_{*-n}(\E_{N_3}).$$ 
\end{ex}

\begin{ex}\label{ex3bis}
Consider now the tubular neighborhoods $\bar N_2$ (resp.~$\bar N_3$)  of $M_1\x M_2$ (resp.~$M_1$) in $M_1\x M_2\x D^k$, 
$$\xymatrix{& \bar N_3\ar@{^(->}[r] \ar[d]& \bar N_2=M_1\x M_2\x \RR^k \ar@{^(->}[r]^-{\bar \nu_2}\ar[d]^{p_{\bar N_2}}&  M_1\x M_2\x D^k \\
M_1\ar[d]_f\ar@{^(->}[r] &   U_{1,2}\ar@{^(->}[r] \ar[d] & M_1\x M_2 \ar@{^(->}[ur] \ar[d]^{f\x \id} & \\
 M_2\ar@{^(->}[r] &    U_{2}\ar@{^(->}[r] & M_2\x M_2  &  }$$
where $\bar N_3\cong \bar N_2|_{U_{1,2}}\cong \bar N_2|_{U_{2}}$, see Proposition~\ref{prop:tub12bis}. Just as above, we have pull-backs
 $$\xymatrix{\E_{\bar N_3}\ar[d]_{p}\ar[r] & \E_{\bar N_2}\ar[d]_{p}\ar[r] & \E_{1,2}\x D^k\ar[d]^{p}\ar[r] & \E_{1,2}\ar[d]^{p}\ar[r] & \E_2 \ar[d]^p\\
 \bar N_3\ar@{^(->}[r]& \bar N_2\ar@{^(->}[r]&   M_1\x M_2\x D^k \ar@{->>}[r] & M_1\x M_2\ar[r]^-{f\x \id} & M_2\x M_2  }$$
  The composition $$\E_{\bar N_2}\rar  \bar N_2\xrightarrow{p_{\bar N_2}} M_1\x M_2\xrightarrow{f\x\id} M_2\x M_2$$
  is not necessarily a fibration even if $p$ was one, but we still have a short  intersection product as in the previous example: 
  $$int_{U_{2}}: C_*(\E_{\bar N_2})\rar C_{*-n}(\E_{\bar N_2}|_{U_{2}})\cong C_{*-n}(\E_{\bar N_3}). $$
 \end{ex}

\subsection{Relative intersection products} 

Given a map $\E\to M\x M$ as above, and in addition a map $\R\to \E$,  
                                             we get a relative short intersection product
\begin{equation}\label{eq:int2rel1}
int_{U_M}:  C_*(\E,\R)\rar C_*(\E,\E|_{M^c}\cup \R)  \xrightarrow{\ \rho\ }  C_*(\E|_{U_M},\E|_{M^c}\cup \,\R|_{U_M}) \xrightarrow{p^*\tau_M\cap} C_*(\E|_{U_M},\R|_{U_M}).  
\end{equation}
This can further be composed with a retraction map 
$$C_*(\E|_{U_M},\R|_{U_M}) \xrightarrow{r}  C_{*-n}(\E|_M,\R|_M)$$
for example when both $\E|_{U_M}\to U_M$ and the composition $\R|_{U_M}\to \E|_{U_M}\to U_M$ are fibrations, giving 
 a relative intersection product 
\begin{equation}\label{eq:int2rel}
int_M: C_*(\E,\R)\xrightarrow{int_{U_M}} C_*(\E|_{U_M},\R|_{U_M}) \xrightarrow{r}  C_{*-n}(\E|_M,\R|_M).
  \end{equation}

  \begin{ex}[The relative coproduct]\label{ex:relcoprod}
    Let $\E=\La M$ and consider again the fibration $\e_{0,\ha}:\La M\to M\x M$ of Example~\ref{ex:CS+trivial}. Let
    $$\R=\big\{\ga\in \La M\ |\ \ga(t)=\ga(0)\ \forall t\in [0,\ha] \textrm{ or } \forall t\in [\ha,1]\big\}$$
    be the subspace of $\La M$ of half-constant loops. The map $\e_{0,\ha}$ restricts to a map
    $\e_{0,\ha}:\R\to M\subset M\x M$ with value on the diagonal. Then $\R|_{U_M}=\R|_{M}$ and we get a relative intersection product
    $$int_M: C_*(\La M,\R)\rar C_{*-n}(\La M|_{M},\R).$$
    Now $\La M|_{M}\cong \La M\x_M\La M$ identifies with the figure eight space and this intersection product can further be composed with a cut map to define a coproduct:
    $$C_*(\La M,\R)\xrightarrow{int_M} C_{*-n}(\La M|_{M},\R)\xrightarrow{cut} C_{*-n}(\La M\x \La M,M\x \La M\cup \La M\x M).$$
As already mentioned in Example~\ref{ex:CS+trivial}, the non-relative version of this coproduct is almost completely trivial in homology. On the other hand, precomposing with a reparametrization map yields the ``good'' string coproduct (see eg.~\cite[Sec 2.5]{NRW}) and it follows from the computations of Goresky-Hingston \cite[Thm 15.3]{GorHin} for spheres and Naef \cite{Nae21} for lens spaces,  that this relative version of the coproduct is highly non-trivial. 
  \end{ex}

\subsection{Invariance of intersection products over a fixed manifold}

Suppose that we are given a map, or relative map over $M\x M$ (taking $\R_i$ empty in the non-relative case):  
\begin{align}\label{eq:maprel}
                                                      \xymatrix@R=1pc@C=1pc{
 \R\ar@/_4ex/[dddrrr]\ar[dr] \ar@{-}[rr]^-{F_0} & & \R' \ar[dr]\ar@{..>}[dddr] & \\
&  \E\ar[rr]^-(.4)F \ar[ddrr]^(.55){p} & & \E'\ar[dd]^(.4){p'} \\
&&&\\
  & & & M\x M 
}
                                             \end{align}
We give here an enhanced version of \cite[Prop 4.6]{NRW}, showing that such maps respects the (short) intersection products over $M$.

\begin{prop}\label{prop:easynat}
  Let $(F,F_0):(\E,\R)\to (\E',\R')$ be a map  of spaces over $M\x M$ as in Diagram \eqref{eq:maprel} (with $\R,\R'$ possibly empty). Then there is a chain homotopy 
$$int_{U_{M}}\circ (F,F_0)_* \simeq (F,F_0)_*\circ int_{U_{M}}: C_*(\E,\R)\rar C_{*-n}(\E'|_{U_{M}},\R'|_{U_{M}}).$$
 If the restrictions $p|_{U_{M}}:\E|_{U_{M}} \to U_M$ and   $p'|_{U_M}:\E'|_{U_M} \to U_M$  are fibrations (and also for $\R$, $\R'$ in the relative case), then the same holds for the intersection product $int_M$. 
  \end{prop}

  \begin{proof}
We consider first the short intersection product.
We need to check that the following diagram commutes: 
\begin{equation*}
  \xymatrix{
    C_*(\E,\R)\ar[r] \ar[d]_F & C_*(\E,\E|_{M^c}\cup \R) \ar[d]^F & \ar[l]_-{\simeq}  C_*(\E|_{U_M},\E|_{M^c}\cup \R|_{U_M})\ar[d]^F \ar[rr]^-{p^*\tau_M\cap} && C_*(\E|_{U_M},\R|_{U_M}) \ar[d]^F\\
     C_*(\E',\R')\ar[r]  & C_*(\E',\E'|_{M^c}\cup\R') & \ar[l]_-{\simeq}  C_*(\E'|_{U_M},\E'|_{M^c}\cup \R'|_{U_M})\ar[rr]^-{(p')^*\tau_M\cap} && C_*(\E'|_{U_M},\R'|_{U_M}) }
 \end{equation*}
 The first square commutes already on the space level, the second by naturality of the excision isomorphism and the last by naturality of the cap product as $F^*(p')^*\tau_M=p^*\tau_M$ since we assumed that $p'\circ F=p$. The first part of the result follows after picking a homotopy inverse to excision via small simplices. 

  When the maps are moreover fibrations over $U_M$, we use the equivalence  $U_M\xleftarrow{\sim} M$ to get a further  diagram 
$$\xymatrix{
C_*(\E|_{U_M},\R|_{U_M}) \ar[d]^F  & \ar[l]_-{\simeq}  C_{*-n}(\E|_M,\R|_{M}) \ar[d]^F\\
 C_*(\E'|_{U_M},\R'|_{U_M})  & \ar[l]_-{\simeq}  C_{*-n}(\E'|_M,\R'|_{M}) }$$
 that commutes as it commutes already on the space level. The fact that $(F,F_0$ respects the full intersection product $int_M$ up to homotopy follows. 
    \end{proof}

\section{Invariance of intersection products along a homotopy equivalence}\label{sec:inv1}

Let  $f:M_1\to M_2$ be a smooth map  and
suppose now that we have fibrations $p_1:\E_1\to M_1\x M_1$ and  $p_2:\E_2\to M_2\x M_2$ together with a map $F:\E_1\to \E_2$ over $f\x f$. Such a map always factors as
\begin{align}\label{eq:fibmap}
\xymatrix{\E_1 \ar[r]^-{} \ar[d]_{p_1} & \E_{11,2}:= (f\x f)^*\E_2 \ar[r]^-{} \ar[]!<-7ex,-2ex>;[dl]!<2ex,0ex>  & \E_2\ar[d]^{p_2}\\
  M_1\x M_1\ar[rr]^-{f\x f} && M_2\x M_2.
                               }\end{align}

If we consider the pull-back $ \E_{11,2} $ as a space over $M_1\x M_1$,  Proposition~\ref{prop:easynat} already tells us that the map $\E_1\to  \E_{11,2} $ respects the intersection product $int_{M_1}$. 
If we instead consider it as a space over $M_2\x M_2$, the same result
tells us that the map $\E_{11,2} \to \E_2$ respects the intersection
product $int_{M_2}$. But we have so far no way of  comparing the
intersection products $int_{M_1}$ and $int_{M_2}$ on $ \E_{11,2}$. To
compare them, we will follow the pattern of argument given in the case of the homology intersection product in Section~\ref{sec:int}, working also with the intermediate spaces  obtained from pulling back $\E_2$ to $M_1\x M_2$ and $M_1\x M_2\x D^k$, as well as various normal bundles considered in that section, see also Examples~\ref{ex3} and~\ref{ex3bis}.

\medskip

We will in what follows assume that we have fibrations $\E_i\to M_i\x M_i$, and we will treat the following two types of relative situations:
\begin{align}\label{eq:maprel2}
\xymatrix@R=1pc@C=1pc{
 \R_1\ar[dddr]\ar[dr] \ar@{-}[rr]^-{F_0} & & \R_2 \ar[dr]\ar@{..>}[dddr] & &&  & \R_1\ar[dr]\ar[dd]\ar@{-}[rr]^-{F_0} & & \R_2 \ar[dr]\ar@{..>}[dd] &
\\
&  \E_1\ar[rr]^-(.4)F \ar[dd]^(.35){p_1} & & \E_2\ar[dd]^(.4){p_2}  &&&  &  \E_1\ar[rr]^-(.35)F \ar[dd]^(.35){p_1} & & \E_2\ar[dd]^(.4){p_2} 
\\
&&&&&  & M_{1}  \ar[dr]\ar@{..}[rr] & & M_{2}\ar[dr] &
\\
\textrm{(A)}   & M_1\x M_2\ar[rr]^-{f\x f}  & & M_2\x M_2 & & &  \textrm{(B)}\ \  & M_1\x M_2\ar[rr]^-{f\x f}  & & M_2\x M_2 
}
\end{align}                                                                                                             

\begin{assum}\label{ass:A} We have a  commuting diagram as in (\ref{eq:maprel2}\,A) with the maps $p_i:\E_i\to M_i\x M_i$ and  $\R_i\to \E_i\to M_i\x M_i$ for $i=1,2$ being fibrations.
  (The case $\R_i=\emptyset$ gives the non-relative case.) 
  \end{assum}

  \begin{assum}\label{ass:B} 
     We have a  commuting diagram as in (\ref{eq:maprel2}\,B) with the maps
$\E_i\to M_i\x M_i$  and  $\R_i\to M_{i}$ being fibrations. 
  \end{assum}

 In Section~\ref{sec:3.1} below, we  will  construct a diagram of the following shape under either assumption:
$$\xymatrix{C_*(\E_1,\R_1) \ar[rr]^-{int_{U_{1}}}  \ar@{<-->}[d]
  \ar@/_5.0pc/[d]_{(F,F_0)}
  \ar@{}@<-2.5pc>[d]
  &&
  C_{*-n}(\E_1|_{U_{1}},\R_1|_{U_{1}}) \ar[d]
    &  C_{*-n}(\E_1|_{{M_1}},\R_1|_{M_1})\ar[l]_{\sim} \ar[d]^{(F,F_0)} \\ 
C_*(\E_2,\R_2) \ar[rr]_-{int_{U_{2}}}   && C_{*-n}(\E_2|_{U_{2}},\R_2|_{U_{2}})   &  C_{*-n}(\E_2|_{{M_2}},\R_2|_{M_2})\ar[l]_{\sim} 
}$$
where the dashed vertical arrow is a zig-zag.
We will then use the above diagram to prove invariance of the intersection product under Assumption~\ref{ass:A} in Section~\ref{sec:assA}, and describe in Section~\ref{sec:assB} an obstruction to invariance when only Assumption~\ref{ass:B} holds.

\begin{rem}[Variants]\label{rem:R}
  We can always replace $(\E_1,\R_1)$ by $(\E_{11,2},\R_{11,2})$ using Diagram~\eqref{eq:fibmap},  and applying Proposition~\ref{prop:easynat} to compare the intersection products of $(\E_1,\R_1)$ by $(\E_{11,2},\R_{11,2})$.  This gives weaker, but less natural, assumptions on $(\E_1,\R_1)$.
\end{rem}

\subsection{Intermediate spaces}\label{sec:spaces}

To construct the above diagrams, we will use the pull-backs $\E_{N_2}$, $\E_{1,2}$  and $\E_{\bar N_2}$ of $\E_2$ that already appeared in Examples~\ref{ex3} and~\ref{ex3bis}: 
$$\xymatrix@R=1pc{\E_1 \ar[rr] \ar[rd] \ar[dd]   && \E_{1,2}\x D^k\ar[dd]\ar[rr] & &\E_{1,2}\ar[r] \ar[dd] \ar[ddr]& \E_2\ar[dd] \\
    &  \E_{N_2} \ar[dl]\ar@{-->}[dd] \ar@{^(->}[ur] & & \E_{\bar N_2}\ar@{-->}[dd] \ar[ur] \ar@{_(->}[ul] \ar[drr] && \\
  M_1\x M_1\ar[rr]  & & M_1\x M_2\x D^k \ar[rr] & & M_1\x M_2\ar[r] & M_2\x M_2\\
   & N_2 \ar@{-->}[ul] \ar@{^(->}[ur] & & \bar N_2\ar@{-->}[ur] \ar@{_(->}[ul]  & &
}$$
with a compatible diagram with spaces  $\R_{1,2}, \R_{N_2}$ and $\R_{\bar N_2}$ similarly as the pull-backs of $\R_2$ along the same maps. 

Note that under Assumption (B), the spaces  $\R_{1,2}, \R_{N_2}$ and $\R_{\bar N_2}$ will live over ``diagonal'' subspaces, namely the top layer in the following diagram: 
\begin{equation}\label{equ:intsp}
\xymatrix@R=1pc{ M_{1} \ar@{}[dd]^(.1){}="a"^(.95){}="b" \ar@{^(->}"a";"b" \ar[rrr]  && & M_{1}\!\x\! D^k\ar@{^(->}[dd]  \ar[rr] & & M_{1}  \ar@{}[dd]^(.1){}="a"^(.95){}="b" \ar@{^(->}"a";"b" \ar[r] & M_2 \ar@{}[dd]^(.1){}="a"^(.95){}="b" \ar@{^(->}"a";"b" \\  
  & & \boxed{N_2\cap M_{1}\!\x\! D^k}\ar@{^(->}[dd]  \ar@{-->}[ull] \ar@{^(->}[ur] & &  M_{1}\!\x\! \RR^k \ar@{^(->}[dd] \ar@{-->}[ur] \ar@{_(->}[ul]  & & \\
  M_1\!\x\! M_1\ar[rrr]  && & M_1\!\x\!M_2\!\x\! D^k \ar[rr] & & M_1\!\x\! M_2\ar[r] & M_2\!\x\! M_2\\
  & & N_2 \ar@{-->}[ull] \ar@{^(->}[ur] & & \bar N_2\ar@{-->}[ur] \ar@{_(->}[ul]  & &
}
\end{equation}

\medskip

The crucial difference between assumptions (A) and (B) is actually visible in this diagram already.  Indeed, if $f:M_1\to M_2$ is a homotopy equivalence, then all the spaces in the bottom part of the diagram are homotopic to the product $M_1\x M_1\simeq M_2\x M_2$, and almost all the spaces in the top part of the diagram are homotopic to their diagonal $M_1\simeq M_2$. The only outlier is the space $N_2\cap (M_1\x D^k)$ (in the box), which in general need not be homotopic to $M_1$.

\begin{lem}\label{lem:N2}
  When $f$ is a homotopy equivalence, the map of pairs $$(\nu_2,\id): (N_2,M_1)\inc (M_1\x M_2\x D^k,M_1)$$ is  a relative homotopy equivalence, where
  $M_1\inc N_2=M_1\x N(f,e) \inc M_1\x M_2\x D^k$ via the map $(\id,f,e)$. 
  \end{lem}

  \begin{proof}
 The map  $\nu_2$ is homotopic to the composition of homotopy equivalences
    $$N_2=M_1\x N(f,e)\xrightarrow{p_2} M_1\x M_1\xrightarrow{\id\x f} M_1\x M_2\arsim M_1\x M_2\x D^k.$$
Hence $\nu_2$ is a homotopy equivalence. Moreover $\nu_2$ restricts to
the identity on the diagonal $M_1$. Hence 
the map $(\nu_2,\id)$ is a map of pairs that is componentwise a homotopy equivalence. The result follows from the fact that the pairs are NDR pairs, applying e.g.~\cite[Chap 6, Sec 5]{MayConcise}. Explicitly, we can exhibit $(N_2,M_1)$, and hence also $(M_1\x M_2\x D^k)$, as an NDR-pair by letting $u:N_2\to I$ be defined by  
$$u(m,n,x)=\left\{\begin{array}{ll} \frac{|(f(m),e(m))-(n,x)|}{\eps_N} & 0\le |(f(m),e(m))-(n,x)|<\eps_N \\
                    1& \textrm{else}   \end{array}\right.$$ where $\eps_N$ is small enough so that $B_{2\eps_N}((f(m),e(m))\subset N(f,e)$ for every $m\in M_1$. Then $q:N_2\x I\to N_2$ is any deformation retraction that collapses these $\eps_N$-balls. 
  \end{proof}

Let $j:(\E_{N_2},\R_{N_2})\to (\E_{1,2}\x D^k,\R_{1,2}\x D^k)$ denote the map identifying the source as a pull-back of the target along the map $N_2\to M_1\x M_2\x D^k$. The above discussion has the following consequence:

    \begin{prop}\label{prop:j}
Suppose $f:M_1\to M_2$ is a homotopy equivalence.  Then the following holds.
 \begin{enumerate}[(i)]  
\item If Assumption~\ref{ass:A} holds, the map $j:(\E_{N_2},\R_{N_2}) \to (\E_{1,2}\x D^k, \R_{1,2}\x D^k)$ is a relative homotopy equivalence. 
\item If only Assumption~\ref{ass:B} holds, $j:(\E_{N_2},\R_{N_2}|_{M_1})\to (\E_{1,2}\x D^k,\R_{1,2}\x D^k|_{M_1})$ is a relative homotopy equivalence.
  \end{enumerate}
\end{prop}

\begin{proof}
The map $j$ fits in a pull-back diagram
$$\xymatrix{\E_{N_2}\ar[r]^-j \ar[d] & \E_{1,2}\x D^k \ar[r]\ar[d]^p & \E_2 \ar[d]^{p_2} \\
N_2\ar[r]^-{\nu_2} & M_1\x M_2\x D^k \ar[r] & M_2\x M_2
}$$
and similarly for the relative terms. 
Now $\nu_2$ is a (relative) homotopy equivalence by Lemma~\ref{lem:N2}. The result follows using that homotopy equivalences and closed cofibrations pull-back along fibrations. 
\end{proof}

\subsection{Comparing intersection products over different manifolds}\label{sec:3.1}

   Recall from Examples~\ref{ex3} and ~\ref{ex3bis} the intersection products $int_{U_{1}}$ on $\E_{N_2}$ and  $int_{U_{2}}$ on $\E_{\bar N_2}$.
The  goal of the section is to construct a homotopy commuting diagram: 
\begin{equation}\label{equ:copdiag2.0}
  { \xymatrixcolsep{3pc}\xymatrix{ C_*(\E_1,\R_1)  \ar[d]^{F_{N_2}} \ar@/_8.0pc/[dddd]_{(F,F_0)}
  \ar@{}@<-4.5pc>[dddd]|(0.64){\fbox{Prop~\ref{prop:lhs}}}  
    \ar[rrr]^-{int_{U_{1}}} && \ar@{}[d]|{\fbox{Prop~\ref{prop:easynat}}}& C_{*-n}(\E_1|_{U_{1}},\R_1|_{U_{1}}) \ar[d]^{F_{N_2}}  \ar@/^5.0pc/[dddd]^{(F,F_0)}
  \ar@{}@<2.5pc>[dddd]|{\fbox{Prop~\ref{prop:retract}}} \\ 
  C_*(\E_{N_2},\R_{N_2}) \ar[rrr]^(.45){int_{U_{1}}} & && C_{*-n}(\E_{N_3},\R_{N_3}) \ar[dd]_{\hat h}
       \ar@{}[ddll]_{\fbox{Prop~\ref{prop:inv1}}} 
    \\ 
  C_{*+k}(\E_{1,2}\!\x\!  D^k,  \E_{1,2}\x \del D^k\cup \R_{1,2}\x D^k) 
  \ar@<-4ex>[u]_-{\llbracket\overline p^*\taue\cap\rrbracket}  \ar@{<-}@/^2pc/@<1ex>[u]^{[\x D^k]\circ j} \ar@<0.7ex>@{}[u]|{\fbox{\ref{lem:diag5}}} 
    \ar[d]^-{\llbracket\overline p^*\etab\,\cap\rrbracket}_-\simeq 
 &&& \\
 C_*(\E_{\bar N_2},\R_{\bar N_2}) \ar[d]
\ar[rrr]^(.45){int_{U_{2}}} 
&&   \ar@{}[d]|{\fbox{Prop~\ref{prop:easynat}}} & C_{*-n}(\E_{\bar N_3},\R_{\bar N_3}) \ar[d]  \\ 
C_*(\E_2,\R_2) \ar[rrr]^-{int_{U_{2}}}   & & & C_{*-n}(\E_2|_{U_{2}},\R_2|_{U_{2}}) 
}}
\end{equation}
where we already know that the top and bottom squares homotopy commute as a direct application of Proposition~\ref{prop:easynat}. In this section, we will
define the map $\hat h$ from the map $h$ of Proposition~\ref{prop:tub3}, and show that the diagram commutes. Note that this last statement is a priori ambiguous for the bigone in the diagram. What will be needed to conclude equality of the two outer compositions in the diagram is that the map denoted $[\x D^k]\circ j$ is a right inverse  to $\llbracket\overline p^*\taue\cap\rrbracket$.  We will show in Lemma~\ref{lem:diag5} that $[\x D^k]\circ j$ is in fact a left inverse under both assumption (A) and (B), and an actual inverse under assumption (A). And this is where the obstruction of invariance comes from in case (B). 
Sections~\ref{sec:assA} and~\ref{sec:assB} will then deal with what we can conclude in each case separately.

We start by carefully defining the map $[\x D^k]\circ j$ and analysing the bigone. 
Let 
\begin{align*}
 \iota \colon & \big(\E_{1,2}\x  D^k,\E_{1,2}\x  \del D^k\cup \R_{1,2}\x D^k\big) \rar \big(\E_{1,2}\x D^k,(\E_{N_2})^c\cup \R_{1,2}\x D^k\big)\\
j \colon & (\E_{N_2},\R_{N_2}) \inc (\E_{1,2}\x D^k,\R_{1,2}\x D^k)
\end{align*}
denote the maps of pairs, and similarly for $\bar \iota$ and $\bar j$ where $N_2$ is replaced by $\bar N_2$.

\medskip

Recall that $\taue$ and $\bar\eta$ are Thom classes for the bundles $N_2$ and $\bar N_2$ respectively.

\begin{lem}\label{lem:diag5}
  If $f$ is a degree 1 map,   then the following diagram  commutes up to homotopy: 
\begin{equation}\label{equ:diag5E}
\xymatrix{C_*(\E_{N_2},\R_{N_2}) \ar[d]^{j} & & C_{*+k}(\E_{1,2}\x D^k,(\E_{N_2})^c\cup \R_{1,2}\x D^k) \ar@{-->}[ll]^-\simeq_-{[ p^*\taue\cap]} \\
  C_*(\E_{1,2}\x D^k,\R_{1,2}\x D^k)  & & C_{*+k}(\E_{1,2}\x D^k,\E_{1,2}\x \del D^k\cup \R_{1,2}\x D^k) \ar@{-->}[d]_\simeq^{\bar \iota} \ar@{-->}[u]_{\iota}   
  \ar[dll]^(0.45){\llbracket p^*\etab\cap\rrbracket}_(0.46)\simeq  \ar[ull]_(0.45){\llbracket p^*\taue\cap\rrbracket} \\
C_*(\E_{\bar N_2},\R_{\bar N_2}) \ar[u]_{\bar j}^\simeq & & C_{*+k}(\E_{1,2}\x D^k,(\E_{\bar N_2})^c\cup \R_{1,2}\x D^k) \ar@{-->}[ll]_-\simeq^-{[p^*\etab\cap]}\\
}
\end{equation}
where the dashed arrows are added for clarity; by definition the two triangles commute.
Moreover, 
the map $\llbracket p^*\taue\cap\rrbracket$ is homotopy invertible if and only if the map $j$ is, and if it is,
$$\llbracket p^*\taue\cap\rrbracket^{-1}\simeq \llbracket p^*\etab\cap\rrbracket^{-1}\circ \bar j^{-1}\circ j\simeq  [\x D^k]\circ j$$
for  $[\x D^k]: C_*(\E_{1,2}\x D^k,\R_{1,2}\x D^k)\simeq  C_*(\E_{1,2},\R_{1,2}) \xrightarrow{\x [D^k]}  C_{*+k}(\E_{1,2}\x D^k, \E_{1,2}\x \del D^k\cup \R_{1,2}\x D^k)$. 
\end{lem}

\begin{proof}
  Lemma~\ref{lem:degree} gives that $i^*[\taue]=[\eta]\in C^k(M_2\x D^k,M_2 \x \del D^k)$.
Using the commutativity of 
$$\xymatrix{ (\E_{1,2}\x  D^k,\E_{1,2}\x \del D^k) \ar[d]_{p}\ar[r]^-{\iota} & (\E_{1,2}\x D^k,(\E_{N_2})^c) \ar[d]^{p}\\
(M_1\x M_2\x D^k,M_1\x M_2\x \del D^k) \ar[r]^-{\id\x i}  & (M_1\x M_2\x D^k,N_2^c).
}$$
we can deduce that  $$\iota^*p^*[\taue]=p^*(\id\x i)^*[\taue]=p^*[1\x \eta]\simeq \bar\iota^*\bar \eta.$$
Hence capping with either class defines a map
$$C_*(\E_{1,2}\x D^k,\R_{1,2}\x D^k)  \lar C_{*+k}(\E_{1,2}\x D^k,\E_{1,2}\x \del D^k\cup \R_{1,2}\x D^k) $$
in the middel of the diagram making both the top and bottom resulting squares homotopy commute. 

The maps $\bar j$ and $\bar \iota$ are rescaling maps in the $\RR^k$ coordinates, and the top and bottom horizontal maps are equivalences by the Thom isomorphism.
Finally, the fact that the composition $\llbracket p^*\etab\cap\rrbracket^{-1}\circ \bar j^{-1}\simeq [\x D^k]$ follows from the fact that the homotopy inverse of the compatible map $[\eta \cap]$ is given by crossing with a disc.  
  \end{proof}

 The commutativity of the left hand side of Diagram~\eqref{equ:copdiag2.0} is given by the following: 

  \begin{prop}\label{prop:lhs}
The left hand side of  Diagram~\eqref{equ:copdiag2.0} commutes up to homotopy. 
    \end{prop}

    \begin{proof}
 Consider the diagram 
      $$\xymatrix{C_*(\E_1,\R_1) \ar[dd]_{(F,F_0)}\ar@/_1pc/[ddr] \ar[r]& C_*(\E_{N_2},\R_{N_2}) \ar[d]_j \ar[ddr]^-{[\x D^k]\circ j} & \\
     &   C_*(\E_{1,2}\x D^k, \R_{1,2}\x D^k) & \\ 
     C_*(\E_2,\R_2) &     C_*(\E_{\bar N_2},\R_{\bar N_2}) \ar[l] \ar[u]^{\bar j}_\cong 
     &   C_{*+k}(\E_{1,2}\x D^k, \R_{1,2}\x D^k\cup \E_{1,2}\x \del D^k) \ar[l]^-{[[p^*\bar\eta \cap]]}_-\simeq
       }$$
 The leftmost triangles commute on the space level, while the left part commutes by Lemma~\ref{lem:diag5}.
    \end{proof}

    We now turn to the map $\hat h$ and the right part of Diagram~\eqref{equ:copdiag2.0}.

    \smallskip

Recall from Proposition~\ref{prop:tub3} the diffeomorphism  $h: M_1\x M_2\x D^k\rar M_1\x M_2\x D^k$, isotopic to the identity,  that identifies the tubular neighborhoods $\nu_3N_3$ and $\bar\nu_3\bar N_3$ fixing their 0-section $M_1$, and that is the identity outside $U_{1,2}\x \mathring{D}^k$.

\begin{prop}\label{prop:hhat} Suppose that $(F,F_0):(\E_1,\R_1)\to (\E_2,\R_2)$ is as in Assumption (A) or (B).
    Then there is a pair of continuous maps $$(\hat h,\hat h_0)\colon (\E_{1,2}\x D^k,\R_{1,2}\x D^k) \rar  (\E_{1,2}\x D^k,\R_{1,2}\x D^k)$$ such that
\begin{enumerate}
\item The  map $p:\E_{1,2} \x D^k \to M_1\x M_2\x D^k$ intertwines the maps $\hat h$ and $h$, for $h$ the diffeomorphism of Proposition~\ref{prop:tub3}.
  Likewise, the map $\R_{1,2} \x D^k \to M_1\x M_2\x D^k$ 
  intertwines the maps $\hat h_0$ and $h$. 
  In particular, $\hat h(\E_{N_3})\subset \E_{\bar N_3}$ and  $\hat h((\E_{N_3})^c)\subset (\E_{\bar N_3})^c$ and $\hat h_0(\R_{N_3})\subset \R_{\bar N_3}$;
\item The map $\hat h$ fixes $\E_{1,2}\x D^k|_{M_1}$ and $\E_{1,2}\x D^k|_{(U_{1,2}\x \mathring D^k)^c}$; 
\item The pair $(\hat h,\hat h_0)$ is homotopic to the identity relative to $\E_{1,2}\x D^k|_{M_1}$ and $\E_{1,2}\x D^k|_{(U_{1,2}\x \mathring D^k)^c}$. 
\end{enumerate}
\end{prop}

\begin{proof} 
  From Proposition~\ref{prop:tub3} we have that $h=\id \x \,\tilde h:M_1\x M_2\x D^k\to M_1\x M_2\x D^k$ restricting to the identity outside $U_{1,2}\x \mathring D^k$ by (1) in the proposition. Write  $h(m,n,x)=(m,n',x')$. If $(m,n,x)\neq (m,n',x')$, which can only happen if they are in $U_{1,2}\x D^k$, let $\id\x \la$ in $U_{1,2}\x D^k$ be the geodesic between these two points, keeping the $M_1$-coordinate fixed. (The $M_2$ component of $\la$ is a straight line in the ball of radius $\eps$ around $f(m)$ in $M_2$, and its $D^k$ component is the straight line from $x$ to $x'$ in $D^k$.)
  Define 
  $$\underline h: U_{1,2}\x D^k\x I\to U_{1,2}\x D^k$$
  by $\underline h(m,n,x,t)=(m,\la(t))$, with $\la$ the above defined path. This is a continuous map, and it extends to a map $\underline h: M_1\x M_2\x D^k\x I\to M_1\x M_2\x D^k$ by defining it to be the projection outside $U_{1,2}\x D^k$. 

Recall that $h$ restricts to the identity on $M_1\xrightarrow{(\id,f,e)} M_1\x M_2\x D^k$ as well as outside $U_{1,2}\x \mathring D^k$. 
 Let $\underline p=p\x \id:\E_{1,2}\x D^k\x I\to M_1\x\M_2\x D^k\x I$  denote the pulled back fibration and consider the lifting problem
  $$\xymatrix{(\E_{1,2}\x D^k)\x\{0\} \cup \big( (\E_{1,2}\x D^k)|_{M_1}\cup (\E_{1,2}\x D^k|_{U_{1,2}^c\x D^k})\big)\x I \ar[r]^-{\id} \ar@{^(->}[d] & \E_{1,2}\x D^k\ar[d]^p\\
    (\E_{1,2}\x D^k)\x I \ar[r]^-{\underline{h}\circ \underline p} \ar@{-->}[ur]^{\widehat {\underline{h}\circ \underline p}} &M_1\x M_2\x D^k
  }$$
  and similarly for $\R_{1,2}\x D^k$,  restricting to the subspace $U_{1,2}\x D^k$ in case (B). 
  We define $\hat h$ by setting $\hat h(m,\ga,x)=\widehat {\underline{h}\circ \underline p}(m,\ga,x,1)$ after solving the resulting relative lifting problem.
All the desired properties then hold by construction. 
\end{proof}

We now show that the middle square in Diagram~\ref{equ:copdiag2.0} commutes.

\begin{prop}\label{prop:inv1} Suppose that $(F,F_0):(\E_1,\R_1)\to (\E_2,\R_2)$ is a map as in Assumption (A) or (B). 
  Then the middle square in Diagram~\eqref{equ:copdiag2.0} 
  $$\xymatrix{ C_*(\E_{N_2},\R_{N_2})\ar[d]_{int_{U_{1}}} & \ar@{-->}[dl]_{[[p^*\tau_3\cap]]}\ar@{-->}[dr]^{[[p^*\bar \tau_3\cap]]} C_{*+k}(\E_{1,2}\x D^k,\R_{1,2}\x D^k \cup \E_{1,2}\x \del D^k)\ar[l]_-{[[p^*\bar\tau_e\cap]]}  \ar[r]^-{[p^*\bar\eta\cap]}  & C_*(\E_{\bar N_2},\R_{\bar N_2}) \ar[d]^{int_{U_{2}}} \\
 C_{*-n}(\E_{N_3},\R_{N_3}) \ar[rr]^-{\hat h}    && C_{*-n}(\E_{\bar N_3},\R_{\bar N_3}) 
}$$
commutes up to chain homotopy. 
Moreover, the maps $[p^*\eta\cap]$ and $\hat h$ are homology isomorphisms. 
  \end{prop}

\begin{proof}  Recall from Propositions~\ref{prop:tub12} and~\ref{prop:tub12bis} that the classes  $\tau_3\in C^{n+k}(M_1\x M_2\x D^k,\nu_3(N_3)^c)$ and  $\bar\tau_3\in C^{n+k}(M_1\x M_2\x D^k,\nu_3(\bar N_3)^c)$ are Thom classes for the bundles $N_3$ and $\bar N_3$. 
  The left and right triangles in the diagrams are the lift to $\E_{1,2}\x D^k$  of the triangles occurring in Propositions~\ref{prop:tub12}(3)  and~\ref{prop:tub12bis}(3), with the left diagonal map capping with $p^*\tau_3$ and right diagonal map capping with $p^*\bar\tau_3$. The commutativity of the diagrams follows from Proposition~\ref{prop:comptub} using the relations  $\tau_{3}:=p_2^*\tau_{M_1}\cup \taue$ and $\bar\tau_{3}:=\bar p_2^*(f\x 1)^*\tau_{M_2}\cup \etab $. 
Finally the middle triangle commutes by Proposition~\ref{prop:hhat} (3), that gives in particular that $(\hat h,\hat h_0)$ is homotopic to the identity relative to $\E_{1,2}\x \del D^k$, and Proposition~\ref{prop:tub3}(4), that gives the compatibility between capping with $\tau_3$ and with $\bar \tau_3$. 
\end{proof}

Finally, the following result gives the compatibility between the right-hand vertical composition $[\hat h]$ induced by $\hat  h$  in Diagram~\ref{equ:copdiag2.0} and the map induced by $(F,F_0)$ on the fibrations restricted to the neighborhoods of the diagonals, or equivalently to the diagonal itself:

\begin{prop}\label{prop:retract}
  The diagram of pairs of spaces
  $$\xymatrix{(\E_1|_{U_{1}},\R_1|_{U_{1}})\ar[d]  && \ar@{_(->}[ll]_\simeq (\E_1|_{M_1},\R_1|_{M_1}) \ar[d] \ar@/^3.0pc/[ddd]^{(F,F_0)} \\
    (\E_{N_3},\R_{N_3})=(\E_{N_2}|_{U_{1}},\R_{N_2}|_{U_{1}})\ar[d]_{(\hat h,\hat h_0)}  && \ar@{_(->}[ll] \ar[dd] \ar@{_(->}[lld] \E_{N_2}|_{M_1} \\
   (\E_{\bar N_3},\R_{\bar N_3})=(\E_{\bar N_2}|_{U_{1}},\R_{\bar N_2}|_{U_{1}}) \ar[d]&& \\
(\E_2|_{U_{2}},\R_2|_{U_{2}}) && \ar@{_(->}[ll]_\simeq (\E_2|_{M_2},\R_2|_{M_2}) 
}$$
commutes. Moreover, under Assumptions~\ref{ass:A} and~\ref{ass:B}, the horizontal maps are homotopy equivalences. 
In particular, the right hand side of Diagram~\ref{equ:copdiag2.0} homotopy commutes. 
  \end{prop}

  \begin{proof}
    The first statement follows from the fact that $(\hat h,\hat h_0)$ is isotopic to the identity relative to $\E_{N_2}|_{M_1}=\E_{1,2}\x D^k|_{M_1}$   by Proposition~\ref{prop:hhat} (3), giving the commutativity of the middle triangle, and the second follows from the fact that $\E_i\to U_{i}$ are fibrations, and likewise for $\R_i\to U_i$ in case (A), and $M_i\to U_{i}$ are homotopy equivalences, for $i=1,2$. In case (B), we use instead that $\R_i|_{U_i}=\R_i|_{M_i}$.  
    \end{proof}

    \subsection{Invariance of intersection products under Assumption~\ref{ass:A} and proof of Theorem~\ref{thm:fibinv0}}
    \label{sec:assA}

In this section, we assemble the results we have proved so far to show homotopy invariance of intersection products under assumption (A), that is we prove Theorem~\ref{thm:fibinv0} from the introduction.  We then look at a few examples.

Theorem~\ref{thm:fibinv0} is closely related to \cite[Thm 4.11]{NRW}. We give here a complete proof for completeness.

\begin{proof}[Proof of Theorem~\ref{thm:fibinv0}]
Proposition~\ref{prop:j} (i) and (iii) gives that the map $$j:(\E_{N_2},\R_{N_2})\to (\E_{1,2}\x D^k,\R_{1,2}\x D^k)$$ is a relative homotopy equivalence. Hence by Lemma~\ref{lem:diag5}, $\llbracket p^*\taue\cap\rrbracket$  is homotopy invertible with $\llbracket p^*\taue\cap\rrbracket^{-1}\simeq [\x D^k]\circ j$ and  Diagram~\eqref{equ:copdiag2.0} gives a meaningful commutivative diagram in homology, once we apply Propositions~\ref{prop:easynat}, \ref{prop:lhs}, \ref{prop:inv1}, and \ref{prop:retract}, as indicated in the diagram. The result follows. 
   \end{proof}

\begin{cor}[Chas-Sullivan product and relative version]\label{cor:CSrel}
  Let $f:M_1\to M_2$ be a degree 1 homotopy equivalence. Then the induced map $F=\La f:\La_1\to \La_2$ respects the Chas-Sullivan product, also when considered relative to ``half-constant loops''. That is,  the diagram
  $$\xymatrix{H_*(\La_1\x \La_1,M_1\x \La_1\cup \La_1\x M_1) \ar[rr]^-{int_{M_1}}\ar[d] && H_{*_n}(\La_1\x_{M_1} \La_1, M_1\x_{M_1} \La_1\cup \La_1\x_{M_1}M_1 )\ar[d]  \\
    H_*(\La_2\x \La_2,M_2\x \La_2\cup \La_2\x M_2) \ar[rr]^-{int_{M_2}} && H_{*_n}(\La_2\x_{M_2} \La_2, M_2\x_{M_2} \La_2\cup \La_2\x_{M_2}M_2 ) }$$
and its non-relative version commute. Composed with the concatenation map, it gives the homotopy invariance of the Chas-Sullivan product, as first proved in \cite{CKS}. 
\end{cor}

\begin{proof}
  Apply Theorem~\ref{thm:fibinv0} to the fibrations $\e_0\x \e_0:\E_i=\La_i\x \La_i\to M_i\x M_i$ for $i=1,2$  and their restriction to $\R_i=M_i\x \La_i\cup \La_i\x M_i$, with the map $F=\La f\x \La f$. 
\end{proof}

\begin{ex}[Products for mapping spaces]\label{ex:pmap}
  Given spaces $X_1,X_2$ with chosen good basepoints $x_1,x_2$ (so the evaluation maps are fibrations), and a map $X_3\to X_1\cup_{x_1\sim x_2}X_2$, one can mimmic the definition of the Chas-Sullivan product and get an intersection product
  $$H_*(\textrm{Map}(X_1,M)\x \textrm{Map}(X_2,M)) \xrightarrow{int_M} H_{*-n}(\textrm{Map}(X_1,M)\x_M\textrm{Map}(X_2,M)) \rar H_{*-n}(\textrm{Map}(X_3,M))$$
  that is homotopy invariant by Theorem~\ref{thm:fibinv0}. For example, one could take $X_1=X_2=X_3=S^n$ with a chosen map
  $$X^3=S^n\to S^n/S^{n-1}\arsim S^n\vee S^n= X_1\cup_{x_1\sim x_2}X_2$$ induced by collapsing an equatorial sphere. 
  \end{ex}

\begin{ex}[Trivial coproduct]
  The homotopy invariance of the non-relative (and almost completely trivial) coproduct follows from the above theorem, using the fibrations $\e_{0,\frac{1}{2}}:\La_i\to M_i\x M_i$ with $F=\La f$.
\end{ex}

\begin{non-ex}[Good coproduct]
  Consider now the subspace $\R_i\subset \La_i$ of {\em half-constant loops}, i.e.~loops $\ga$ such that $\ga$ is constant on $[0,\ha]$ or $[\ha,1]$. 
 The coproduct of the previous example on the pair $(\La_i,\R_i)$ does not satisfy the assumptions of the theorem anymore, since in this case the restriction $\e_{0,\frac{1}{2}}:\R_i\to M_i\x M_i$ is only a fibration over the diagonal.
\end{non-ex}

\begin{non-ex}[$S^1$-coproduct]
  Another example that does not satisfy the assumptions of  Theorem~\ref{thm:fibinv0} is the $S^1$--parametrized coproduct associated to the map $$\e_{0,t}:\La_i\x S^1\to M_i\x M_i$$ evaluating the loop at time 0 and $t\in S^1$. Indeed, this map is not a fibration over $M_i\x M_i$. 
\end{non-ex}

 \subsection{Quantifying the failure of invariance of relative intersection products under Assumption~\ref{ass:B}}\label{sec:assB}

 Suppose that the map $(F,F_0):(\E_1,\R_1)\to (\E_2,\R_2)$ only satisfies the weaker 
 Assumption~\ref{ass:B}, that is $\E_i\to M_i\x M_i$ are fibrations but $\R_i\to M_i$ are only fibrations over the diagonal $M_{i}\subset M_i\x M_i$. 
 Forgetting about the relative parts $\R_1$ and $\R_2$,  Theorem~\ref{thm:fibinv0} gives that the map $F: C_*(\E_1)\to C_*(\E_2)$ respects the intersection products up to chain homotopy when $f$ is a degree 1 homotopy equivalence.
 This can be used to obtain an obstruction of invariance of the relative map $(F,F_0)$ using the following algebraic lemma:

 \begin{lem}\label{lem:alg}
   Suppose that $\phi,\psi:(C_*,D_*)\to (E_*,F_*)$ are relative chain maps that are non-relatively  chain homotopic: $\phi\simeq \psi:C_*\to E_*$. Let $H:C_*\to E_{*+1}$ be a chain homotopy between $\phi$ and $\psi$. Then  for a relative cycle $a\in C_*$ with $\del a\in D_{*-1}$, we have that
   $$\phi(a)-\psi(a)=H(\del a)\in H_*(E_*,F_*).$$ 
 \end{lem}

 \medskip

 We want to apply the lemma to the maps $\phi=F\circ int_{M_1}$ and $\psi=int_{M_2}\circ F$ from $(C_*,D_*)=\big(C_*(\E_1), C_*(\R_1)\big)$ to $(E_*,F_*)=\big(C_*(\E_2),C_*(\R_2)\big)$, with the homotopy $H$ coming from the homotopy commutativity of Diagram~\eqref{equ:copdiag2.0}, as given by Theorem~\ref{thm:fibinv0} in the non-relative case.

 As we have seen, a key difference between Assumptions~\ref{ass:A} and~\ref{ass:B}, given in Proposition~\ref{prop:j}, is whether or not the map
 $j:(\E_{N_2},\R_{N_2})\to (\E_{1,2}\x D^k,\R_{1,2}\x D^k)\simeq (\E_{1,2},\R_{1,2})$  is a relative homotopy equivalence. By Lemma~\ref{lem:diag5}, this determines whether or not the map $[\x D^k]\circ j$ in Diagram~\eqref{equ:copdiag2.0} is a right homotopy inverse to the map $\llbracket\overline p^*\taue\cap\rrbracket$. 
Let 
$$\alpha: H_*(\E_{N_2},\R_{N_2})\xrightarrow{j} H_*(\E_{1,2},\R_{1,2})\xrightarrow{ [\x D^k]} H_*(\E_{1,2}\x D^k,\R_{1,2}\x D^k) \xrightarrow{[[p^*\bar \tau_e\cap]]}   H_*(\E_{N_2},\R_{N_2}).$$
By Proposition~\ref{prop:j} and Lemma~\ref{lem:diag5}, its restriction 
$\alpha': H_*(\E_{N_2},\R_{N_2}|_{M_1})\rar  H_*(\E_{N_2},\R_{N_2}|_{M_1})$
is homotopic to the identity, but under Assumption~\ref{ass:B}, we do not have in general that $\al$ itself is homotopic to the identity. 
In the next result, we will express the failure of commutativity of Diagram~\eqref{equ:copdiag2.0} in terms of the map $\al$ and the tightly related failure of the map $j$ to be a relative homotopy equivalence.

\medskip

Let 
$L_{N_2}:\E_{N_2}\x I\to \E_{N_2}$
be a homotopy between the identity and $j^{-1}\circ j: H_*(\E_{N_2},\R_{N_2}|_{M_1}) \to H_*(\E_{N_2},\R_{N_2}|_{M_1})$, as given by Proposition~\ref{prop:j}. It induces a map 
$$L_{N_2}(-\x I): H_*(\R_{N_2})\rar  H_{*+1}(\E_{N_2},\R_{N_2})$$
 measuring the failure of $L_{N_2}$ to be a homotopy equivalence relative to $\R_{N_2}$.

\begin{thm}\label{thm:Hcommute0} 
   Let $F:(\E_1,\R_1)\to (\E_2,\R_2)$ be as in Assumption~\ref{ass:B}. The difference $F\circ int_{M_1}- int_{M_2}\circ F$ in homology is given by the composition
  \begin{multline*}H_*(\E_1,\R_1)\xrightarrow{\ \del\ } H_{*-1}(\R_1) \xrightarrow{\,F_{N_2}\,} H_{*-1}(\R_{N_2})\xrightarrow{\ \al\ } H_{*-1}(\R_{N_2})\xrightarrow{L_{N_2}(-\x I)} H_{*}(\E_{N_2},\R_{N_2})  \\
    \xrightarrow{int_{U_{1}}} H_{*-n}(\E_{N_3},\R_{N_3}) \xrightarrow{\ \hat h\ } C_{*-n}(\E_{\bar N_3},\R_{\bar N_3}) \xrightarrow{\ \ \ } C_{*-n}(\E_2|_{{M_2}},\R_2|_{{M_2}}).
      \end{multline*}
      That is for $A\in H_*(\E_1,\R_1)$, we have 
$$\big(F\circ int_{M_1}- int_{M_2}\circ F\big) (A)=int_{U_1}(L_{N_2}(\al(\del_{N_2}A)\x I)),$$
where $\del_{N_2}A$ denotes the image of  $\del A$ in $\R_{N_2}$ and where  we have suppressed the last two maps from the notation. 
\end{thm}

\begin{proof}
  We start as in the proof of Theorem~\ref{thm:fibinv0} with Diagram~\eqref{equ:copdiag2.0}, where we use the map $[\x D^k]\circ  j$ and remove the map $\llbracket\overline p^*\taue\cap\rrbracket$.
  With this choice, the middle pentagon does not commute anymore. More precisely, from Proposition~\ref{prop:inv1}, we see that what does not commute anymore is a triangle
   $$\xymatrix{C_{*}(\E_{N_2},R_{N_2})\ar[dr]_(.4){[\x D^k]\circ j\ \ } \ar[rr]^-{int_{U_{1}}}  & & C_{*-n}(\E_{N_3},R_{N_3}) \\
    & C_{*+k}(\E_{1,2}\x D^k,\E_{1,2}\x \del D^k \cup \R_{1,2}\x D^k)  \ar@{}[u]|(.6){\mbox{}} \ar[ur]_(.6){[[p^*\tau_3\cap]]}& 
  }$$
  where we know that the triangle where we replace $[\x D^k]\circ j$ by its non-relative homotopy inverse commutes, i.e.
  $$int_{U_1}\circ \llbracket\overline p^*\taue\cap\rrbracket\ \simeq_{H_1}\, [[p^*\tau_3\cap]]$$
for some homotopy $H_1$, which we know exists by Proposition~\ref{prop:inv1}, respecting the relative terms. We thus have
\begin{align*}   [[p^*\tau_3\cap]] \circ [\x D^k]\circ  j\simeq_{H_1} {int_{U_{1}}}\circ \llbracket\overline p^*\taue\cap\rrbracket \circ [\x D^k]\circ  j= {int_{U_{1}}}\circ \al
    \end{align*}
   by definition of $\al$. So we can obtain a homotopy for the commutativity of the above triangle (non-relatively) from a homotopy witnessing that the map $\al$
   is (non-relatively) homotopic to the identity. We can get such a homotopy in terms of the homotopy $L_{N_2}$ as follows:
\begin{align*}
   \al\ \simeq_{L_{N_2}\circ \al} \ j^{-1}\circ j\circ\al & =\ j^{-1}\circ j\circ [[p^*\bar \tau_e\cap]]\circ [\x D^k]\circ j\\
                                                                    &\simeq \ j^{-1}\circ\bar j \circ [[p^*\bar \eta\cap]]\circ [\x D^k]\circ j\\
        &\simeq \ j^{-1}\circ\bar j \circ \bar j^{-1}\circ j\ \simeq\  j^{-1}\circ j\  \simeq_{\overline{L_{N_2}}}\ \id \\
\end{align*}
where $\overline{L_{N_2}}$ is the homotopy $L_{N_2}$ in reverse. So in relative homology, the homotopy between $\al$ and the identity is the sum $H=(L_{N_2}(\al(-)\x I) + \overline{L_{N_2}}(-\x I)$.

Concretely, the homotopy $H$ gives a homotopy for the commutativity of middle pentagon in Diagram~\eqref{equ:copdiag2.0} as follows:
\begin{align*}
  \hat h\circ int_{U_1}\ \simeq_{\hat h\circ int_{U_1}\circ \overline H}\ \hat h\circ int_{U_1}\circ \al\ &=\  \hat h\circ int_{U_1}\circ [[p^*\bar \tau_e\cap]]\circ [\x D^k]\circ j \\
  & \simeq\ int_{U_2}\circ  [[p^*\bar \eta\cap]]\circ [\x D^k]\circ j
  \end{align*}
where ${\overline H}=(\overline{L_{N_2}}(\al(-)\x I)) + L_{N_2}(-\x I)$ is $H$ in reverse and where the last homotopy is given by Proposition~\ref{prop:inv1}, and respects the relative terms.

It follows then from Lemma~\ref{lem:alg} that the failure of invariance on a class $A\in H_*(\E_1,\R_1)$ is computed by the resulting homotopy applied to the image $\del_{N_2}A$ of its boundary $\del A$ in $H_{*-1}(\R_{N_2})$, followed by the forgetful map $u: C_{*-n}(\E_{\bar N_3},\R_{\bar N_3}) \rar C_{*-n}(\E_2|_{{M_2}},\R_2|_{{M_2}})$:
$$u\circ (\hat h\circ int_{U_1}\circ\overline H(\del_{N_2} A))=u\circ \hat h\circ int_{M_1}({\overline L_{N_2}}(\al(\del_{N_2} A)\x I))+u\circ \hat h\circ int_{M_1}(L_{N_2}(\del_{N_2}A\x I)),$$
where only the first term is non-trivial, since $\del_{N_2} A$ lies inside $\R_{N_2}|_{M_1}$, which is preserved by $L_{N_2}$.  
\end{proof}

\section{Mapping spaces}\label{sec:map}

In this section, we give an explicit construction of the homotopy equivalence $L_{N_2}$ appearing in Theorem~\ref{thm:Hcommute0}, satisfying some properties that will be useful in the following section to give a proof of  Theorems~\ref{thm:formula} and~\ref{thm:formula2}. We restrict our attention to certain pairs $(\E,\R)$ coming from mapping spaces. We did not attempt to formulate the general set-up in which the construction makes sense, by lack of other motivating examples and fear of making the exposition less clear. The construction will make use of ``higher homotopy data'' for the map $f$.
We start by describing the mapping spaces considered. 

\medskip

Let $S$ be a topological space, and $s_0,s_1\in S$ two points such that $(S,\{s_0,s_1\})$ is an NDR-pair. This means in particular that $S$ comes equipped with a map $u$ as in Figure~\ref{fig:u}. 
  \begin{figure}[h]
    \centering
 \begin{lpic}{Mapu1(0.4,0.4)}
  \lbl[b]{75,20;$\xrightarrow{\ u\ }$}
  \end{lpic}
\caption{The map  $u:S\cong S\x \{1\} \inc S\x I \rar S\x\{0\}\cup \{s_0,s_1\}\x I$, \ $u(s_i)=(s_i,1)$}\label{fig:u}
\end{figure}
We will here consider evaluation maps on mapping spaces of the form  
$$\E=\maps(S,M)\xrightarrow{\e_{s_0,s_1}} M\x M.$$
The map $\e_{s_0,s_1}$ can be seen to be a fibration using the map $u$ just defined.
We will denote $$\E^0:=\{\ga\in \maps(S,M)\ |\ \ga \textrm{ is constant}\}\cong M$$ the subspace of constant maps.

\smallskip

We will consider subspaces of such mapping spaces of the following form: let $P=\{P_a\}_{a\in A}$ be a collection of subspaces of $S$, all containing $s_0$ and $s_1$ and with each $P_a\inc S$ a cofibration. Define 
\begin{align*}\R&:=\cup_{a\in A}\{\ga\in \maps(S,M)\ |\ \ga(s)=\ga(s_0) \ \forall s \in P_a\}
\end{align*}
Then $\e_{s_0,s_1}|_{\R}$ is a fibration over the diagonal in $M\x M$. 

\smallskip

Our main examples will be of the following form: let $S=I^r/\del I^r$ with $s_0={\bf 0}$ and $s_1={\bf c}=(\ha,\dots,\ha)$, and with $P=\{I^{j-1}\x[0,\ha]\x I^{r-j}\}_{j=1}^r\cup \{I^{j-1}\x[\ha,1]\x I^{r-j}\}_{j=1}^r$. Then $\E=\La^rM$ is the $r$-loop space and 
$$\R=\cup_{j=1}^r\{\ga:(I^r/\del I^r)\to M\ |\ \ga|_{I^{j-1}\x[0,\ha]\x I^{r-j}} \ \textrm{or}\  \ga|_{I^{j-1}\x[\ha,1]\x I^{r-j}}\ \textrm{is constant}\} \subset \La^rM,$$
is the subspace of $r$-loops  that are half-constant in some direction $j\in \{1,\dots,r\}$.

\subsection{Higher homotopy data}\label{sec:K}

Let $f:M_1\to M_2$ be a degree 1 homotopy equivalence. Given a tuple $(S,s_0,s_1,P)$ as above, we let $\E_i=\maps(S,M_i)$, with their respective subspaces $\R_i$.  Then $f$ induces a map of pairs of fibrations $(\E_1,\R_1)\to (\E_2,\R_2)$ over $f\x f$ as in Assumption~\ref{ass:B}. 
Using the same notations as in Section~\ref{sec:inv1}, we want to analyse the homotopy equivalence $$j:\E_{N_2}\to \E_{1,2}\x D^k,$$ and most particularly its failure to be an equivalence relative to the subspaces $\R_{N_2}$ and $\R_{1,2}\x D^k$.

\medskip

Recall from Lemma~\ref{lem:N2} that $j$ is a homotopy equivalence (relative to $\R_{N_2}|_{M_1}$ and $\R_{1,2}\x D^k|_{M_1}$) because the map of pairs  $\nu_2: (N_2,M_1)\inc (M_1\x M_2\x D^k,M_1)$ is a relative homotopy equivalence. Let $g$ be a homotopy inverse to $f$. 
An explicit homotopy inverse $\nu_2^{-1}:M_1\x M_2\x D^k\to N_2$ that respects the diagonal $M_1$ is given by setting 
$$\nu_2^{-1}(m,n,x)=\big(g\circ f(m),f\circ g(n),e\circ g(n)\big).$$ Indeed, the  map lands in $N_2$, in fact in its zero section, by the choice of second and third coordinates, and one checks that, with this choice on the first coordinate, it takes the diagonal $M_1$ to itself.

\smallskip

To produce the homotopies that witness that $\nu_2^{-1}$ is indeed a homotopy inverse requires some care, since on the first coordinate one would want to use a homotopy  $h_1:M_1\x I\to M_1$ witnessing that  $g\circ f\simeq \id$, while on the second coordinate, one needs to rather use a homotopy $h_2:M_2\times I\to M_2$ witnessing that $f\circ g\simeq \id$, and these two choices are not compatible with the diagonal, since $f\circ h_1$ is typically not equal to $h_2\circ (f\x \id)$.

Given a homotopy equivalence $f:M_1\to M_2$, it is actually always possible to choose the maps $g,h_1,h_2$ such that there exists  a further homotopy 
$K:M_1\x I\x I\to M_2$ between $h_2\circ (f\x \id)$ and $f\circ h_1$ relative to $M_1\x \del I$, where both maps agree:
$$K(m,s,t)=\left\{\begin{array}{ll}f(m) & s=0 \\
                    f\circ g\circ f(m) & s=1\\
                    h_2(f(m),s) & t=0 \\
                    f\circ h_1(m,s) & t=1
                    \end{array}\right.$$
                  This is possible by Vogt's lemma \cite{Vogt}, and, given the map $f$, a contractible choice by e.g.; Riehl-Verity \cite[Prop 4.4.7]{RV16}.
  Such a map $K$ can be used to show that $\nu_2^{-1}$ is indeed a homotopy inverse to $\nu_2$. 
   
One can think of $K$ as a continuous family of path connecting $h_2(f(m),s)$ and $f(h_1(m,s))$ for each $(m,s)$, 
starting at the constant path at $f(m)$ when $s=0$ and ending at the constant path at $f\circ g\circ f(m)$ when $s=1$. 
To $K$ one can also associate a family of ``push-maps'' pushing $h_2(f(m),s)$ along $K(m,s,-)$ up to time $t$, as given in the following statement. (See also Figure~\ref{fig:vmt}.) 

  \begin{lem}\label{lem:Kdef}
    The homotopy data $(f,g,h_1,h_2,K)$ defines a continuous family of maps 
    $$\Psi\!_{K}: M_1\x I\x I\to \maps(M_2,M_2); \ \ \
    (m,s,t)\mapsto \Psi_K^t(m,s)$$
 with the following properties:
 \begin{enumerate}[(i)]
 \item $\Psi_K^0(m,s)= \Psi_K^t(m,0)=\Psi_K^t(m,1)=\id_{M_2}$; 
 \item $\Psi_K^t(m,s)(h_2(f(m),s))=K(m,s,t)$.
 \end{enumerate}
 In particular, $\Psi_K^1(m,s)(h_2(f(m),s))=f\circ h_1(m,t)$.
\end{lem}

In other words, for every $(m,s)\in M_1\x I$, we have  $\Psi_K^0(m,s)=\id_{M_2}$ while $\Psi^1_K$ is a map with the property that $\Psi_K^1(m,s)(h_2(f(m),s))=f(h_1(m,s))$, and $\Psi_K^{t}(m,s)$, $t\in [0,1]$, is a homotopy between these two maps that follows $K$. 
Note that the two conditions in the statement are compatible since $K(m,s,0)=h_2(f(m),s)$, while $K(m,0,t)=f(m)=h_2(f(m),0)$ and $K(m,1,t)=f\circ g\circ f(m)=h_2(f(m),1)$ for all $t$.

    \begin{proof}
      Let $\eps>0$ be smaller than the injectivity radius as before, and define $\eps_s=\min(s\eps,(1-s)\eps)$ so that $\eps_0=0=\eps_1$ and $0<\eps_s<\frac{\eps}{2}$ for all $s\in (0,1)$.     We define $\Psi_K^t(m,s)$ to be the identity outside a ball of radius $2t\eps_s$ around $h_2(f(m),s)$. Note in particular that such a map $\Psi_K$ will automatically satisfy condition (i) in the statement since $2t\eps_s=0$ when $s=0,1$ or $t=0$.
      
      Inside the ball, we define $\Psi^t_K$ as follows: for $V\in UTM_{h_2(f(m),s)}$ a unit tangent vector and $s\in (0,1)$, define
      $$\Psi_K^t(m,s)(\exp_{h_2(f(m),t)}(rV))=\left\{\begin{array}{ll} K(m,s,t-\frac{r}{\eps_s}) &  0\le r\le t\eps_s \\
                                         \exp_{h_2(f(m),s)}(2(r-t\eps_s)V)          &  t\eps_s\le r\le 2t\eps_s 
                                                \end{array}\right.$$
                                              (See also Figure~\ref{fig:vmt}.)   That is the map $\Psi_K^t(m,s)$ takes each ray in the ball of radius $t\eps_s$ around $h_2(f(m),s)$ to the same path following $K$ from $K(m,s,0)=h_2(f(m),s)$ (when $r=t\eps_s$) to $K(m,s,t)$ (when $r=0$) in the time interval $[0,t]$. In particular, it satisfies condition (ii). 
                                                 The points at distance $r=t\eps_s$ are taken to $h_2(f(m),s)$ and   
     the annulus of points at distance $t\eps_s$ to $2t\eps_s$ is then used to interpolate between this map shrinking the middle sphere to the center of the ball and the identity outside the ball.

                                              Continuity at $s=0,1$ follows from the fact  that $K(m,s,-)$  is a continuous family of paths starting at the constant path at $h_2(f(m),0)=f(m)$ when $s=0$ and ending at the constant path at $h_2(f(m),1)=f\circ g\circ f(m)$ when $s=1$. 
                                                   \end{proof}
        \begin{figure}[ht]
                    \begin{lpic}{Disk2K(0.4,0.4)}
                      \lbl[l]{318,50;$f\!\circ\! h_1(m,s)$}
                      \lbl[bl]{49,47;$h_2(f(m),s)$}
                      \lbl[b]{203,48;$h_2(f(m),s)$}
                       \lbl[b]{125,49;$\Psi_K^{-}(m,s)$}
                        \lbl[l]{280,10;$K(m,s,-)$}
                      \end{lpic}
                    \caption{Deformation map $\Psi^{-}_K(m,t):M_2\to M_2$}\label{fig:vmt}
                    \end{figure}

\subsection{Construction of the main homotopy in the case of mapping spaces}\label{sec:invert}

Let  $\E_{N_2}^0$ and $\E_{1,2}^0$ denote the subspaces of constant loops in $\E_{N_2}$ and $\E_{1,2}$.  Note that
$$\E_{N_2}^0\cong N_2\cap (M_{1}\x D^k)=\{(m,n,x)\in M_1\x M_2\x D^k\ |\ f(m)=n \ \textrm{and}\ (n,x)\in N(f,e)\}$$
identifies with the problematic space appearing in Diagram~\ref{equ:intsp}. 
The choice of homotopy inverse to $\nu_2$ given at the start of Section~\ref{sec:K} naturally lifts to define a map 
$j^{-1}:\E_{1,2}\x D^k\rar \E_{N_2}$ by setting
$$j^{-1}(m,\ga,x)\ =\ \big(g\circ f(m),f\circ g\circ \ga,e\circ g(\ga(s_1))\big).$$
This map has the property that it takes constant loops to constant loops. 
We will see below that $j^{-1}$ is indeed a homotopy inverse to the map $j:\E_{N_2}\to \E_{1,2}\x D^k$, but the homotopies witnessing that will in general not preserve constant loops, as is to be expected by the above comments, since $\E_{N_2}^0$ is in general not homotopic to $\E_{1,2}^0\x D^k$.

\begin{prop}\label{prop:L12} Let $f:M_1\to M_2$ be a homotopy equivalence and 
 $(\E_2=\maps(S,M_2),\R_2)$ be a mapping space associated to a triple $(S,\{s_0,s_1\},P)$ as above.
Consider  the associated homotopy equivalence $j: \big(\E_{N_2},\R_{N_2}|_{M_1}\big) \inc \big(\E_{1,2}\x D^k,(\R_{1,2}\x D^k)|_{M_1}\big)$ as in Proposition~\ref{prop:j}. 

Let $g:M_2\to M_1$ be a homotopy inverse for $f$. Then  map $j^{-1}: \E_{1,2}\x D^k\rar \E_{N_2}$  defined by $$j^{-1}(m,\ga,x)=(g\circ f(m),f\circ g\circ \ga,e\circ g(\ga(s_1)))$$
 is a relative homotopy inverse for $j$ that  preserves all constant and partially constant maps: 
  \begin{enumerate}
  \item[(i)]  $j^{-1}(\E_{1,2}^0\x D^k)\subset \E_{N_2}^0$ and $j^{-1}(\R_{1,2}\x D^k)\subset \R_{N_2}$. 
   \end{enumerate}
Let $h_1:M_1\x I\to M_1$ be a homotopy from the identity to $g\circ f$. 
There are homotopies $L_{1,2}:\E_{1,2}\x D^k\x I\to \E_{1,2}\x D^k$ from the identity to $j\circ j^{-1}$ and $L_{N_2}:\E_{N_2}\x I\to \E_{N_2}$ from the identity to $j^{-1}\circ j$ 
   with the following properties
  \begin{enumerate}
  \item[(ii)] $\e_{s_0}(L_{1,2}(m,\ga,x,t))=\e_{s_0}(L_{N_2}(m,\ga,x,t))=h_1(m,t)$ (up to a reparametrization);
\item[(iii)]  $L_{1,2}$ preserves the constant loops $\E_{1,2}^0\x D^k$ and half-contant loops $\R_{1,2}\x D^k$. 
\item[(iv)] $L_{N_2}$ preserves $\E^0_{N_2}|_{M_1}$ and $\R_{N_2}|_{M_1}$; 
\item[(v)] $L_{N_2}$ takes constant maps $\E_{N_2}^0$ to the subspace of ``thin maps''
$$\E^{thin}_{N_2}:=\{(m,\ga,x)\in \E_{N_2}\ |\ \ga:S\xrightarrow{u} S\cup \{s_0,s_1\}\x I\surj \!\!\bigvee_{\{s_0,s_1\}}\!\! I\to  M_2\}\subset M_1\x \mathcal{P}M_2\x D^k.$$
  \end{enumerate}
\end{prop}

\begin{proof} 
By construction, the map $j^{-1}$ has image inside $\E_{N_2}$, landing
in the subspace of maps evaluating to the zero-section of $N_2$. The map
$j^{-1}$ preserves constant loops and partially constant loops since it only changes the loops by post-composing with the map $f\circ g$,
so property (i) in the statement holds.
It also preserves $(\E_{1,2}\x D^k)|_{M_1}=\E_{N_2}|_{M_1}$ since if $(m,\ga,x)$ satisfies that $\ga(s_1)=f(m)$ and $x=e(m)$, then 
$(g\circ f(m),f\circ g\circ \ga,e\circ g(\ga(s_1)))$ also satisfies that $f\circ g\circ \ga(s_1)=f(g\circ f(m))$ and $e\circ g(\ga(s_1))=e(g\circ f(m))$. In particular it preserves $(\R_{1,2}\x D^k)|_{M_1}=\R_{N_2}|_{M_1}$ as required.  

\smallskip

We will check that $j^{-1}$ is a relative homotopy inverse by writing down explicit homotopies $L_{N_2}:\E_{N_2}\x I\to \E_{N_2}$ and $L_{1,2}:\E_{1,2}\x D^k\x I\to \E_{1,2}\x D^k$ between the identity and $j^{-1}\circ j$ (resp. $j\circ j^{-1}$). 
We start by defining the homotopy $L_{1,2}$, which is easier to construct. Let $(m,\ga,x)\in \E_{1,2}\x D^k$. Then $$j\circ j^{-1}(m,\ga,x)=(g\circ f(m),f\circ g\circ \ga,e\circ g(\ga(s_1))).$$
To construct $L_{1,2}$ (see Figure~\ref{Htpy1} for an illustration), we use the homotopy $h_1$ on the evaluation at $s_0$, and $h_2$ on the map component,  making use of the deformation $\Psi_K^1$ of Lemma~\ref{lem:Kdef} to ensure compatibility between these two contradictory definitions: set
$$L_{1,2}(m,\ga,x,t)=\big(h_1(m,t),\Psi_K^1(m,t)\big(h_2(\ga,t)\big),\ell(x,t)\big)$$
for  $\ell(x,-)$  a path in $D^k$ between $x$ and $e\circ g(\ga(s_1))$, continuous in $x$ and so that $\ell(e(m),-)=e\circ h_1(m,t)$ when $\ga(s_1)=f(m)$.

Note first that $(h_1(m,t),\Psi^1_K(h_2(\ga,t)))$ is indeed a well-defined element of $\E_{1,2}$ since $h_2(\ga(s_0),t)=h_2(f(m),t)$ and
$\Psi^1_K(m,t)(h_2(f(m),t))=f(h_1(m,t))$ as required. Recall from Lemma~\ref{lem:Kdef} that $\Psi^1_K(m,0)=\id=\Psi^1_K(m,1)$. Hence $L_{1,2}(m,\ga,x,0)=(m,\ga,x)$ and
$L_{1,2}(m,\ga,x,1)=(g\circ f(m),f\circ g\circ \ga,e\circ g\circ \ga(s_1))$, so it is indeed a homotopy from the identity to the composition $j\circ j^{-1}$. 

By definition, $ev_0\circ L_{1,2}(m,\ga,x,t)=h_1(m,t)$. 
Also, the map $L_{1,2}$ respects constant loops and half-constant loops, and preserves $(\E_{1,2}\x D^k)|_{M_1}$ since for $(m,\ga,e(m),t)$ with $\ga(s_0)=f(m)=\ga(s_1)$, we have that
$f(h_1(m,t))=\Psi^1_K(m,t)(h_2(\ga(s_i),t))$ for $i=0,1$ and $\ell(e(m),t)=e\circ h_1(m,t)$ by definition. In particular it preserves  $(\R_{1,2}\x D^k)|_{M_1}$ and $(\E^0_{1,2}\x D^k)|_{M_1}$.
Hence it is a homotopy from the identity to $j\circ j^{-1}$, and satisfies properties (ii) and (iii) in the statement.

\begin{figure}[!htb]
    \centering
    \begin{minipage}{.5\textwidth}
\begin{lpic}{Htpy1newbis(0.3,0.3)}
\lbl[b]{59,164;$\ga$}
\lbl[bl]{210,62;$f\!\circ\! g\!\circ\! \ga$}
\lbl[bl]{210,122;$f\!\circ\! g\!\circ\! \ga(s_1)$}
\lbl[b]{105,120;$h_2(\ga,t)$}
\lbl[bl]{78,60;$\Psi^1_K(m,t)$}
\lbl[br]{49,119;$f(m)\!=\!\ga(s_0)$}
\lbl[b]{125,178;$\ga(s_1)$}
\lbl[bl]{135,126;$h_2(\ga(s_1),t)$}
\lbl[tl]{142,50;$f\!\circ\! g\!\circ\! \ga(s_0)$}
\lbl[tl]{142,37;$=f\!\circ\! g\!\circ\! f(m)$}
\lbl[br]{57,49;$f\!\circ\! h_1(m,t)$}
\end{lpic}
        \caption{ The map component  of $L_{1,2}(m,\gamma,-)$.}
        \label{Htpy1}
    \end{minipage}%
    \begin{minipage}{0.5\textwidth}
\begin{lpic}{Htpy2newbis(0.3,0.3)}
\lbl[b]{59,164;$\ga$}
\lbl[bl]{210,62;$f\!\circ\! g\!\circ\! \ga$}
\lbl[b]{105,120;$h_2(\ga,t)$}
\lbl[bl]{78,60;$\Psi^1_K(m,t)$}
\lbl[br]{48,120;$f(m)$}
\lbl[b]{121,180;$f(m')$}
\lbl[b]{115,160;$\ga(s_1)$}
\lbl[b]{170,164;$K(m')$}
\lbl[br]{57,49;$f\!\circ\! h_1(m,t)$}
\lbl[bl]{204,167;$f\!\circ\! h_1(m',t)$}
\end{lpic}
        \caption{ The map component  of $L_{N_2}(m,\gamma,x,-)$.}
        \label{Htpy2}
    \end{minipage}
  \end{figure}

  \medskip

  We are left to define $L_{N_2}:\E_{N_2}\x I\to \E_{N_2}$, a homotopy between the identity and $j^{-1}\circ j$.  We have again that  $j^{-1}\circ j(m,\ga,x)=(g\circ f(m),f\circ g(\ga),e(g\circ \ga(s_1)))$.
  We will construct $L_{N_2}$ as a modification of $L_{1,2}$ that ensures that the evaluation at $s_1$ now stays within $N(f,e)$ at all times. We cannot use the deformation $\Psi^1_K$ to achieve this, like we did to control the evaluation at $s_0$, as the deformations for $s_0$ and $s_1$ could attempt to pull the same point in different directions. We will instead control the evaluation at $s_1$ by adding sticks:
  Define first a map $L_{N_2}^\uparrow:\E_{N_2}\x I\to \E_{N_2}$ given by
  $$L^{\uparrow}_{N_2}(m,\ga,x,t)=\Big(h_1(m,t),L^{\uparrow}_{N_2}(m,\ga,x,t)_{\E_2}, e\circ h_1(m',t)\Big)$$
  with $\E_2$-component (illustrated in Figure~\ref{Htpy2}) given by 
  $$L^{\uparrow}_{N_2}(m,\ga,x,t)_{\E_2}: S\to S\cup \{s_1\}\x I \xrightarrow{L_{1,2}\cup \delta} M_2$$
  attaching a path $\delta$ to the mapping component of $L_{1,2}$, defined as follows. 
  By assumption, $(\ga(s_1),x)\in N(f,e)$ can be written as $\exp_{f(m')}(W,X)$ for some $m'\in M_1$, and $(W,X)\in TM_2\oplus \RR^k$. The path $\delta(m,\ga,x,t)$ is a concatenation $\delta=\delta_1*K(m',t,-)$ with 
  $$\Psi^1_K(m,t)(h_2(\ga(s_1),t)) \ \ \stackrel{\delta_1}{\leadsto}\ \  h_2(f(m'),t) \stackrel{K(m',t,-)}{\leadsto} f\circ h_1(m',t)$$ where the first path is defined by  
\begin{align*}
  \delta_1(r)& =\Psi^{1-r}_K(m,t)(h_2(\exp_{f(m')}(1-r)W),t).
\end{align*}

\smallskip

\noindent
{\em Modification 1:} The map $L_{N_2}^\uparrow$ as defined is not quite a homotopy between the identity and $j^{-1}\circ j$ since at $t=0$ it yields $(m,\ga*_{s_1}\delta_W,e(m'))$,
and at $t=1$ it yields
$(g\circ f(m),f\circ g(\ga*_{s_1}\delta_W),e(g\circ f(m')))$ where $\delta_W(r)=\exp_{f(m')}(1-r)W$ is non-trivial if $W\neq 0$. We thus need to concatenate $L_{N_2}^\uparrow$ with a homotopy at the start growing the path $\delta_W$, with its associated disc-coordinate, and the end a homotopy retracting $(f\circ g(\delta_W), e\circ g(rW))$. We denote the resulting homotopy $\tilde L_{N_2}^\uparrow$. It is a homotopy between the identity and $j^{-1}\circ j$ and it satisfies property (ii) in the statement by construction, but it is not yet a homotopy that preserves $\R_{N_2}|_{M_1}$ and $\E_{N_2}^0|_{M_1}$.

\smallskip

\noindent
{\em Modification 2:} If $(\ga(s_1),x)=(f(m),e(m))$, then $W=0$ and the path $\delta_1$ has  $\delta_1(r)=\Psi_K^{1-r}(m,t)(h_2(f(m),t))$ $=K(m,t,1-r)$, with the second half of $\delta$  running exactly the same path in reverse.
So the homotopy does not preserve constant and partially constant loops but instead creates a path that goes back and forth along $K(m,t,-)$.
This is homotopic to the definition we want on $\E_{N_2}|_{M_1}$, which would be the same map but without these added thin loops, i.e.~the map $L_{1,2}$ instead. To glue the two maps together, we need to know that the pair $(\E_{N_2},\E_{N_2}|_{M_1})$ is an NDR pair. This can be deduced from the fact that the pair $(N_2,M_1)$ is an NDR pair: let $u:N_2\to I$ and $q:N_2\x I\to N_2$
be the maps defined in the proof of Lemma~\ref{lem:N2} to show that $(N_2,M_1)$ is an NDR-pair. Then define $\hat u=u\circ p: \E_{N_2}\to N_2 \to I$ and $\hat q:\E_{N_2}\x I\to \E_{N_2}$  by $\hat q(m,\ga,x,t)=(m,\ga',x')$ where $q(m,\ga(s_1),x)=(n',x'):=(\ga'(s_1),x')$ and $\ga'(s)=q(m,\ga(s),x)|_{M_2}$ if $(\ga(s),x)\in N(f,e)$, and $\ga'(s)=\ga(s)$ otherwise, which gives a continuous map $\ga'$ since $q$ only affects points close to the zero-section of $N(f,e)$.

The resulting modified $L_{N_2}$ now satisfies properties (ii) and (iv) in the statement. 

Lastly, we have that $L_{N_2}$ adds (at most) paths $\delta$ to otherwise constant loops, replacing them by thin loops, as claimed. 
   \end{proof}

\begin{rem}\label{rem:LN20}
The homotopy $L_{N_2}$ is constructed so that it takes ``diagonal constant loops", that is elements of $\La^{r|0}_{N_2}|_{M_1}$, to constant loops. To non-diagonal constant loops,  the homotopy associates instead a family of thin loops, namely the paths $\delta$ in the above proof. Let $(m,[f(m)],x)\in \La^{r|0}_{N_2}|_{U_1^c}$, with $(f(m),x)=\exp_{f(m')}(W,X)$ in $N(f,e)$. 
As elements of $M_1\x PM_2\x D^k$, this family of thin loops is explicitly given (up to a small deformation) as
           $$L_{N_2}(m,[f(m)],x,t) \approx \left\{\begin{array}{ll} \big(m,\delta_W[0,3t],e(m')\big) & 0\le t\le \sfrac{1}{3}\\
                                                                                \big(h_1(m,3t-1),\la_K,e(h_1(m',3t-1)\big) & \sfrac{1}{3}\le t\le \sfrac{2}{3}\\
                                \big(f\circ g(f(m)),f\circ g(\delta_W[0,3t-2]),e\circ g(\delta_W(3t-2)))\big) & \sfrac{2}{3}\le t\le 1 
                         \end{array}\right.$$
where $\delta_W(r)=\exp_{f(m')}(1-r)W$ is a path from $f(m)$ to the nearby point $f(m')$, and  $$\la_K=\overline{K(m,3t-1,-)}*h_2(\delta_W,3t-1)*K(m',3t-1,-)$$
is a path connecting $h_1(m,t)$ and $h_1(m',t)$.  
Indeed, the path $\delta$ is given as a concatenation $\delta=\delta_1*K(m',t,-)$ with $\delta_1$ obtained by gradually applying the deformation $\Psi_K(m,t)$ to $h_2(\delta_W),t)$. The above is obtained by instead applying $\Psi_K$ only to the startpoint $h_2(f(m),t)$ of the path $h_2(\delta_W),t)$, creating the path $K(m,t,-)$ (in reverse), and then concatenating with the undeformed path $h_2(\delta_W),t)$. 
\end{rem}

\section{The string coproduct and the proof of Theorems~\ref{thm:formula} and~\ref{thm:formula2}}\label{sec:finalproof}

We will now deduce the formula in Theorems~\ref{thm:formula} and~\ref{thm:formula2} from Theorem~\ref{thm:Hcommute0}, applied to the case of the string coproduct, and more generally to the case of the higher coproducts described in the introduction. To apply Theorem~\ref{thm:Hcommute0}, we first need to describe these operations as lifted relative intersection products. We start by recalling in Section~\ref{sec:trivial} how to obtain the string coproduct from a relative version of the trivial coproduct. We then generalize this description to the higher coproducts in Section~\ref{sec:higher}. In both cases, the operation starts by crossing with an interval $I$ or a cube $I^r$. We will see that the shape of the boundary $\del I^r$ will be crucial for the formulas of Theorems~\ref{thm:formula} and~\ref{thm:formula2}.
Sections~\ref{sec:al} and \ref{sec:LN2} study the interaction between the boundary and the maps $\al$ and $L_{N_2}$ appearing in Theorem~\ref{thm:Hcommute0}, and Section~\ref{sec:proof} will assemble the results obtained to prove Theorems~\ref{thm:formula} and~\ref{thm:formula2} modulo the geometric description of $\T_f$ in terms of the homotopy $h_1$ and the fake diagonal, that will be given in Section~\ref{sec:fake}.

\subsection{The string coproduct in terms of the trivial coproduct}\label{sec:trivial}

Let $\e=\e_{0,\frac{1}{2}}: \La M\to M\x M$ be the evaluation map defined by $\e(\ga)=(\ga(0),\ga(\frac{1}{2}))$. Then
$$\La M|_M=\{\ga\in \La M\ |\ \ga(\frac{1}{2})=\ga(0)\}$$
identifies with the ``figure eight space'', seen as the subspace of $\La M$ of loops $\ga$ having a self-intersection
$\ga(0)=\ga(\frac{1}{2})$. 
We let $\R\subset \La M|_M$ denote the subspace of half-constant loops:
  $$\R=\{\ga\in \La M\ |\ \ga(t)=\ga(0) \ \forall t\in [0,\frac{1}{2}] \ \textrm{or}\ \forall t\in [\frac{1}{2},1]\}.$$

Following \cite{GorHin} (see also  \cite[Thm 2.13]{HinWah0}, combining with Section 3.2 of that paper, or \cite[sec 2.5]{NRW}), we can compute the string coproduct  as the
composition 
\begin{multline}\label{eq:rel}
 H_*(\La M) \xrightarrow{\x I} H_{*+1}(\La M\x I,\La M\x \del I\cup M\x I) \xrightarrow{\;J\;} H_{*+1}(\La M,\R) 
 \xrightarrow{int_{U_M}}  H_{*+1-n}(\La M|_{U_M},\R)  \\
 \xrightarrow{\ r\ }  H_{*+1-n}(\La M|_M,\R)\xrightarrow{cut}  H_{*+1-n}(\La M\x \La M,M\x \La M\cup \La M\x M)
\end{multline}
where  $$J:\La M\x I\to \La M$$  is the reparametrizing map defined by $J(\ga,s)=\ga\circ \theta_{\frac{1}{2}\to s}$ for $ \theta_{\frac{1}{2}\to s}:[0,1]\to [0,1]$ the piecewise linear map that fixes $0$ and $1$ and takes $\frac{1}{2}$ to $s$, and 
$r: \La M|_{U_M}\to \La M|_M$  
is the retraction map deforming the loops so the almost self-intersection at time $\ha$ becomes an actual intersection as in \cite[Sec 2.4]{HinWah0}.
(For simplicity, we consider the coproduct here with source $H_*(\La M)$ and not $H_*(\La M,M)$ as is often considered.)

\smallskip

Any map $f:M_1\to M_2$ induces a map $\La f:\La M_1\to \La M_2$ that respects the constant loops and half-constant loops.
Also, the maps $\x I$, $J$ and $cut$ are all natural in maps of the form $\La f$. Hence the question of homotopy invariance of the coproduct reduces to the homotopy/non-homotopy invariance of the middle composition
$$int_M: H_*(\La M,\R) \xrightarrow{int_{U_M}}  H_{*-n}(\La M|_{U_M},\R)  \xrightarrow{r}  H_{*-n}(\La M|_M,\R)$$
in \eqref{eq:rel}, which is a relative version  of the {\em short trivial coproduct}
$$int_M : H_{*-n}(\La M) \rar  H_{*-n}(\La M|_M)\cong H_*(\La M\x_M\La M)$$
(denoted $\vee_{\frac{1}{2}}^{\F}$ in \cite{HinWah0}) that looks for self-intersections of the form $\ga(0)=\ga(\ha)$ and does not cut. This last operation has long been known to be essentially trivial (see \cite{Tam} and \cite[Lem 4.5]{HinWah0}). Proposition~\ref{prop:easynat} implies that it is homotopy invariant, a property that could also be checked directly in this case. What we are interested in here is how invariant the relative version of this map is. Given that  half-constant loops $\R$ define a fibration over $M$, and not over $M\x M$, we are in the situation of Assumption~\ref{ass:B}, and we know from Theorem~\ref{thm:Hcommute0} that there is an obstruction to invariance. Our goal here is to analyse this obstruction.  

\smallskip

The obstruction given in Theorem~\ref{thm:Hcommute0} takes the form of a sequence of maps applied to the boundary of the class considered in the source. In our case, starting from a cycle $A\in C_*(\La_1)$ in \eqref{eq:rel}, we create the product $A\x I\in C_{*+1}(\La_1\x I, \La_1\x \del I)$ that has boundary $A\x \del I$. The first step in our analysis of the obstruction is to rewrite this boundary term, after applying the reparametrization map $J$, in terms of the Chas-Sullivan product:

\begin{lem}\label{lem:dA}
Suppose $A\in C_*(\La_1)$ is a non-relative cylce. Then $$\del (J(A\x I))=[M_1]\wedge A-A\wedge [M_1]\ \ \in H_*(\R_1)$$ is the left and right half-constant loops versions of $A$. 
\end{lem}

\begin{proof}
  Given that $J$ is a chain map, $\del(J(A\x I))=J(A\x \del I)$. Now $J(\ga,0)=\ga(0)*\ga$ reparametrizes the loop to become constant on $[0,\ha]$ and similarly $J(\ga,1)=\ga*\ga(1)$. This is precisely the effect on chain of taking a Chas-Sullivan product with the fundamental class on the left and right respectively. 
\end{proof}

We want to emphasize that the chain $A\x I$ in the above lemma is never a cycle in $C_*(\La_1\x I)$, as $\del(J(A\x I)))$ is always non-trivial, and indeed can be computed as $[M_1]\wedge A-A\wedge [M_1]$. Now this last cycle is always a boundary in $C_*(\La_1)$, corresponding to the fact that $[M_1]$ is a unit for the string product. It is however not a boundary in general  in $C_*(\R_1)$, where we will be considering it. 

\medskip

From Theorem~\ref{thm:Hcommute0}, we have that the failure of invariance for such a cycle $A$ is computed by 
\begin{align}\label{equ:1}
  cut\circ \hat h\circ int_{M_1}\circ  L_{N_2}((\al([M_1]\wedge A-A\wedge [M_1]))\x I)\big) 
  \end{align}
   To prove Theorem~\ref{thm:formula}, we will ``pull out'' $\wedge A$ and $A\wedge$ in the above expression.   We start by introducing the more general set-up of higher mapping spaces, in which we will write the rest of the proof.

   \subsection{Higher coproducts}\label{sec:higher}

Let $\La^rM=\maps(S^r,M)=\maps(I^r/\del I^r,M)$ denote the space of {\em $r$-loops} in $M$, and recall from the introduction the higher coproduct
$$\vee^r: H_*(\La^r M)\xrightarrow{int_{M}\circ (\x I^r)} H_{*+r-n}((\La^r M\x I^r)|_M,\La^r\x \del I^r)\xrightarrow{\reread} H_{*+r-n}(\maps(T^{r},M),\R_{T^r})$$
that looks for self-intersections of the form $\ga({\bf 0})=\ga({\bf t})$ for some ${\bf t}=(t_1,\dots,t_r)\in I^r$, and then ``rereads'' the loops by precomposing  loops $\ga: (I^r/\del I^r)/_{{\bf 0}\sim {\bf t}} \rar M$ with the map
$$T^r=S^1\x S^{r-1}\cong S^1\x \del I^r \rar (I^r/\del I^r)/_{{\bf 0}\sim {\bf c}}\xrightarrow{(\theta_{\frac{1}{2}\to t_j})_{j=1}^r} (I^r/\del I^r)/_{{\bf 0}\sim {\bf t}}$$
where the first map takes $S^1\x \{{\bf t}\}$ to the line from ${\bf t}$ to ${\bf c}=(\ha,\dots,\ha)$, which is a circle in the target, and the map $ \theta_{\frac{1}{2}\to t_j}:[0,1]\to [0,1]$ is the same reparametrization map as in the previous section, applied one coordinate at a time. 
The subspace $\R_{T^r}$ in the target of the higher coprodruct is the space of half-constant maps
$$\R_{T^r}=\bigcup_{j=1}^r\{\ga\in \maps(T^{r},M) \ |\ \ga(s,{\bf s'})=\ga(0,{\bf 0}) \ \forall {\bf s'}\in \del_j^-I^{r}\textrm{ or } \forall {\bf s'}\in \del_j^+I^{r}\}$$
where $\del_j^-I^{r}=\del I^{r}\cap (I^{j-1}\x [0,\ha]\x I^{r-j}) $ and $\del_j^-I^{r}=\del I^{r}\cap (I^{j-1}\x [\ha,1]\x I^{r-j})$ are the ``left'' and ``right'' hemispheres of $\del I^r\cong S^{r-1}$ in the $j$th direction. Indeed, just like in the case of single loops in the previous section, the reparametrization $\theta_{\frac{1}{2}\to t_j}$ has the effect of making the map half-constant in the $j$th direction when $t_j=0$ or $1$. 

\medskip

This operation is a direct generalization of the string coproduct considered above, and just like the string coproduct, it can be computed using a relative version of a ``trivial coproduct''.
Let
$$\R_{\La^r}=\bigcup_{j=1}^r\{\ga\in \La^rM \ |\ \ga({\bf s})=\ga({\bf 0}) \ \forall {\bf s}\in I^{j-1}\x [0,\ha]\x I^{r-j}\textrm{ or }\forall {\bf s}\in I^{j-1}\x [\ha,1]\x I^{r-j}\}$$
be the subspace of half-constant loops. 
We can compute the higher  string coproduct  as the
composition
\begin{multline}
 H_*(\La^rM) \xrightarrow{\x I^r} H_{*+r}(\La M^r\x I^r,\La M\x \del I^r) \xrightarrow{\;J^r\;} H_{*+r}(\La^r M,\R_{\La^r}) 
 \xrightarrow{int_{U_M}} \\
 H_{*+r-n}(\La^r M|_{U_M},\R_{\La^r})  \xrightarrow{r}  H_{*+r-n}(\La^r M|_M,\R_{\La^r})\xrightarrow{\reread}  H_{*+r-n}(\maps(T^{r},M),\R_{T^r})
\end{multline}
where  $$J^r:\La^r M\x I^r\to \La^r M$$  is the reparametrizing map defined by $J^r(\ga,{\bf t})=\ga\circ (\theta_{\frac{1}{2}\to t_j})_{j=1}^r$, 
and  $r: \La ^r M|_{U_M}\to \La^r M|_M$
is the retraction map deforming the $r$-loops so the almost self-intersection at time ${\bf c}=(\ha,\dots,\ha)$ becomes an actual intersection, just  as in \cite[Sec 2.4]{HinWah0}.
 The reread map is now just 
 precomposing with the map $S^1\x S^{r-1}\to (I^r/\del I^r)/_{{\bf 0}\sim {\bf c}}$ described above, as we have already reparametrized the loops.

\medskip

To give an analogue of Lemma~\ref{lem:dA} in higher dimensions, we need an appropriate higher version of the string product that will allow to describe boundary classes of the form $J(A\x \del I^r)$.
This will take the form of a standard intersection product, followed by a concatenation map {\em along $\del I^r$}, done using the action of the little $r$-disc operad $\C_r$ on $r$-loops sharing the same basepoint. In the case $r=1$, this operation will recover the two terms $[M_1]\wedge A-A\wedge [M_1]$ in one go.

\smallskip

Let $\C_r$ denote the little cube operad in dimension $r$, with $b$ is a chain representative of the generator of $H_{r-1}(\C_r(2))$. We denote
$$[-,-]: C_*(\La^r M\x_M\La^r M)\xrightarrow{-\x b} C_{*+r-1}(\La^r M\x_M\La^r M\x \C_r(2))\rar C_{*+r-1}(\La^rM)$$
the associated Browder (or Gerstenhaber) bracket, defined just as for based loop spaces, where we note that the standard little cube action on based loops also makes sense for free loops that share the same basepoint.
Now define the higher product
\begin{equation}\label{eq:wedger}
  \wedge_{\del I^r}: C_*(\La^r M\x \La^rM)\xrightarrow{int_M}  C_{*-n}(\La^r M\x_M\La^r M)\xrightarrow{[-,-]} C_{*+r-1-n}(\La^rM).
\end{equation}
We will consider the restriction of this operation to $M\x \La^rM$, considered as a map with value in half-constant loops:
\begin{equation*}
  \wedge_{\del I^r}: C_*(M\x \La^rM)\xrightarrow{int_M}  C_{*-n}(M\x_M\La^r M)\xrightarrow{[-,-]} C_{*+r-1-n}(\R_{\La^r}).
\end{equation*}
where $b:\del I^r\to \C_r(2)$ is a chosen representative the generator of $H_{r-1}(\C_r(2))$, given in Figure~\ref{fig:Dr} (right figure), with the property that the second cube lies inside a single hemisphere, so that the value of the bracket on classes coming from $M\x \La^rM$ is indeed within half-constant loops.

  \begin{figure}
\begin{lpic}{D2rdouble2(0.4,0.4)}
  \lbl[b]{27,57;$1$}
  \lbl[b]{62,35;$2$}
  \lbl[b]{41,35;${\bf c}$}
  \lbl[b]{0,66;${\bf t}$}
   \lbl[b]{154,46;$1$}
 \lbl[b]{175,35;$2$}
 \lbl[b]{150,35;${\bf c}$}
 \lbl[b]{107,66;${\bf t}$}
\end{lpic} \caption{Two homotopic choices of homology generator in $C_{r-1}(\C_r(2))$ parametrized by ${\bf t}\in \del I^r$}\label{fig:Dr}
\end{figure}

\begin{lem}\label{lem:dAr}
Suppose $A\in C_*(\La^rM)$. Then $\del (J^r(A\x I^r))= [M]\wedge_{\del I^r} A\in H_{*+r-1}(\R_{\La^r})$. 
  \end{lem}

  \begin{proof}
    Let $b':\del I^r\to D_r(2)$ be a choice of representative for $[-,-]$ given in Figure~\ref{fig:Dr} (left figure). 
The second cube in $b'({\bf t})$ is chosen to be the cube with interior $\theta_{\bf t}^{-1}(I^r\minus \del I^r)$. 
Then for any loop $\ga\in \La^rM_1$ and any  ${\bf t}\in \del I^r$, the loop $J^r(\ga,{\bf t})$ is a linear reparametrization of $b'({\bf t})(\ga(0),\ga)$. Interpolating between the two parametrization shows that they represent the same homology class. Now we have that
$b'({\bf t})(\ga(0),\ga)\simeq b({\bf t})(\ga(0),\ga)\in \R_{\La^r}$. 
    The result follows from the fact that $int_M([M]\x A)=(\e_{0}(A)\x A)$, computed as a transverse intersection.  
  \end{proof}

  The reason we introduced two homotopic generators of $H_{r-1}(\C_r(2))$ in Figure~\ref{fig:Dr} is that the left generator ($b'$) gives essentially a direct description of the boundary $\del (J^r(A\x I^r))$, while the second generator ($b$) has the additional property that the center of the cube ${\bf c}$ coincides with the center of the first cube. While this last feature does not make any difference here, it will become relevant in the next section when we define versions of this operation involving spaces like $\La_{N_2}$. 

\begin{rem}
  Just like in the case $r=1$, the class $\del (J^r(A\x I^r))$ would be trivial if considered in $H_{*+r-1}(\La^rM)$, fillable by a corresponding $I^r$-parametrized class, but it is in general not trivial in $H_{*+r-1}(\R_{\La^rM})$. It is a crucial for the non-homotopy invariance of these higher coproducts
  that this boundary class is non-trivial in $H_{*+r-1}(\R_{\La^rM})$.
One could for example extend the definition of the $r$-higher coproduct to $s$-loops with $s>r$, and in this case the corresponding boundary class would be trivial in half-constant $s$-loops, and the obstruction would vanish, so the corresponding operation would be homotopy invariant.  
\end{rem}
    
\subsection{Commuting higher products with the map $\al$}\label{sec:al}

In what follows, we write $$\La_i^r=\La^r M_i$$ for the $r$-loop space of $M_i$, for $i=1,2$, and $\R_i$ for its subspaces half-constant loops, and use the notations $\La^r_{N_2},\La^r_{1,2}$, etc.~for the pulled-back fibrations from $\La^r_2$ and $\R_2$ using the same notational convention as in Section~\ref{sec:inv1}. 
We write $\La_i^{r|0}$ for the subspace of constant loops in the $r$-loop space $\La^r_i$, and likewise for $\La^{r|0}_{N_2}$ etc.

Recall from Section~\ref{sec:assB} the map
$$\al=[\e^*\tau_e\cap]\circ [\x D^k]\circ j: C_*(\La^r_{N_2},\R_{N_2})\to C_*(\La^r_{N_2},\R_{N_2}).$$
For simpicity, we will also denote by $\al$ the composite map 
$$C_*(\La^r_1,\R_1) \xrightarrow{F_{N_2}} C_*(\La^r_{N_2},\R_{N_2})\xrightarrow{\ \al\ } C_*(\La^r_{N_2},\R_{N_2}).$$
This map does not affect the loop component, and thus respects constant loops and half-constant loops. 
In particular, $\al[M_1]$ can be considered in $C_n(\La_{N_2}^{r|0})$. 

The goal of the present section is to show the following result:

\begin{prop}\label{prop:step1} 
  For any $A\in H_*(\La^r_1)$, 
  $$\al\big([M_1]\wedge^1_{\del I^r}A\big)=\al [M_1]\wedge^{N_2}_{\del I^r}A\ \in H_*(\R_{N_2}), \hspace{2cm} (*)$$
  where
\begin{equation*}
  \xymatrix@R=1pc{\wedge^1_{\del I^r}: H_*(\La_1^{r|0}\x \La^r_1) \ar[r]^-{int_{M_{1}}}&  H_{*-n}(\La_1^{r|0}\x_{M_1} \La^r_1)   \ar[r]^-{[-,-]} & H_{*+r-1-n}(\R_1)\\
 \wedge^{N_2}_{\del I^r}:   H_*(\La_{N_2}^{r|0}\x \La^r_1) \ar[r]^-{int_{M_{1}}} &  H_{*-n}(\La_{N_2}^{r|0}\x_{M_1} \La^r_1)  \ar[r]^-{[-,-]} & H_{*+r-1-n}(\R_{N_2})}
\end{equation*}
are higher products as in~\eqref{eq:wedger}, 
where the bracket in the second case is defined by setting $b_{\bf t}((m,\ga,x),\la):=(m,b_{\bf t}(\ga,f(\la)),x)$. 
   \end{prop}

   Note that the definition of the bracket $[-,-]: H_*(\La_{N_2}^{r|0}\x_{M_1} \La^r_1)\to H_{*+r-1-n}(\R_{N_2})$ makes sense with the chosen generator $b$ (right picture in Figure~\ref{fig:Dr}) because
   $\e_{\bf c}\big(b_{\bf t}((m,\ga,x),\la)\big)=\e_{\bf c}(\ga)$. This gives that if $(m,\ga,x)\in \La^r_{N_2}$, so is $(m, b_{\bf t}((m,\ga,x),\la),x)$. 

\begin{proof}
The sequence of maps 
  $$\xymatrix{\La_1^{r|0}\x \La^r_1  \ar[dr]_{\e=(\e_0,\e_0)\  \ } \ar[r]^-{e\x \id}  & \La^{{r|0}}_{N_2}\x \La^r_1  \ar[d] \ar[r]^-{j\x \id} & (\La_{1,2}^{r|0}\x D^k)\x \La^r_1 \ar[dl]^{\ \ \e=(\e_0,\e_0)} \\
 &M_1\x M_1&&
}$$
over $M_1\x M_1$ gives compatible intersection products $int_{U_1}$. 
The first and last vertical maps are fibrations. The second vertical
map is not in general a fibration because of the factor $\La_{N_2}^{r|0}$,  but we can none-the-less construct a retraction map
$$r_1: \La^{{r|0}}_{N_2}\x_{U_1} \La^r_1\rar  \La^{{r|0}}_{N_2}\x_{M_1} \La^r_1 $$
simply by adding a stick to the loop in $\La^r_1$. (Adding a stick to an $r$-loop is formally done using the map $u:S\to S^{\uparrow}$ of Section~\ref{sec:invert} in the case $S=I^r/\del I^r$.)
So we in fact have compatible $int_{M_1}$ products (using Proposition~\ref{prop:easynat}), giving the commutativity of the first square in the following diagram: 
  $$\xymatrix{H_*(\La^{r|0}_1\x \La^r_1) \ar[d]_-{ (j\circ e) \x \id}\ar[r]^{int_{M_1}} &    H_{*-n}(\La^{r|0}_1\x_{M_1} \La^r_1) \ar[d]^-{}\ar[r]^-{[-,-]} & H_{*+r-1-n}(\R_{1}) \ar[d]^-{j\circ e}  \\
    H_*(\La^{{r|0}}_{1,2}\x \La^r_1)\ar[r]^{int_{M_1}}\ar[d]_-{\x D^k} \ar@{}[dr]|{\fbox{$(-1)^{nk}$}} &  \ar@{}[dr]|{\fbox{$(-1)^{(r-1)k}$}}  H_{*-n}(\La^{r|0}_{1,2}\x_{U_1} \La^r_1))\ar[r]^-{[-,-]}\ar[d]^-{\x D^k} & H_{*+r-1-n}(\R_{1,2}) \ar[d]^-{\x D^k} \\
    H_{*+k}(\La^{{r|0}}_{1,2}\x \La^r_1\x D^k,\del D^k)\ar[r]_-{int_{U_1}}  \ar[d]_-{\tau} &  H_{*+k-n}(\La^{{r|0}}_{1,2}\x_{U_1} \La^r_1\x D^k,\del D^k)\ar[r]_-{[-,-]}
    \ar[d]^-{\tau} &  H_{*+k+r-1-n}(\R_{1,2}\x D^k,\del D^k)   \ar[d]^-{\tau} \\
       H_{*+k}(\La^{{r|0}}_{1,2}\x D^k\x \La^r_1,\del D^k)\ar[r]^-{int_{U_1}}  \ar[d]_-{[\e^*\tau_e\cap]\x 1}\ar@{}[dr]|{\fbox{$(-1)^{nk}$}} & \ar@{}[dr]|{\fbox{$(-1)^{(r-1)k}$}}  H_{*+k-n}((\La^{{r|0}}_{1,2}\x D^k)\x_{U_1} \La_1,\del D^k)\ar[r]^-{[-,-]}
    \ar[d]^-{[\e^*\tau_e\cap]\x 1} &  H_{*+k+r-1-n}(\R_{1,2}\x D^k,\del D^k)   \ar[d]^-{[\e^*\tau_e\cap]} \\
 H_*(\La^{{r|0}}_{N_2}\x \La^r_1) \ar[r]_-{int_{U_1}} 
 & H_{*-n}(\La^{{r|0}}_{N_2}\x_{U_1}  \La^r_1) \ar[r]_-{[-,-]}  
    & H_{*+r-1-n}(\R_2)  
 }$$
 The third square in the left column commutes likewise by naturality of the intersection product over a fixed manifold (Proposition~\ref{prop:easynat}).  The second square commutes a cap and cross product, giving a sign $(-1)^{nk}$, and the fourth commutes two cap products, giving again a sign $(-1)^{nk}$. For the right column, recall that the map $[-,-]$ is the composition of a map $\x [b]=\x [S^{r-1}]$ with the topological action of the little cube operad $\C_r(2)$. The first map creates likewise signs $(-1)^{(r-1)k}$ when commuted with crossing with the disc $D^k$ in the second square, and when commuting with capping with the pull-back of $\tau_e$ in the last square. The other squares commute by naturality.

Now in total the signs cancel out and the diagram commutes. The left vertical composition takes a product class $[M_1]\x A$ to $\al[M_1]\x A$, while the right vertical composition applies $\al$ to the product. The top horizontal composition is the higher product $\wedge^1_{\del I^r}$ and the bottom one is  the higher product $\wedge^{N_2}_{\del I^r}$. The statement follows.  
\end{proof}

\subsection{Commuting higher products with the homotopy $L_{N_2}$}\label{sec:LN2}

Recall that $L_{N_2}$ is a homotopy between the identity and $j^{-1}\circ j$ that can be chosen to agree with the homotopy $h_1$ under the evaluation at $s_0={\bf 0}$ (Proposition~\ref{prop:L12}(ii)). 
Because $j$ and $j^{-1}$ preserve constant loops and half-constant loops, it  
induces maps
 $$L_{N_2}\colon (\La_{N_2}^{{r|0}}\x I,\La_{N_2}^{{r|0}}\x \del I) \rar (\La^r_{N_2},\La^{{r|0}}_{N_2}) \ \ \ \textrm{and}\ \ \ \ L_{N_2}\colon (\R_{N_2}\x I, \R_{N_2}\x \del I) \rar (\La^r_{N_2},\R_{N_2}) .$$
 We will use in the following result a deformation of the map $L_{N_2}$ restricted to constant loops, that keeps the evaluation at ${\bf 0}$ constant instead of following $h_1$: define
 $$\tilde L_{N_2}\colon (\La_{N_2}^{{r|0}}\x I, \La_{N_2}^{{r|0}}\x \del I)  \rar (\La^r_{N_2},\La^{{r|0}}_{N_2}) $$
 by 
 $$\tilde L_{N_2}(m,[f(m)],x,t)=\left\{\begin{array}{ll} \big(m,f(h_1(m,[0,2t]))*_{s_0}L_{N_2}(m,[f(m)],x,2t)_{\La_2}, L_{N_2}(m,[f(m)],x,2t)_{D^k}\big)  & 0\le t\le \frac{1}{2} \\
                                         \big(m,f(h_1(m,[0,2-2t])), e(h_1(m,2-2t))\big)                            & \frac{1}{2}\le t \le 1  \end{array}\right.$$
                                     That is, in the first part of the map, the homotopy adds a stick along $f\circ h_1$ to the loop component of $L_{N_2}$, where we use the map $I^r/\del I^r=S\to S\cup_{s_0}s_0\x I$ induced by the map $u$ of Figure~\ref{fig:u} to add the stick, and in the second part, it retracts the thin loop created along $h_1$.

\begin{lem}\label{lem:LN2t}
The map $\tilde L_{N_2}\colon (\La_{N_2}^{{r|0}}\x I, \La_{N_2}^{{r|0}}\x \del I)  \rar (\La^r_{N_2},\La^{{r|0}}_{N_2}) $ is homotopic to the restriction of $L_{N_2}$ to the pair $(\La_{N_2}^{{r|0}}\x I, \La_{N_2}^{{r|0}}\x \del I)$. 
\end{lem}
  
\begin{proof}
A homotopy $\tilde L_{N_2}\simeq_H L_{N_2}$ can be defined by retracting the added paths: for $0\le t\le \ha$, $H(-,s)$ takes the concatenation of the loop component of $L_{N_2}(-,t)$ with the path $f(h_1(m,[2st,2t]))$, with $M_1$--component given by $h_1(m,2st)$, and for $\frac{1}{2}\le t \le 1$, $H(-,s)$ has loop component $ f(h_1(m,[s,s+(1-s)(2-2t)]))$, with $M_1$-component $h_1(m,s)$. The evaluation at $s_1$ stays constant when $0\le t\le \ha$, and we keep the disc component constant as well. For $\ha\le t\le 1$, the evaluation at $s_1$ is $f(h_1(m,s+(1-s)(2-2t)))$ and we set the disc component to be $e(h_1(m,s+(1-s)(2-2t)))$. At time $1$ this yields a reparametrization of $L_{N_2}$. 
\end{proof}

 The goal of this section is to prove the following result:

 \begin{prop}\label{prop:step2}
   For any cycles $A\in C_*(\La^r_1)$ and $B\in C_*(\La_{N_2}^{r|0})$, we have that 
   $$L_{N_2}((B\wedge^{N_2}_{\del I^r} A)\times I)= \tilde L_{N_2}(B\x I)\wedge^{N_2}_{\del I^r} A \ \ \in H_*(\La^r_{N_2},\R_{N_2}).$$
   \end{prop}

   \begin{proof}
We need to check that the following diagram commutes: 
$${\xymatrix{ 
  C_{*}(\La_{N_2}^{{r|0}}\x \La^r_1\x I,\del I)\ar[d]_{[\e_{0,0}^*\tau_1\cap]} \ar[r]^{\tau} &  C_{*}(\La_{N_2}^{{r|0}}\x I\x \La^r_1,\del I) \ar[r]^-{\tilde L_{N_2}\x \id } \ar[d]_{[\e_{0,0}^*\tau_1\cap]} & C_{*}(\La^r_{N_2}\x \La^r_1, \La^{{r|0}}_{N_2}\x \La^r_1)
  \ar[d]^{[\e_{0,0}^*\tau_1\cap]}\\
   C_{*-n}(\La_{N_2}^{{r|0}}\x_{M_1}\! \La^r_1\x I,\del I) \ar[r]^{\tau} \ar[d]_{[\ ,\ ]}  & C_{*-n}((\La_{N_2}^{{r|0}}\x I) \x_{M_1}\! \La^r_1,\del I) \ar@/^/[r]^-{\tilde L_{N_2}\x\id} \ar@/_/[r]_-{L_{N_2}\x h_1}     
   & C_{*-n}(\La^r_{N_2}\x_{M_1}\! \La^r_1, \La^{{r|0}}_{N_2}\x_{M_1}\! \La^r_1) \ar[d]^{[\ ,\ ]} \\
  C_{*+r-1-n}(\R_{N_2}\x I,\del I) \ar[rr]^-{L_{N_2}} & & C_{*+r-1-n}(\La^r_{N_2},\R_{N_2}). 
}}$$
The two top squares commute by the naturality of the cap product, because the evaluation maps are compatible, where we use that $\tilde L_{N_2}$ preserves the evaluation at ${\bf 0}$ for the second square. 
The bigon commutes up to homotopy, using the homotopy $H\x h_1$, with $H$ the homotopy between $\tilde L_{N_2}$ and $L_{N_2}$ of  Lemma~\ref{lem:LN2t}, and $h_1$ considered as a homotopy from $\id=h_1(-,0)$ to $h_1(-,t)$. 
Finally, the bottom square commutes up to homotopy on the space level, as we detail now.

Recall that the map $[\ ,\ ]$ is the action of the chosen generator $b=\{b_{\bf t}\}_{{\bf t}\in \del I^r}$ of $H_{r-1}(\C_r(2))$ (Figure~\ref{fig:Dr}, right). 
For each $b_{\bf t}$,
starting from $(m,[f(m)],x,\la,t)$ in the top left corner of that square, going to the right and then down gives a triple $(m,\ga,x')$ where $\ga$ is the $r$-loop concatenated at $f(m)$ defined by
 $$\ga'=b_{\bf t}(\delta_{m,x,t},f(h_1(\la,t)))$$
where the thin loop, or path, $\delta_{m,x,t}=L_{N_2}(m,[f(m)],x,t)_{\La_2}\approx \overline{K(m,t,-)}*h_2(\delta_W,t)*K(m',t,-)$ (up to a retraction if $m$ is close to $m'$, following Modification 2 in the proof of Proposition~\ref{prop:L12}), see Remark~\ref{rem:LN20}. 
           The disc component $x'=e(h_1(m',2t))$ is that of $L_{N_2}$ if $ 0\le t\le \frac{1}{2}$, and otherwise $x'=e(h_1(m,2-2t))$. 
The resulting $r$-loop is depicted in Figure~\ref{fig:D2H} (left), using the same color-coding at Figure~\ref{Htpy2}. 

             Now going the other way around the square takes $(m,[f(m)],x,\la,t)$ to $(h_1(m,t),\ga'',x')$  with 
\begin{align*}\ga''=L_{N_2}(m,b_{\bf t}([f(m)],f\circ \la),x,t)_{\La_2}&\approx \overline{K(m,t,-)}*_{s_0}h_2\big(b_{\bf t}([f(m)],f\circ \la),t\big)*_{s_1}h_2(\delta_W,t)*K(m',t,-) \\
  &=\overline{K(m,t,-)}*_{s_0}b_{\bf t}(h_2(\delta_W,t)*K(m',t,-),h_2(f\circ \la,t))
\end{align*}
where the indicated homotopy comes from replacing the deformation $\Psi^1_K$ in the definition of $L_{N_2}$, used to connect $h_2(f(-),t)$ to $f(h_1(-),t)$ at $s_0$,  with adding instead a stick $K(m,t,-)$, and the last equality is obtained by commuting $h_2$ and the operation $b_{\bf t}$, and likewise passing the concatenation at $s_1$ inside $b_{\bf t}$. 
This $r$-loop is depicted in Figure~\ref{fig:D2H} (right). A homotopy between the left and right side of that figure can be obtained by sliding the image of $\la$ along $K(m,t,-)$ (the red part of the figure).
  \begin{figure}[ht]
    \begin{lpic}{D2Hc(0.46,0.46)}
      \lbl[b]{95,65;$f(h_1(\la,t))$}
      \lbl[b]{35,30;$f(h_1(m,t))$}
       \lbl[r]{151,5;{\small $f(h_1(m,t))$}}
       \lbl[b]{257,30;$h_2(f(\la),t)$}
          \lbl[b]{200,30;$h_2(f(m),t)$}
        \end{lpic} \caption{The $r$-loops $\ga'$ and $\ga''$. In both figures, the boundary is mapped to $f(h_1(m,t))$ while the center is mapped to $f(h_1(m',t))$. The map is constant on the gray zones, mapped to $f(h_1(m,t))$ in the first case and to $h_2(f(m),t)$ in the second. There two sets of marked radial lines, that are mapped to $K(m,t,-)$ and $K(m',t,-)$ in both cases. }\label{fig:D2H}
\end{figure}
When $t=0$ or $1$, $K$ is a constant path and hence this homotopy is constant.  
\end{proof}

\subsection{The obstruction class $\T_f$ and the proof of Theorems~\ref{thm:formula} and \ref{thm:formula2} (part 1)}\label{sec:proof}

We give here a proof of the two formulas stated in the introduction, though with the class $\T_f$ defined in terms of the maps $\al$ and $L_{N_2}$, postponing to Section~\ref{sec:fake} the definition in terms of homotopy data and the fake diagonal.

\smallskip

Recall from Proposition~\ref{prop:L12}(v) that the homotopy $L_{N_2}$ takes constant loops to ``thin loops'' $\E^{thin}_{N_2}$, whose loop component is a path in $M_2$, from the evaluation at $s_0={\bf 0}$ to that at $s_1={\bf c}$. The same then holds for the map $\tilde L_{N_2}$ defined in the previous section, as it is obtained from $L_{N_2}$ by concatenating an extra path at $s_0={\bf 0}$. Evaluating at $(s_0,s_1)=({\bf 0},{\bf c})$ defines a map $\E^{thin}_{N_2}\xrightarrow{\e_{{\bf 0},{\bf c}}} N_2\xrightarrow{\ p_2\ } M_1\x M_1$ with associated intersection product 
$$int_{M_1}: C_*(\E_{N_2}^{thin}) \to C_{*-n}(\La_2)$$
where we only remember the loop component. Note that this loop naturally lives in the actual loop space $\La_2$ rather than in $\La^r_2$. 

\begin{Def}\label{def:Tf}
Let $\al[M_1]\in H_n(\La^{r|0}_{N_2})$ as before, with $\tilde L_{N_2}(\al[M_1]\x I) \in H_{n+1}((\La^r_{N_2})^{thin})$. Define 
$$\T_f:=(-1)^{n} int_{M_1}\circ \tilde L_{N_2}(\al[M_1]\x I) \ \ \in H_{1}(\La_2,M_2),$$
\end{Def}
where $int_{M_1}$ is as above, and where we note that $\T_f$ is independent of the chosen $r\ge 1$. 

\medskip

We are now ready to prove the two formulas given in the introduction.

\begin{proof}[Proof of Theorems~\ref{thm:formula}  and \ref{thm:formula2} (part 1)] 
Let $A\in C_*(\La_1)$. 
Theorem~\ref{thm:Hcommute0} gives that the failure of invariance is computed by 
$$(\star)=int_{M_1}\circ  L_{N_2}\big(\al(\del_{N_2}(A))\x I\big)$$
followed by an appropriate ``reread'' map. Using Lemma~\ref{lem:dA}, and then Propositions~\ref{prop:step1} and~\ref{prop:step2}, we get 
$$(\star)= int_{M_1}\circ  L_{N_2}\big(\al([M_1]\wedge^1_{\del I^r} A)\x I\big) = int_{M_1}\Big(\tilde L_{N_2}\big(\al[M_1]\x I\big)\wedge^{N_2}_{\del {I^r}} A\Big).$$
 Here the intersection product $int_{M_1}$ is with respect to the evaluation at ${\bf 0}$ and ${\bf c}$, which can be computed solely in the first entry of the higher product $\wedge^{N_2}_{\del I^r}$ by our choice of generator $b$ of $H_{r-1}(\C_r(2))$ (Figure~\ref{fig:Dr}(right)).
Hence this evaluation map commutes with taking the bracket, and
$$(\star)=(-1)^{rn}\Big(int^{N_2}_{M_1}\Big(\tilde L_{N_2}\big(\al[M_1]\x I\big)\Big)\Big)\wedge^{N_2}_{\del {I^r}} A$$
with a sign coming from commuting  capping  with $\e_1^*\tau_{M_1}$ (to achieve $int_{M_1}$) with crossing with $[b]$ and  capping with $\e_2^*\tau_{M_1}$ (to achieve $\wedge_{\del I^r}$), that is $(-1)^{n^2+(r-1)n}=(-1)^{rn}$. Here we have emphasized in the notation that the intersection product $int_{M_1}^{N_2}$ remembers $N_2$-data, that is 
$$int^{N_2}_{M_1}\Big(\tilde L_{N_2}\big(\al[M_1]\x I\big)\Big)\in H_1(\La_{N_2}|_{M_1}, \La_{N_2}^0|_{M_1})=H_1(\La_{N_3},\La_{N_3}^0).$$ 

Next we want to replace the product $\wedge^{N_2}_{\del I^r}$ computed via an intersection in $M_1$, with the product $\wedge^2_{\del I^r}$ computed via an intersection in $M_2$ instead. For this we verify that we can apply Theorem~\ref{thm:fibinv0}, namely the invariance of intersection product under Assumption (A). Indeed,
there is a map of pairs of fibrations:
$$\xymatrix{(\La_{N_3}\x \La_1, \La^0_{N_3}\x \La_1) \ar[r]^-{F\x f}\ar[d]^{\e_{0,0}} & (\La_2\x \La_2, \La^0_2\x \La_2)\ar[d]^{\e_{0,0}}\\
  M_1\x M_1\ar[r]^-{f\x f} & M_2\x M_2}$$
where the top map is a forgetful map on the first component. (This is of the same form as the relative Chas-Sullivan product of Corollary~\ref{cor:CSrel}.) Hence the map $F\x f$ commutes with the corresponding intersection products. As the map is also compatible with the bracket, we have that
$$(\star)=(-1)^{rn}\Big(int^{N_2}_{M_1}\tilde L_{N_2}\big(\al[M_1]\x I\big)\Big)_{\La^r_2}\wedge^{2}_{\del {I^r}} f(A)=(-1)^{(r-1)n}[\T_f]_r \wedge^{2}_{\del {I^r}} f(A)$$
with $[\T_f]_r$ is the the class $\T_f$ of Definition~\ref{def:Tf} considered as a self-intersecting (thin) $r$-loop via the maps $\La_2\inc PM_2\cong \E^{thin}\to \La_2^r$. Note that when $r=1$, this composition corresponds to the anti-diagonal $\diag$, a thin loop when $r=1$ being a loop going back and forth between $s_0$ and $s_1$. 

When $r=1$, the reread map takes these self-intersecting loops and cuts them to create two loops. The self-intersections $\ga({\bf 0})=\ga({\bf c})$ arise in $\T_f$ and the cutting map cuts in half these loops going back and forth along the same path. 
This gives precisely the formula in Theorem \ref{thm:formula}, with one term for each component of $\del I$. 

When $r>1$, the reread map precomposes with the map $T^r=S^1\x S^{r-1}\to I^r/_{{\bf 0}\sim {\bf c}}$. Again, the self-intersection happens inside $\T_f$ and we can rewrite
$$\reread((-1)^{(r-1)n}[\T_f]_r \wedge^{2}_{\del {I^r}} f(A))\simeq (-1)^{(r-1)n}\reread([\T_f]_r) \wedge^{2}_{\del {I^r}} f(A)$$
with $\wedge^2_{\del I^r}$ now considered as defining an action of $r$-loops on tori. 
The formula in Theorem~\ref{thm:formula2}, with $T_f$ as in Definition~\ref{def:Tf}, then follows by noting that 
the class $\reread([\T_f]_r)$ identifies with the suspension $s^{r-1}T_1$, as defined in the introduction.   
        \end{proof}

\subsection{The obstruction class $\T_f$ and the fake diagonal, and part 2 of the proof}\label{sec:fake}

The obstruction class $\T_f$ of Definition~\ref{def:Tf} has two main ingredients: the class $\al[M_1]$, which we will see below is  a transverse version of the ``fake diagonal'' $\De_1^f$,
and the map $\tilde L_{N_2}$, which, using Proposition~\ref{prop:L12}, can be built from higher homotopy data $(f,g,h_1,h_2,K)$. The goal of this section is to give an explicit description of $T_f$, and use it to give conditions under which it vanishes.  
Two main reasons for the obstruction to vanish will come out of our analysis: one from  $\al[M_1]$  being close enough to the actual diagonal $\De_1$, and one from the homotopy data being appropriately small. 

\medskip

Let $$\De_1^f=\{(m,m')\in M_1\x M_1\ |\  f(m)=f(m')\}.$$ This subspace is essentially never a manifold, and in particular does not immediately define a homology class.
The issue comes from the fact that the map $f\x f:M_1\x M_1\to M_2\x M_2$ is not transverse to the diagonal of $M_2$. One can though always deform $f$ by a small homotopy to a map $\tilde f$ transverse to $f$. 
i.e.~such that the pair $(f,\tilde f)$ is transverse to the diagonal in $M_2\x M_2$.  
We fix such a transverse deformation $\tilde f\simeq_{h_0}f$ close enough to $f$ that the map $(\tilde f,e):M_1\to M_2\x D^k$ has image inside $N(f,e)$. Define
$$ \underline\De_1^f :=\{(m,m')\in M_1\x M_1\ |\ f(m)=\tilde f(m')\}\subset M_1\x M_1$$
This is now an $n$-dimensional submanifold of $M_1\x M_1$, that also comes with an embedding $\id\x (\tilde f,e):\underline\De^f_1\inc N_2=M_1\x N(f,e)$ by our choice of $\tilde f$. 
We will write $$\hat{\underline\De}_1^f =\{(m,[n],e(m'))\in
\La^0_{N_2}\ |\ m,m'\in M_1\ \textrm{with}\ f(m)\!=\!n=\!\tilde f(m')\}\subset \La_{N_2}^0$$ for the corresponding
subspace of constant loops, that will be most relevant for us. 
Note that the analogous lift $\hat \De_1^f$ of $\De_1^f$ to $\La_{N_2}^0$ is isomorphic to the space $N_2\cap M_1\x D^k$ in Diagram~\eqref{equ:intsp}.  

\begin{ex}\label{ex:wiggle}
    Let $f:I\to I$ be a smooth map that folds part of the interval, as illustrated in Figure~\ref{fig:wiggle}. (The map $f$ is homotopic to the identity.) We have that $\De_1^f\cong \De_1\cup_{S^0} S^1$. On the other hand $\underline \De_1^f\simeq \De_1$ is a submanifold of $M_1\x M_1$. 
    \begin{figure}[ht]
      \centering
\begin{lpic}{wiggle4(0.25,0.25)}
  \lbl[b]{65,15;$a$}
  \lbl[b]{86,15;$b$}
  \lbl[b]{120,15;$c$}
  \lbl[b]{140,15;$d$}
   \lbl[b]{98,80;$f$}
    \lbl[b]{287,15;$a$}
  \lbl[b]{310,15;$b$}
  \lbl[b]{343,15;$c$}
  \lbl[b]{364,15;$d$}
    \lbl[r]{260,45;$a$}
  \lbl[r]{260,78;$b$}
  \lbl[r]{260,103;$c$}
  \lbl[r]{260,129;$d$}
  \lbl[b]{323,127;$\De_1^f$}
     \lbl[b]{509,15;$a$}
  \lbl[b]{532,15;$b$}
  \lbl[b]{565,15;$c$}
  \lbl[b]{586,15;$d$}
    \lbl[r]{482,45;$a$}
  \lbl[r]{482,78;$b$}
  \lbl[r]{482,103;$c$}
  \lbl[r]{482,129;$d$}
     \lbl[b]{547,127;$\underline\De_1^f$}
 \end{lpic}
 \caption{Graph of the map $f:I\to I$ of Example~\ref{ex:wiggle} and associated fake diagonal $\De_1^f$ and deformation $\underline \De^f_1$}\label{fig:wiggle}
\end{figure}
\end{ex}

   \begin{rem}  
   While $\De_1^f$ has a priori no nice geometric property, one can ask whether it is always possible modify $f$ within its homotopy class so that $\De_1^f=\De_1\cup D^f$ with $D^f$ a submanifold of $M_1\x M_1$ that is transverse to $\De_1$.
   (This question was raised by Tom Goodwillie in a question session during the Andrew Ranicki memorial conference.)
If $f$ is {\em generic} in the sense that whenever $|m-m'|>\eps$, there exists neighborhoods $U_m$ of $m$ and $U_{m'}$ of $m'$ such that $f|_{U_m}$ is transverse to $f|_{U_{m'}}$, then $\De_1^f\setminus U_1$ is a manifold. As we will see in Corollary~\ref{cor:fake} below, this is the part of the fake diagonal that is in fact most relevant for our purpose. 
     \end{rem}

\begin{prop}\label{prop:alphaDe1} Let $\al$ be the map  defined in Section~\ref{sec:assB}, here considered as a map $\al:C_*(\La^0_{N_2})\to C_*(\La^0_{N_2})$. Then  $\al[M_1]= [\hat{\underline\De}_1^f]\in H_n(\La^0_{N_2})$, where  $[M_1]$ is considered as a class of constant loops in $C_n(\La^0_{N_2})$. 
  \end{prop}

  \begin{proof}
    Let $D:M_1\x D^k\to \La^0_1\x D^k\to \R_{1,2}\x D^k$ be defined by $D(m,x)=(m,[f(m)],x)$, where $[f(m)]$ denotes as before the constant loop at $f(m)$. By definition, 
    $$\al[M_1]=\bar p^*\bar\tau_e\cap D[M_1\x D^k]\in H_n(\La^0_{N_2}).$$
    To compute the cap product as an intersection, we would need transversality in $M_1\x M_2\x D^k$ of the leftmost and rightmost maps in 
    $$\xymatrix{M_1\x D^k \ar[r]^-D\ar[dr]_{(\id,f)\x \id} & \R_{1,2}\x D^k\ar[d]^\e &\\
    & M_1\x M_2\x D^k & \ar[l]^-{\id\x (f,e)} M_1\x M_1.}$$
  These maps are not transverse where $df$ does not have full rank. To achieve transversality, we replace the right map by the homotopic map (in $N_2$) given by $\id\x (\tilde f,e)$, with $\tilde f$ as above. 
  It follows that  $\al[M_1]$ can be computed as the points $D(M_1\x D^k)$ lying over the intersection  
    $$\e\circ D[M_1\x D^k]\cap (\id\x (\tilde f,e))[M_1\x M_1]=\underline \De_1^f.$$
These points of $D(M_1\x D^k)$ identify precisely with $\hat{\underline \De}_1^f$, which proves the result. 
  \end{proof}

  \begin{prop}\label{prop:TfK}
    As before, let $g$ be a homotopy inverse for $f$ and $h_1$ a homotopy between the identity and $g\circ f$. Let $\tilde f$ be a small transverse deformation of $f$ as above, and let likewise $\tilde h_1$ be a small deformation of $h_1$ transverse to it (away from $M_1\x \del I$).
        The class $\T_f\in H_1(\La_2,M_2)$ can be represented as the family of loops 
        $$\overline{K(m,t,-)}*h_2(\delta_2,t)*K(m',t,-)*f(\delta_1)$$
based at $h_1(m,t)$,   parametrized by the 1-dimensional manifold 
        \begin{align*}\mathcal{D}_f&=\{(m,m',t)\in \underline \De_1^f\ |\ h_1(m,t)=\tilde h_1(m',t)\}\\ 
        &=\{(m,m',t)\in M_1^2\x I\ |\ f(m)=\tilde f(m')\ \textrm{and}\ h_1(m,t)=\tilde h_1(m',t)\}
        \end{align*}
      where      $\delta_1=\overline{h_1(m,t)h_1(m',1)}$ and $\delta_2=\overline{\tilde f(m')f(m')}$  are geodesic paths in $M_1$ and $M_2$ respectively. 
\end{prop}

         \begin{proof}
           The homology class $\T_f=int_{M_1}\circ \tilde L_{N_2}(\al[M_1]\x I)=int_{M_1}\circ L_{N_2}(\underline{\hat\De}^f_1\x I)$ by Lemma~\ref{lem:LN2t} and Proposition~\ref{prop:alphaDe1}. 
Let $(f,g,h_1,h_2,K)$ be higher homotopy data for $f$. 
        From Remark~\ref{rem:LN20}, we get    an explicit representative of $L_{N_2}(\underline{\hat\De}^f_1\x I)$ as a class of (thin) loops parametrized by 
           $\underline\De_1^f\x I$. As thin loops, that is considered inside $M_1\x PM_2\x D^k$, these are given as (after a reparametrization and for $(m,m')\notin U_1$)
           $$L_{N_2}(m,m',t)\approx \left\{\begin{array}{ll} \big(m,\delta_W[0,3t],e(m')\big) & 0\le t\le \sfrac{1}{3}\\
                                                                                \big(h_1(m,3t-1),\la_K,e(h_1(m',3t-1)\big) & \sfrac{1}{3}\le t\le \sfrac{2}{3}\\
                                \big(f(g(f(m))),f(g(\delta_W[0,3t-2]),e(g(\delta_W(3t-2))))\big) & \sfrac{2}{3}\le t\le 1 
                         \end{array}\right.$$
where $\la_K=\overline{K(m,3t-1,-)}*h_2(\delta_W,3t-1)*\overline K(m',3t-1,-)$, and where $\delta_W$ is a short path from $f(m)=\tilde f(m')$ to $f(m')$ that can be replaced by the geodesic $\delta_2$.

       To obtain a representative of $\T_f$, we are left with applying the intersection product $int_{M_1}$ associated to the evaluation map
       $\e: \underline\De_1^f\x I\rar M_1\x M_1$ evaluating the loops at $s_0$ and $s_1$. Only the times $\sfrac{1}{3}\le t\le \sfrac{2}{3}$ can yield non-trivial loops. For such time coordinates, the evaluation at $s_0$ follows $h_1(m,t)$ while at $s_1$ it follows $h_1(m',t)$. After replacing $h_1$ by a transverse deformation $\tilde h_1$, we find the formula given in the statement. 
      \end{proof}

     \begin{proof}[Proof of Theorems~\ref{thm:formula}  and \ref{thm:formula2} (part 2)]
The above proposition allows us to deduce the last part of Theorems~\ref{thm:formula}  and \ref{thm:formula2}, namely the stated formula with the class $T_f$  given in terms of $\underline\De_1^f$ and $h_1$. 
\end{proof}

         \begin{rem}[Bounded homotopies]
A direct consequence of the above result is that $T_f$ vanishes if $f$ is a homeomorphism, as expected.          
The above description more generally shows that  $T_f$  
 vanishes if the homotopy is {\em boundede} in the sense that higher homotopy $K$ is small.  
\end{rem}

    \begin{cor}\label{cor:fake}
      If the fake diagonal $\De^f_1$ has support in $U_1\subset M_1\x M_1$, then $\T_f$ vanishes. More generally, if its lift $\hat{\underline\De}_1^f$ to $\La^0_{N_2}$
       is trivial in  $H_n(\La^0_{N_2},\La^0_{N_2}|_{M_1})$,
      then the map induced by $f$ on loop spaces respects the string and higher string coproducts. 
   \end{cor}

   \begin{proof}
 The first part follows directly from the above statement since $T_f$ consists only of canonically contractible loops in that case. 
     For the second part of the statement, if $[\hat{\underline \De}_1^f]$ lives in the image of $H_n(\La^0_{N_2}|_{M_1})$ in $H_n(\La^0_{N_2})$, then for any $A\in H_*(\La_1)$, 
$\al(\del_{N_2}(A\x I))=\al[M_1]\wedge_{\del I^r}^{N_2}A$ lives in  the image of $H_n(\R_{N_2}|_{M_1})$ in $H_n(\R_{N_2})$. As $L_{N_2}$ preserves $\R_{N_2}|_{M_1}$ (by Proposition~\ref{prop:L12}(iv)),  we see that the obstruction to invariance given in Theorem~\ref{thm:Hcommute0}, namely $int_{M_1}(L_{N_2}(\al(\del_{N_2}(A\x I))\x I))$ lives in half-constant loops, and thus will not contribute. 
     \end{proof}

\begin{ex}\label{ex:elementary} Suppose that $M_2=M_1/B$ for $B\subset M_1$ a contractible subspace, and suppose that $f:M_1\to M_2$ is the collapse map. Assume that $B\subset V\subset M_1$ lies inside a contractible neighborhood containing all the points $\eps$-distant from $B$. Then 
  $\De_1^f=\De_1\cup_{\De B}(B\x B)$  and $\underline\De_1^f$ is a submanifold of $U_1\cup V\x V$.
  
  Let $F:V\x I\to V$ be a homotopy between the identity and the projection to a chosen point $b_0\in B\subset V$. Then for $(v,[n],e(v'))\in \hat{\underline\De}_1^f$, where $f(v)=n=\tilde f(v')$, the tracks of the deformation define paths $$s\mapsto (F(v,s),[f\circ F(v,s)],e\circ F(v',s))$$ inside $\La^0_{N_2}$ from any $(v,[f(v)],e(v'))\in \hat{\underline \De}_1^f$ to $(b_0,[f(b_0)],e(b_0))\in \La^0_{N_2}|_{M_1}$.  It then follows from Corollary~\ref{cor:fake} that such a map $f$ respects the string and higher string coproducts. 
  \end{ex}

     \begin{rem}[Simple homotopy equivalences]
Elementary collapses, the building bloks of simple homotopy equivalences, are of the same form as the maps considered in the previous example. Building on Corollary~\ref{cor:fake}, one can show that $\T_f$ vanishes for such maps. A strategy to show that $\T_f$ vanishes on simple homotopy equivalences more generally, as expected given the result of \cite{NaeSaf}, would be to define $T_f$ more generally for maps between cell complexes, and then understand how it behaves under inverses and compositions.
       \end{rem}

         \begin{rem}[h-cobordisms and relationship to the work of Kenigsberg-Porcelli \cite{KenPor}]\label{rem:KP}
    To a homotopy equivalence $f:M_1\to M_2$, we have here associated an inclusion $N(f,e)\inc M_2\x D^k$ of a disc bundle over $M_1$ into a (trivial) disc bundle over $M_2$.     One can show that the space $M_2\x D^k\minus int(N(f,e))$ is an h-cobordism. In \cite{KenPor}, an obstruction for the homotopy invariance of the string coproduct is constructed using such an h-cobordism. More precisely, let $F: M_2\x D^k\x I\rar M_2\x D^k$ be a retraction of $M_2\x D^k$ to $N(f,e)$.  
    The obstruction in \cite{KenPor} is defined as the family of loops $F_{[0,t]}(n,x)$ (projected to $M_2$), parametrized by the intersection $F_t(n,x)=(n,x)$ (made transverse), where $F_t:=F(-,-,t)$. 

While this obstruction is not clearly the same as ours, the data used to define it, namely the retraction $F$, is in fact equivalent to the higher homotopy data we work with in the present paper. Indeed, 
write
$$\bF:  M_2\x D^k\x I\rar M_2\x D^k$$ for the map that further retracts to $M_1$, considered as the 0-section of $N(f,e)$, and let $\bF_t=\bF(-,-,t)$, where $\bF_1$ thus has image in $M_1\subset M_2\x D^k$. 
We can use $F$ to define a homotopy inverse $g:M_2\to M_1$ for $f$ by setting $g(n)=\bF_1(n,{\bf 0})$ that is taking the composition
$$g:M_2\cong M_2\x\{{\bf 0}\}\inc M_2\x D^k\xrightarrow{\bF_1} M_1.$$ The fact that it is a homotopy inverse is witnessed by the homotopies
$$h_1:\ M_1\x I\to M_1 \ \ \ \textrm{defined by } h_1(m,t)=\bF_1(f(m),te(m))$$
and
$$h_2:\ M_2\x I\to M_2  \ \ \ \textrm{defined by } h_2(n,t)=\bF_t(n,{\bf 0})_{M_2}.$$
Finally we also have the homotopy $K:M_1\x I^2\to M_2$ between $f\circ h_1$ and $h_2\circ (f\x \id)$ given by setting 
$$K(m,s,t)=\bF_{st+(1-s) }(f(m),(1-s)te(m))$$
varying both parameters. And one can likewise use the maps $(g,h_1,h_2,K)$ to define a retraction $F$.

We expect that our obstruction $T_f$ in fact defines the same class in $H_1(\La_2,M_2)$ as that of \cite{KenPor}.  
\end{rem}


  \begin{rem}[Related work]
    (i) The lens spaces $L(1,7)$ and $L(2,7)$ used in Naef's example in \cite{Nae21} also appeared in the work of Longoni-Salvatore, who showed that their 2-points configuration spaces $F_2(L(1,7))$ and $F_2(L(2,7))$ are not homotopy equivalent \cite{LonSal}. Note that these configuration spaces are exactly the complements of the diagonal in $M\x M$ for $M$ the one or the other manifold, and the difference between the diagonal and the fake diagonal is exactly what obstructs the map $f\x f$ to induce a map on configuration spaces from $\textrm{Conf}_2(M_1)$ to $\textrm{Conf}_2(M_2)$.
    The complement of  $\De_1^f$ appears in \cite{Rap11} in the context of studying the question of homotopy invariance of configuration spaces.
It is a natural question to ask whether a condition for a homotopy equivalence $f$ to respect the coproduct can be formulated in terms of the configuration spaces of two points in the manifolds. See also \cite{NaeWil} for further relationship between the string topology product and coproduct, and the configuration space of two points in the manifold. \\
  (ii)   A naive version of the assumption of Corollary~\ref{cor:fake} would be that the inclusion $\De_1\to \De_1^f$ is a homotopy equivalence, which  is equivalent to the condition that $\De_1^f\to \De_2$ is a homotopy equivalence, that is requiring that the product homotopy equivalence $f\x f: M_1\x M_1\to M_2\x M_2$ restricts to a homotopy equivalence on the inverse image of $\De_2$. As pointed out to us by Wolfgang L\"{u}ck, this would be a codimension $n$ analogue of Cappell's splitting theorem, that gives conditions involving the Whitehead torsion for this to hold when considering instead a codimension 1 submanifold of the target \cite{Cap76}.  \\
 (iii) The coproduct is part of a larger (expected) family of string topology operations parametrized by (the harmonic compactification of) the  moduli space of Riemann surfaces, see e.g.~\cite[Sec 6.6]{WahWes08} or \cite{DCP,HinWah2}, where non-homotopy invariance can only come from the compactification by \cite{Bia25}. Obstructions related to higher diagonals are to be expected for operations such as the families \cite[Sec 4]{Wah16}. The obstructions would relate to configuration spaces of any number of points in the manifold. 
  \end{rem}

\appendix

\section{Composing tubular embeddings}\label{app:tubcomp}

Given a composition of embeddings $M_1\ \sta{f_1}\inc\ M_2\
\sta{f_2}\inc\  M_3$ and choices of tubular neighborhoods for each
embedding, we will construct here an associated 
tubular neighborhood for the composed map and study its properties.

\begin{prop}\label{prop:comptub}
Consider 
\begin{equation}\label{equ:comp}
\xymatrix{N_1\ar[d]_{p_1}\ar@{^(->}[drr]^{\nu_1} &&  N_2\ar[d]_{p_2}\ar@{^(->}[drr]^{\nu_2} && \\
M_1 \ar@{^(->}[rr]^{f_1}&& M_2 \ar@{^(->}[rr]^{f_2}&& M_3.}
\end{equation}
with $f_1,f_2$ codimension $k$ and $l$ embeddings respectively, $N_1,N_2$ the associated normal bundles and $\nu_1,\nu_2$ chosen tubular neighborhoods. 
Then 
\begin{enumerate} 
\item The bundle $N_{12}:= N_{1}\oplus f_{1}^{\ast }N_{2}\rar M_{1}$ is isomorphic to the normal bundle of $f_2\circ f_1$. 
\item There is an isomorphism 
  $$\xymatrix{N_{12}= N_{1}\oplus f_{1}^{\ast }N_{2}  
       \ar@{-->}[rr]^\cong \ar[dr]^p & & \nu_1^*N_2 \cong
    N_2|_{\nu_1\! N_1}\ar@{^(->}[r] \ar[dl]& N_2\ar[d]^{p_2} \ar@{^(->}[dr]^{\nu_2}& \\
    M_1 \ar@{<-}[u]+<-5ex,-2ex>\ar@<-1ex>@/_/[rrr]_-{f_1}&N_1 \ar[l]_-{p_1}\ar@{^(->}[rr]^-{\nu_1}& & M_2\ar[r]_-{f_2} & M_3
  }$$
  as bundles over $N_1$, where $p$ is the natural projection,  with the property that the resulting composition 
$$\nu_{12}:=\nu_2\circ\hat\nu_1:N_{12}\xrightarrow{\hat \nu_1} N_2\xrightarrow{\nu_2}  M_3$$ is a tubular neighborhood for $ f_2\circ f_1:M_1\inc M_3$. 
Explicitly,  $\nu _{12}(m,V,W)=\nu_2(\nu _{1}(m,V),W^{\prime })$ for $W^{\prime }\in (N_{2})_{\nu _{1}(m,V)}$ the parallel transport of $W$ along the path $\nu _{1}(m,tV)$.
\item Let $\tau_1\in C^k(M_2,\nu_1(N_1)^c)$ and $\tau_2\in C^l(M_3,\nu_2(N_2)^c)$ be Thom classes for $(N_1,\nu_1)$ and $(N_2,\nu_2)$.  Let $p_2^*\tau_1$ be the image of $\tau_1$ along the maps
$$C^k(M_2,\nu_1(N_1)^c)\sta{p_2^*}\rar C^k(N_2,\hat\nu_{1}(N_{12})^c) \stackrel{\sim}\rar C^k(M_3,\nu_{12}(N_{12})^c),$$
where the second map is the quasi-isomorphism induced by $\nu_2$ and excision. 
 Then $$\tau_{12}:=p_2^*\tau_1\cup\tau_2\in C^{k+l}(M_3,\nu_{12}(N_{12})^c)$$ is a Thom class for the tubular embedding $\nu_{12}$. Moreover, the diagram 
$$\xymatrix{C_*(M_3)\ar[rr]^-{\llbracket\tau_2\cap\rrbracket} \ar[drr]_{\llbracket\tau_{12}\cap\rrbracket} && C_{*-l}(M_2)\ar[d]^{\llbracket\tau_1\cap\rrbracket}  \\
& & C_{*-k-l}(M_1) }$$
commutes up to homotopy, where 
\begin{align*}
\llbracket\tau_1\cap\rrbracket\colon & C_*(M_2)\to C_*(M_2,\nu_1(N_1)^c) \xrightarrow{[\tau_1\cap]} C_{*-k}(N_1)\arsim C_{*-k}(M_1)\\
\llbracket\tau_2\cap\rrbracket\colon & C_*(M_3)\to C_*(M_3,\nu_2(N_2)^c) \xrightarrow{[\tau_2\cap]} C_{*-l}(N_2)\arsim C_{*-k}(M_2)\\
\llbracket\tau_{12}\cap\rrbracket\colon & C_*(M_3)\to C_*(M_3,\nu_{12}(N_{12})^c) \xrightarrow{[\tau_{12}\cap]} C_{*-k-l}(N_{12})\arsim C_{*-k-l}(M_1)
\end{align*}
are the maps induced by capping with the respective Thom classes after  taking small simplices  (see \cite[A.2]{HinWah0} for details about the map $[\tau \cap]$), and retracting to the zero section at the end.
\end{enumerate}
\end{prop}

To prove the proposition, we will, as in the rest of the paper, pick Riemannian metrics. This allows us for example to consider the normal bundle of an embedding as a subbundle of the tangent bundle of the target, as we explain now: 
Let 
\begin{equation*}
f:M\hookrightarrow M^{\prime }
\end{equation*}%
be an embedding of manifolds, with $M$ compact and $M^{\prime }$ a 
Riemannian manifold. Let $p:N\rightarrow M$ be the normal bundle of $f$. 
The metric of $M'$ gives an identification of $N$ with a subbundle of $TM^{\prime}|_{M}:$ 
\begin{equation}\label{equ:splitb}
N\equiv TM^{\bot }=:\{(m,W)\ |\ m\in M\text{, }W\in T_{f(m)}M^{\prime }\text{
and }W\bot f_{\ast }T_{m}M\}
\end{equation}
and thus splits the exact sequence 
\begin{equation*}
0\rightarrow TM\rightarrow TM^{\prime }\rightarrow N\rightarrow 0
\end{equation*}%
of bundles over $M$.  We will identify $N$ with this subbundle.  

\medskip

The following lemma is an elementary consequence of the inverse function
theorem. 

\begin{lem}Let $N\to M$ be the normal bundle of the embedding $f: M\inc M'$ as above.  Suppose that  
\begin{equation*}
\nu :N\longrightarrow M^{\prime }
\end{equation*}%
has the following properties:
\begin{enumerate}[(i)]
\item $\nu (m,0)=f(m)$\textit{\ for all }$m\in M$\textit{.}

\item  The map 
\begin{equation*}
d\nu:T_{(m,0)}N\rightarrow T_{f(m)}M^{\prime }
\end{equation*}%
is surjective. 
\end{enumerate}
Then there exists $\varepsilon >0$ so that
the restriction of $\nu $ to $N_{\varepsilon }$ is a
tubular embedding of $M$ in $M^{\prime }$, for  $N_{\varepsilon }=\{(m,W)\in N\ |\ |W|<\varepsilon \}$. 
\end{lem}

\begin{proof}[Proof of Proposition~\ref{prop:comptub}]
Choose a Riemannian metric on $M_{3}$.  Let $M_{1}$ and $M_{2}$ carry the
induced metrics. 
Statement (1) follows from (suppressing the inclusion maps $f_1,f_2$):
\begin{equation*}
N\equiv TM_{1}^{\bot }\subset TM_{3}\text{; \ }N_{1}\equiv TM_{1}^{\bot }\subset TM_{2}%
\text{; }N_{2}\equiv TM_{2}^{\bot }\subset TM_{3}, 
\end{equation*}%
so that for each $m\in M_{1},$ 
\begin{equation*}
N_{m}\equiv (N_{1})_{m}\oplus (N_{2})_{f(m)}.
\end{equation*}%
As a consequence we can and will identify%
\begin{equation*}
N\cong N_{12}\equiv \left\{ (m,V,W) \ \Big|\ 
\begin{array}{l}
m\in M_{1}\text{,} \\ 
V\in T_{f_{1}(m)}M_{2}\text{, }V\bot \, df_{1}T_{m}M_{1}\text{, \ and } \\ 
W\in T_{f(m)}M_{3}\text{, }W\bot \, df_{2}T_{f_{1}(m)}M_{2}%
\end{array}%
\right\}.
\end{equation*}%

\medskip

To prove (2), define $\hat\nu _{1}:N_{12}\rightarrow N_{2}$ by 
\begin{equation*}
\hat\nu _{1}(m,V,W)=(\nu _{1}(m,V),W^{\prime })
\end{equation*}%
where $W^{\prime }\in (N_{2})_{\nu _{1}(m,V)}$ is the vector obtained
from $W$ by parallel transport along the path $\nu _{1}(m,tV)$. The inverse is given by parallel transporting back, which gives the claimed isomorphism by uniqueness of such transports. 

We then apply the lemma in the case $M=M_{1}$, $M^{\prime }=M_{3}$, 
$f=f_{2}\circ f_{1}$, and $\nu =\nu _{12}$. 
We need to check condition (ii) of the lemma.  Fix $m\in M_{1}$.  Because $\nu_{1}:N_{1}\rightarrow M_{2}$ is a tubular embedding, the map $d\nu_{1}:T_{(m,0)}N_{1}\rightarrow T_{m}M_{2}$ is surjective,  and the
tangent space to $M_{2}$ in $M_{3}$ is in the image of $d\nu_{12}=:d(\nu_{2}\circ \hat\nu _{1})$.  Now let $W\in N_{2}$; that is, $W\in T_{m}M_{3}$, with $W\bot T_{m}M_{2}$.  We find by definition $\hat\nu _{1}(m,0,W)=(m,W)\in N_{2}$ and $\nu_{12} (m,0,W)=\nu _{2}\circ \hat\nu_{1}(m,0,W)=(m,W)$.  (When $V=0$, the parallel transport is trivial.) 
Since $T_{m}M_{3}$ is spanned by $T_{m}M_{2}$ and $d\nu _{2}N_{2}$, we
conclude that 
\begin{equation*}
\nu _{12*}:T_{(m,0)}N_{12}\rightarrow T_{m}M_{3}
\end{equation*}%
is surjective.

\medskip

To show the commutativity of the diagram in statement (3), we decompose the diagram as 
\begin{equation}\label{equ:triangle}
\xymatrix{C_*(M_3) \ar[r]\ar[dr] & C_*(M_3,N_2^c)  \ar[r]^-{[\tau_2\cap]} & C_{*-k}(N_2)\ar[d]\ar[r]^-{p_2} & C_{*-k}(M_2) \ar[d]\\
& C_*(M_3,N_{12}^c) \ar[dr]_{[\tau_{12}\cap]} & C_{*-k}(N_2,N_{12}^c) \ar[d]^{[p_2^*\tau_1\cap]} \ar[r]^{p_2} & C_{*-k}(M_2,N_1^c)  \ar[d]^{[\tau_1\cap]}  \\
&& C_{*-k-l}(N_{12})   \ar[r]^{p_2} \ar[dr]^{p_{12}} & C_{*-k-l}(N_1) \ar[d]^{p_1} \\
&&& C_{*-k-l}(M_1)
}\end{equation}
where we have suppressed the embeddings $\nu_i$, $\hat \nu_i$ for readability. 
Commutativity  follows
from the naturality of the cap product $[\tau\cap]$ 
and the compatibility of the cup and cap products.

We are left to check that $\tau_{12}$ is the Thom class of the composition, which follows from its definition as
$$\tau_{12}\cap [M_3]=p_2^*\tau_1\cap (\tau_2\cap [M_3])=p_2^*\tau_1\cap (f_2)_*[M_2]=(f_2)_*(f_2^*p_2^*\tau_1\cap [M_2])=(f_2)_*(\tau_1\cap [M_2])=(f_2)_*(f_1)_*[M_1],$$
in homology, where we used that $p_2\circ f_2: M_2\to N_2\to M_2$ is the  identity. 
\end{proof}

\bibliographystyle{plain}
\bibliography{biblio}

\end{document}